\documentclass[11pt,reqno,a4paper]{amsart}

\usepackage[margin=3cm]{geometry}

\usepackage{graphicx}
\usepackage{amsmath,amsthm}
\usepackage{amsfonts}

\usepackage{float}
\usepackage[all]{xy}
\usepackage{amssymb}
\usepackage{graphicx}	
\usepackage[utf8]{inputenc}
\usepackage{mathtools}
\usepackage{pgf,tikz,pgfplots}
\usepackage{bbm}
\usepackage{enumitem}
\usepackage{soul}
\pgfplotsset{compat=1.17}

\usepackage[hidelinks,colorlinks=true,citecolor={blue}, linkcolor={blue}]{hyperref}

\usepackage[T1]{fontenc}
\usepackage[utf8]{inputenc}
\usepackage[english]{babel}
\usepackage{pdfpages}

\usepackage{libertine}
\usepackage{libertinust1math}
\usepackage[export]{adjustbox}
\usepackage{verbatim}
\usepackage{accents}
\usepackage{BOONDOX-cal}

\newtheorem{thm}{Theorem}[section]
\newtheorem{thmx}{Theorem}

\newtheorem{claim}{Claim}

\newtheorem{prop}{Proposition}[section]
\newtheorem{lema}{Lemma}[section]
\newtheorem{corol}{Corollary}[section]
\newtheorem{defn}{Definition}[section]
\newtheorem{rmk}{Remark}[section]

\usepackage{hyperref}

\usepackage{mathrsfs} 

\makeatletter
\newcommand{\eqnum}{\refstepcounter{equation}\textup{\tagform@{\theequation}}}
\makeatother
\makeatletter
\def\namedlabel#1#2{\begingroup
  #2%
  \def\@currentlabel{#2}%
  \phantomsection\label{#1}\endgroup
}
\makeatother

\newcommand{\diam}{\text{diam} \ }

\newcommand{\rr}{\mathbb{R}}

\newcommand{\zz}{\mathbb{Z}}
\newcommand{\nn}{\mathbb{N}}
\newcommand{\TT}{\mathbb{T}}
\newcommand{\PP}{\mathbb{P}}

\newcommand{\ol}{\overline}

\newcommand{\Leb}{\text{Leb}}
\newcommand{\End}{\text{End}}

\title{Inverse SRB measures for endomorphisms on surfaces}
\author{Victor Janeiro, Radu Saghin}

\begin{document}

\begin{abstract}
    We extend D. Burguet's construction of SRB measures \cite{burguet2024srb} for the non-invertible scenario obtaining hyperbolic invariant measures with absolutely continuous disintegrations on stable manifolds for a certain class of endomorphisms on the two-torus. The constructed measures also maximize the folding entropy, and we call them equidistributed inverse SRB measures. For dissipative perturbations of the examples given by M. Andersson, P. Carrasco and R. Saghin in \cite{martin-a} we obtain existence, uniqueness and continuity with respect to the map of the equidistributed inverse SRB measures as well as the SPR property of the stable geometric potential for such maps. 
\end{abstract}

\maketitle

\tableofcontents

\section{Introduction and results}

\subsection{Introduction}

Among the most important tools in the study of Dynamical Systems are the so-called SRB measures, and this is why they are extensively studied in the literature.

Let $f:M\rightarrow M$ be a $C^{1+\alpha}$ diffeomorphism on a compact Riemannian manifold $M$. A hyperbolic invariant measure $\mu$ is called {\it SRB} if it satisfies one of the following equivalent conditions:
\begin{itemize}
    \item The disintegrations of $\mu$ along the unstable manifolds are absolutely continuous with respect to the normalized Lebesgue measure;
    \item The measure $\mu$ satisfies the Pesin formula:
    \begin{equation}
        h_{\mu}(f)=\int_M\sum_{\lambda_i>0}\lambda_id\mu,    \end{equation}
    where $\lambda_i$ are the Lyapunov exponents, which are well-defined at $\mu$-almost every point.
\end{itemize}
For precise definitions, see Section \ref{sect:SRB}. The {\it basin of the invariant measure $\mu$} is 
\begin{equation}
    \mathcal B(\mu)=\{x\in M:\ \lim_{n\rightarrow\infty}\frac 1n\sum_{k=0}^{n-1}\delta_{f^k(x)}=\mu\}.
\end{equation}
In other words, the measure $\mu$ statistically describes the future orbit of the points from its basin. If the invariant hyperbolic measure $\mu$ is SRB and ergodic, then it satisfies the following important property: {\it The basin of $\mu$ has positive Lebesgue measure}. This means that the measure $\mu$ is {\it physical}, it describes the statistics of a (physically) relevant set of points.

One can define the {\it inverse SRB measures} as SRB measures for $f^{-1}$. Equivalently, they have absolutely continuous disintegrations along stable manifolds and satisfy the inverse Pesin formula:
\begin{equation}
    h_{\mu}(f)=-\int_M\sum_{\lambda_i<0}\lambda_id\mu.
\end{equation}
Again if an inverse SRB measure is also ergodic, then it actually describes the statistics of the backward orbits of a set of points of positive Lebesgue measure.

Now suppose that $f:M\rightarrow M$ is not invertible, it is just a $C^{1+\alpha}$ endomorphism. In order to avoid critical points, an endomorphism throughout this paper is always a local diffeomorphism. The forward SRB measures can still be defined as the hyperbolic measures which satisfy the forward Pesin formula. The unstable manifolds are not well defined on $M$ because they depend on the backward trajectory. However, they are well defined for the natural extension $\hat f$ of $f$ to the space of orbits $L_f$ (see Section \ref{sect:SRB} for more details). 

The map $\hat f$ is invertible and semi-conjugated to $f$ by the natural projection $\pi_{f}:L_f\rightarrow M$. Every invariant measure $\mu$ for $f$ has a unique lift $\hat\mu$ on $L_f$ invariant under $\hat f$. Then $\mu$ is a forward SRB measure if and only if the disintegrations of $\hat\mu$ along the unstable manifolds are absolutely continuous; see e.g. \cite{ledrappier1985metric}. Again, if $\mu$ is SRB and ergodic, then its basin has positive Lebesgue measure.

For endomorphisms we lose the symmetry between forward and inverse SRB measures, because the backward trajectory is not uniquely defined. We can still define the inverse SRB measures using the absolute continuity property.
\begin{defn}\label{def.inverseSRB}
    The hyperbolic invariant measure $\mu$ of the local $C^{1+\alpha}$ diffeomorphism $f$ is an {\it inverse SRB measure} if the disintegrations of $\mu$ along the stable manifolds are absolutely continuous.     
\end{defn}

If $\mathcal{E}$ is the partition of $M$ into points, the {\it folding entropy of $f$ with respect to $\mu$} is
\begin{equation}
    F_{\mu}(f)=H_{\mu}(\mathcal{E}\mid f^{-1}\mathcal{E}).
\end{equation}
Liu \cite{liu2003ruelle} proved that for every $\mathcal{C}^{1+}$ endomorphism
\begin{equation}\label{eq Ruelle ineq endo}
    h_\mu(f) \le F_\mu(f)-\int_M\sum_{\lambda_i<0}\lambda_i \ d\mu.
\end{equation}
If the measure $\mu$ is inverse SRB for the local diffeomorphism $f$, then it satisfies the {\it inverse Pesin formula for endomorphisms:}
\begin{equation}
    h_{\mu}(f)=F_{\mu}(f)-\int_M\sum_{\lambda_i<0}\lambda_id\mu.
\end{equation}
It is not clear if the converse holds in general, but we will describe a situation when it does. The disintegrations of $\mu$ on the elements of the partition $f^{-1}\mathcal E$ give rise to the {\it Jacobian of $\mu$}, $J_{\mu}:M\rightarrow[0,\infty)$, defined $\mu$-almost everywhere, see again Section \ref{sect:SRB} for more details. If the hyperbolic invariant measure $\mu$ satisfies the inverse Pesin formula for the local diffeomorphism $f$, and the Jacobian $J_{\mu}$ is H\"older continuous, then $\mu$ is an inverse SRB measure \cite{liu2008invariant}.

The basin of an inverse SRB measure $\mu$ can be defined using the natural extension of $f$, more precisely, one can consider the basin of $\hat\mu$ for $\hat f^{-1}$. The basin lies in the inverse limit space $L_f$; it can be projected to the manifold $M$, however we may have many inverse SRB measures such that the projections of their basins overlap in $M$, even for Anosov endomorphisms. This is because a point $x\in M$ may have different backward trajectories converging to different inverse SRB measures. This is related to the fact that different inverse $SRB$ measures $\mu$ lift to different measures $\hat\mu$ on $L_f$, and these in turn may have different disintegrations $\hat\mu_x$ along the fibers $\pi^{-1}(x)$ (the set of backward orbits from $x$).

A natural measure to consider on the fibers $\pi^{-1}(x)$ is the Bernoulli measure $p_x$, i.e. the measure which gives equal weight $\frac 1d$ to each preimage of every point ($d$ being the degree of $f$). A special type of measures is the following.
\begin{defn}\label{def.equidistributed}
    An invariant measure $\mu$ of the $C^{1+\alpha}$ local diffeomorphism $f$ of degree $d$ is equidistributed if one of the following 3 equivalent conditions holds: \begin{itemize}
        \item The disintegrations of $\hat\mu$ on $\pi^{-1}(x)$ are $p_x$ for $\mu$-almost every $x$;
        \item $J_{\mu}(x)=\frac 1d$ for $\mu$-almost every $x$;
        \item $F_{\mu}(f)=\log d$ (or $F_{\mu}(f)$ is maximal).
    \end{itemize}
\end{defn}

In this paper we are interested in the existence and uniqueness of equidistributed inverse SRB measures for endomorphisms on surfaces.

The SRB measures were extensively studied during the last years, especially in the case of diffeomorphisms. It is known that they exist for uniformly hyperbolic attractors, and also in lots of specific partially hyperbolic or non-uniformly hyperbolic settings. There are also important advances in the case of endomorphisms, like the work on the existence of forward SRB measures for partially hyperbolic endomorphisms on the torus or the annulus by Tsujii \cite{Tsujii2005-if}. There are also works on a similar notion of inverse SRB measure for hyperbolic repellers by Mihailescu and  Urbansky \cite{mihailescu2010physical}, \cite{mihailescu2013entropy}. We are not aware of any works on the existence of SRB measures for endomorphisms without dominated splittings.

Recently, Burguet \cite{burguet2024srb} and Buzzi-Crovisier-Sarig \cite{buzzi2023existence} independently obtained the following remarkable result:

\begin{thm}[Burguet, Buzzi-Crovisier-Sarig] Suppose that $f:M\rightarrow M$ is a $C^{\infty}$ surface diffeomorphism, and consider the set
$$A_0=\{x\in M:\ \limsup_{n\rightarrow\infty}\frac 1n\log\|Df^n(x)\|>0\}.$$
If $A_0$ has positive Lebesgue measure, then $f$ has $SRB$ measures.
\end{thm}

This important result proves Viana's Conjecture in dimension 2 for $\mathcal{C}^\infty$ diffeomorphisms, however it is not clear how to obtain concrete and robust examples satisfying the hypothesis outside the class of hyperbolic or partially hyperbolic diffeomorphisms.

The goal of our paper is to provide concrete examples of robust families of endomorphisms where we can apply the techniques developed in \cite{burguet2024srb} and \cite{buzzi2023existence}. We are able to obtain existence and uniqueness of equidistributed inverse SRB measures for an open class of surface endomorphisms without a dominated splitting. Let us explain which are these endomorphisms.

Let $\text{End}^r(\TT^2)$ be the space of $C^r$ local diffeomorphisms of $\mathbb T^2$. For any $f\in \text{End}^r(\TT^2)$ of degree $d\geq 2$, we define for $(x,v) \in T^1\TT^2$ and for $n \in \nn$
\begin{equation}\label{eq defn Ixvfn}
I(x,v;f^n):= \frac{1}{d^n}\sum\limits_{y \in f^{-n}(x)}\log \|(Df^n_y)^{-1}\cdot v\|,
\end{equation}
and
\begin{equation}\label{eq defn Ix}
I_f(x):= \limsup\limits_n \frac{1}{n} \inf_{v \in T^1_x\TT^2} I(x,v;f^n).
\end{equation}
When $\inf_xI_f(x)>0$, we say that $f$ is inverse expanding on average.

We consider the sets
\begin{equation}\label{eq f0}
    \mathcal F_0=\{f\in\text{End}^{1+}(\TT^2): \inf_x I_f(x)>0,  \ \text{and} \det(Df_x)=\deg(f) \}
\end{equation}
and
\begin{equation}\label{eq f1}
    \mathcal F_1=\{f\in\mathcal F_0:\ \pm 1 \text{ is not an eigenvalue of the linear part of } f\}
\end{equation}

Through a series of papers (\cite{martin-a}, \cite{janeiro2023existence}, \cite{ramirez2023non}, \cite{lima2024measuresmaximalentropynonuniformly},\cite{andersson2016transitivity}) we know the following.

\begin{thm}\label{thm previous results for F_1}
    Every $f\in\mathcal F_0$ is conservative, and ($C^1$ robustly) non-uniformly hyperbolic in the sense that Lebesgue almost every point has one positive and one negative Lyapunov exponent. Every $f\in\mathcal F_1$ is transitive and stably ergodic.

    $\mathcal F_0$ (and $\mathcal F_1$) contains endomorphisms without dominated splitting and intersects almost every homotopy class of endomorphisms on $\TT^2$.
\end{thm}

We are interested in dissipative perturbations of maps from $\mathcal F_0$ or $\mathcal F_1$. We will show not only that they have unique equidistributed inverse SRB measures, but also that these SRB measures are very well behaved from a statistical point of view.

Let us quickly recall the notion of SPR property from \cite{buzzi2025strong}. We start again with the diffeomorphism case and assume that $f$ is a $C^{1+\alpha}$ diffeo on a compact Riemannian manifold $M$.

Given $\lambda,\epsilon>0$, we say that the nonempty $\Lambda\subset M$ is a {\it $(\lambda,\epsilon)$-Pesin block for $f$} if there exists a measurable invariant decomposition $TM=E^s\oplus E^u$ on $\cup_{n\in\mathbb Z}f^n(\Lambda)$ and $K>0$ such that for any $n\in\mathbb Z$, $k\geq 0$, and $y\in\Lambda$, we have
\begin{equation}
    \max(\|Df^k|_{E^s(f^n(y))}\|,\|Df^{-k}|_{E^u(f^n(y))}\|)\leq Ke^{-\lambda k+\epsilon|n|}.
\end{equation}

Assume that $X\subset M$ is a Borel invariant set and $\phi:X\rightarrow\mathbb R\cup\{-\infty\}$ is a quasi-Hölder potential with the top pressure $P_{TOP}(f|_X,\phi)$ (see sub-section \ref{sect spr} for more details). In particular $\phi$ can be the stable or the unstable geometric potential:
\begin{eqnarray}
    \phi_s(x)&=\begin{cases}
        \log|\det Df|_{E^s(x)}|,\ \hbox{ if }\ E^s(x) \ \hbox { exists,}\\
        -\infty,\ \hbox{ otherwise}
    \end{cases};\\
     \phi_u(x)&=\begin{cases}
        -\log|\det Df|_{E^u(x)}|,\ \hbox{ if }\ E^u(x) \ \hbox { exists,}\\
        -\infty,\ \hbox{ otherwise}
    \end{cases}.
\end{eqnarray}

We say that {\it $f$ is SPR for the potential $\phi$ on the Borel set $X$} if there exists $\lambda>0$ such that for every $\epsilon>0$, there are a Borel $(\lambda,\epsilon)$-Pesin block $\Lambda$ and numbers $P_0<P_{TOP}(f|_x,\phi),\ \tau>0$ such that for any ergodic invariant measure $\nu$ on $X$, $P(f,\nu,\phi)>P_0$ implies that $\nu(\Lambda)>\tau$.

In \cite{buzzi2025strong} the authors show that the SPR property implies many good statistical properties of the equilibrium states: they are Bernoulli times a permutation, the quasi-Hölder observables satisfy the large deviation property, almost sure invariance principle, asymptotic variance of Birkhoff sums, exponential decay of correlations (in the mixing case).

The definition of the SPR property can be adapted to the noninvertible case. Assume now that $f$ is a local diffeomorphism on $M$ with the invertible lift $\hat f$ to $L_f$. The $(\lambda,\epsilon)$-Pesin blocks can be defined on $L_f$ using $\hat f$, and for any quasi-Hölder potential $\phi$ on $L_f$ we can define the SPR property using the lift $\hat f$. The SPR property will again imply the same good statistical properties for $\hat f$, for more details see sub-sections \ref{sect Pesin}, \ref{sect spr}.

\subsection{Main result}

Now we are ready to state our main theorem. Recall that $\mathcal F_1$ is the $C^1$ open set of endomorphisms of $\mathbb T^2$ defined in \eqref{eq f0}, \eqref{eq f1}.

\begin{thmx}\label{thm main}
There exists a $C^2$ open neighborhood $\mathcal U \subset \End^{1+}(\TT^2)$ of $\mathcal F_1$ such that:
\begin{enumerate}
\item Every $C^{\infty}$ endomorphism $f\in\mathcal U$ has a unique equidistributed inverse SRB measure $\mu_f$
\item $\mu_f$ has full basin, meaning that for Lebesgue almost every point $x\in\mathbb T^2$, for Bernoulli ($p_x$) almost every backward trajectory starting at $x$, the Birkhoff averages of the Dirac measure along the backward orbit converge weakly to $\mu_f$. 
\item Every $f \in \mathcal{U}$ (not necessarily $\mathcal{C}^\infty$) admitting an equidistributed inverse SRB measure is SPR for the stable geometric potential. In particular, $\mu_f$ satisfies many ergodic properties (see \S~\ref{sect spr}). \item The maps $f\mapsto \chi^{\pm}(\mu_f)$ and $f\mapsto h_{\mu_f}(f)$ are continuous ($C^1$ topology in $\mathcal U$).
\item The map $\End^{\infty}(\TT^2)\cap \mathcal{U} \ni f \mapsto \mu_f$ is continuous (in the $\mathcal{C}^1$-topology in $\mathcal{U}$ restricted to $C^2$ bounded sets and the weak$^*$-topology in the space of probabilities).
\end{enumerate}
\end{thmx}

\begin{rmk}
    There is an asymmetry between the future and the past in our examples. Even if the future is deterministic, and the past is in some sense "random", from the statistical point of view, we know much more about the past than about the future. We know where most of the past orbits go, but these methods do not tell us anything about the future orbits. It is exactly the "randomness" that helps us to understand the past better than the future.
\end{rmk}

\subsection{Existence, uniqueness and SPR property of equidistributed inverse SRB measures}

Theorem \ref{thm main} is a consequence of several results which present below. We can state the results under more general hypothesis, because they may have their own separate interest.

For a $C^1$ endomorphism of $\mathbb T^2$ we can define
\[
R^-(f):= \lim\limits_{n \to +\infty}\frac{1}{n} \log^+\sup_{x \in \TT^2}\|(Df^n_x)^{-1}\|.
\]

Following the work of Burguet and Buzzi-Crovisier-Sarig, we can in fact prove the following result which gives the existence of equidistributed inverse SRB measures for endomorphisms with finite $C^r$-regularity.

\begin{thmx}[Existence of equidistributed inverse SRBs]\label{main thm burguet}
    Let $f \in \text{End}^r(\TT^2)$, $r \ge 2$. If $I_f(x)>\frac{R^-(f)}{r}$ on a set of positive Lebesgue measure, then $f$ admits at least one and at most countably many equidistributed inverse SRB measures $\{\mu_i\}_{i \in I}$.
\end{thmx}

Building on the method of Burguet we can also provide more general criteria for the uniqueness of the equidistributed inverse SRB measures, using only the $C^2$ regularity of the endomorphism $f$.

\begin{thmx}[Uniqueness of equidistributed inverse SRBs]\label{main thm uniqueness}
    Given $f \in \mathcal{F}_1$, there exists $\mathcal{U}\subset \End^{1+}(\TT^2)$ a $\mathcal{C}^2$-neighborhood of $f$ such that every $g \in \mathcal{U}$ has at most one equidistributed inverse SRB measure $\mu$ with full basin. 
\end{thmx}

We can also adapt to endomorphisms the criteria of Buzzi-Crovisier-Sarig and obtain the SPR property of equidistributed inverse SRBs (assuming that they exist) for some large class of $C^{1+}$ endomorphisms on $\mathbb T^2$.

\begin{thmx}[SPR property of equidistributed inverse SRBs]\label{main thm SPR}
    Let $f \in End^1(\TT^2)$ be such that $\inf_xI_f(x)>0$ and $\inf_x |\det Df_x|>1$. If $f$ admits an equidistributed inverse SRB measure, then $f$ is SPR for the stable geometric potential $\phi_s = \log |\det Df|_{E^s}|$.
\end{thmx}

Finally, we provide specific examples of coexistence of inverse SRB measures with different folding entropies, even when the equidistributed one is unique.

\begin{prop}\label{cor}
    There are examples of transitive maps $g \in End^{1+}(\TT^2)$ with, at least, two inverse SRB measures. These measures have different folding entropy and their unique lifts to the natural extension disintegrate along its fibres in different ways. The basin of each inverse SRB projects to the whole torus.
\end{prop}

\subsection{Ideas of the proofs}
The proof of Theorem \ref{main thm burguet} is based on the tools developed by Burguet and Buzzi-Crovisier-Sarig in order to solve the Viana conjecture for surfaces. We adapt these tools to inverse preimages of curves, and we are able to prove that the hypothesis on the backward expansion implies the existence of invariant measures which are equidistributed and satisfy the inverse Pesin formula for endomorphisms, i.e. equidistributed inverse SRB measures.

In order to prove the uniqueness of equidistributed inverse SRBs from Theorem \ref{main thm uniqueness} we use two facts: the stable and unstable manifolds are large for many points, and the map is close to being transitive. These two facts allow us to conclude that two ergodic equidistributed inverse SRB measures must be homoclinically related, and standard results on homoclinic classes of measures will imply the uniqueness.

The lower bounds on the invariant manifolds follow from the inverse expanding on average property. Also the maps in $\mathcal F_1$ have a strong form of transitivity: the pre-images of any curve are dense in the torus. We can extend this property to a neighborhood of $\mathcal F_1$ using a compact family of curves and $\epsilon$-transitivity, and this is enough to conclude the proof of uniqueness.

The proof of the full basin property is again an adaptation of the work of Burguet to the endomorphism case.

The proof of Theorem \ref{main thm SPR} is based on results from \cite{buzzi2025strong}, which relate the SPR property to some form of continuity of the Lyapunov exponents for measures of high pressure. In order to prove the continuity of the Lyapunov exponents, we use the standard method of passing to the projectivization of the derivative cocycle, and make use of the equidistribution of the measure and the inverse expanding on average property.

Putting together Theorems \ref{main thm burguet}, \ref{main thm uniqueness} and \ref{main thm SPR} gives us Theorem \ref{thm main}. The continuity of the Lyapunov exponents follows from the same methods used in the proof of Theorem \ref{main thm SPR}, and the continuity of the metric entropy is a consequence of the inverse Pesin formula. We also note that continuity of the exponents implies upper-semicontinuity of the metric entropy at equidistributed inverse SRB measures, and this together with the inverse Pesin formula imply the continuity of the inverse equidistributed SRB measure with respect to the map.

\subsection{Organization of the paper}

In Section \ref{sect:SRB} we present the definitions and basic results related to endomorphisms on surfaces, invariant measures, Pesin Theory, inverse SRB measures and the SPR property.

In Section \ref{sect:ieap} we discuss the Inverse expanding on average property for endomorphisms on surfaces: we give equivalent definitions and we present several important properties: the unstable directions at every point are spread in different directions, unstable manifolds exist at every point and they are in general large, there exist large stable manifolds. We also have uniform transitivity properties in a neighborhood of conservative endomorphism with the inverse expanding on average property.

Section \ref{sect:tools} is dedicated to the presentation of the main tools used by Burguet in his work on SRB measures and how to translate it to the non-invertible scenario: F{\o}lner sequences and empirical measures associated to them, and some notions of the Yomdin Theory as the Reparametrization Lemma and geometric times.

In Section \ref{sect:B} we prove Theorem \ref{main thm burguet} and in Section \ref{sect:A} we prove Theorem \ref{main thm uniqueness}. In Section \ref{sect:C} we prove the continuity of the exponents, and as a consequence we obtain Theorem \ref{main thm SPR}. Theorem \ref{thm main} is proved in Section \ref{sect:main} and in the last section we prove Proposition \ref{cor}.

\section{Inverse SRB measures}\label{sect:SRB}

In this section, we will present the definitions and the basic properties of hyperbolic and inverse SRB measures for endomorphisms. Most of the definitions and results hold in a greater generality, however for simplicity we will restrict ourselves to the case of the two-torus whenever it is more convenient to do so.

\subsection{The natural extension $L_f$}\label{subsec Solenoid}

Let $f: M \to M$ be an endomorphism (local diffeomorphism) on a compact manifold $M$, then the natural extension of $f$ is $L_f = \{(x_n)_n \in M^{\zz}: f(x_{n}) = x_{n+1} \ \forall n \in \zz \}$. This is a bundle over $M$ with the projection  $\pi_f:L_f \to M$ as $\pi_f((x_n)_n) = x_0$, and the fiber over $x\in M$ consists of all the orbits starting at $x$. The map $f$ lifts to the homeomorphism $\Hat{f}: L_f \to L_f$ given by $\Hat{f}((x_n)_n) = (x_{n+1})_n$.

We may explicitly obtain the local structure of $L_f$ as a product of a $(\dim M)$-disk with a Cantor set, more specifically the Cantor set is $\Sigma:= \{1,\cdots,d\}^\nn$, $d = \deg f$. Given a simply connected open set $V \subset M$, we may identify $\pi_{f}^{-1}(V)$ with $V \times \Sigma$ as follows. Since $f$ is a local diffeomorphism and $V$ is simply connected, $f^{-1}(V)$ consists of exactly $d$ disjoint simply connected open sets $V_1,\cdots V_d$. Now, for each $i=1,\cdots, d$, we denote by $\{V_{i,j}\}_{j=1}^d$ the $d$ disjoint simply connected open sets that form $f^{-1}(V_i)$, then $f^{-2}(V)$ is the disjoint union of $d^2$ disjoint simply connected open sets $\{V_{i,j}\}_{i,j=1}^d$. We continue the process inductively to obtain that for every $n \ge 0$, $f^{-n}(V)$ is the disjoint union of sets $\{V_{\tau}\}_{\tau \in \{1,\cdots,d\}\}^n}$.

Then, we may consider the homeomorphism $\Xi: V\times \Sigma \to \pi_{f}^{-1}(V)$ given for $x \in V$ and $\omega = (\omega_n)_{n\ge 0}$ by
\begin{equation}\label{eq ident fiber with Cantor}
\Xi(x,\omega) = (x_n)_n, \ \ \text{as the unique points satisfying} \ \ x_{-n} = f^{-n}(x) \cap V_{\omega_1,\cdots, \omega_n}, \ \forall n \ge 0.
\end{equation}

\subsection{Invariant measures for endomorphisms}

We consider $f\in End^1(M)$ and $\Hat{f}:L_f\to L_f$ its natural extension, the invariant measures on $M$ and on $L_f$ are related as follows. We say that a measure $\Hat{\mu}$ on $L_f$ projects to a measure $\mu$ on $M$ if $(\pi_{f})_*\Hat{\mu} = \mu$. For every $f$-invariant probability measure $\mu$ on $M$, there exists a unique $\Hat{f}$-invariant probability $\Hat{\mu}$ on $L_f$ that projects to $\mu$. This unique measure $\Hat{\mu}$ is called the lift of $\mu$ to $L_f$. One may verify that $\mu$ is ergodic if and only if $\hat\mu$ is ergodic. Furthermore the metric entropy on the natural extension is the same as the one in the manifold, i.e. 
\begin{equation}\label{eq entropy manifold and sol}
h_\mu(f) = h_{\Hat{\mu}}(\Hat{f}),
\end{equation}
this is a classical result on the natural extension; see e.g. \cite{Qian2009smooth}).

\subsubsection{The Jacobian}
Let $\mathcal E$ be the partition of $M$ into points. Then $f^{-1}\mathcal E$ is a measurable partition of $M$ into subsets with $d$ elements (we assume that $f$ has degree $d$). We can disintegrate the measure $\mu$ along the elements of the partition $f^{-1}\mathcal E$, so for $\mu$-almost every $x\in M$, we obtain a probability measure $\overline \mu_{f^{-1}(x)}$. Using these disintegrations, we can define the {\it Jacobian of $f$ with respect to $\mu$ } as
\begin{equation}
    J_{\mu}(y)=\frac 1{\overline\mu_{f^{-1}(x)}(\{y\})},\ \ \text{ if } \ y\in f^{-1}(x).
\end{equation}

Let us remark that the Jacobian $J_{\mu}$ is well defined and inside $[1,\infty)$ on a set of full $\mu$-measure. An equivalent way to define the Jacobian is the following. Consider the lift $\hat\mu$ of $\mu$ to $L_f$, and consider the partition of $L_f$ given by the fibers $\pi_{f}^{-1}(x),\ x\in M$. Let $\mu_x$ be the disintegrations of $\hat\mu$ on $\pi_{f}^{-1}(x)$ for $\mu$-almost every $x$. Then
\begin{equation}\label{eq disit lift and jacobian}
    J_{\mu}(x)=\frac 1{\Hat{\mu}_{f(x)}(\hat f(\pi_{f}^{-1}(x)))}.
\end{equation}

The Jacobian is in fact the Radon-Nicodym derivative $\frac{df^*\mu}{d\mu}$, it satisfies the following equivalent conditions. If $A\subset M$ is such that $f|_A$ is injective, with the inverse $(f|_A)^{-1}:f(A)\rightarrow A$, then
\begin{equation}
    \mu(f(A))=\int_AJ_{\mu}d\mu,\ \text{and} \ \ 
    \mu(A)=\int_{f(A)}\frac 1{J_\mu\circ (f|_A)^{-1}}d\mu.
\end{equation}

It is also the essentially unique measurable function satisfying that for every continuous function $\varphi: M \to \rr$
\begin{equation}\label{eq jacobian and continuous functions}
    \int \varphi \ d\mu =  \int \sum\limits_{y \in f^{-1}(x)} \frac{\varphi(y)}{J_\mu (y)} \ d\mu(x).
\end{equation}

In general, convergence of measures do not imply convergence of the Jacobians. However, we may identify the Jacobian of a limit measure through the following result. 

\begin{prop}\label{prop limit meas is equidist}
        Let $f_n, f \in \End^1(M)$ be such that $f_n$ converges to $f$ in the $\mathcal{C}^1$-topology and consider a sequence of $f_n$-invariant measures $\mu_n$ converging to some $\mu$ in the weak$^*$ topology. Suppose that there is a continuous function $J: M \to \rr$ such that for any $\epsilon>0$
        \[
        \mu_n\{x \in M : \text{for some} \ y \in f_n^{-1}(x), \ |J_{\mu_n}(y) - J(y)| \ge \epsilon\} \to 0, \ \text{as} \ n \to \infty.
        \]
        Then $J_\mu  = J$ $\mu$-a.e.
\end{prop}
    \begin{proof}
        Consider $\varphi: M \to \rr$ a given continuous function. Note that 
        \begin{align*}
        \int \varphi \ d\mu_n &= \int \sum_{y \in f_n^{-1}(x)} \frac{\varphi(y)}{J_{\mu_n}(y)} \ d\mu_n (x)\\ 
        &= \int \sum_{y \in f_n^{-1}(x)}\left(\frac{1}{J_{\mu_n}(y)}-\frac{1}{J(y)}\right) \varphi(y) \ d\mu_n(x)  \ + \int \sum\limits_{y \in f_n^{-1}(x)} \frac{\varphi(y)}{J(y)} \ d\mu_n(x).
        \end{align*}

        The first term of the sum goes to zero as $n \to \infty$ from the hypothesis on $J_{\mu_n}$, and the second converges since $J$ is continuous and both $f_n \to f$ and $\mu_n \to \mu$. We conclude
        \[
        \int \varphi \ d\mu = \int \sum\limits_{y \in f^{-1}(x)} \frac{\varphi(y)}{J(y)}\ d \mu(x).
        \]
        Since $J_\mu $ is essentially unique satisfying the above relation, we conclude that $J_\mu = J$, $\mu$-almost everywhere.
    \end{proof}

The {\it folding entropy of $f$ with respect to $\mu$} is the relative entropy
\begin{equation}
    F_{\mu}(f)=H_{\mu}(\mathcal E\mid f^{-1}\mathcal E),
\end{equation}
or equivalently
\begin{equation}\label{eq folding entropy and jacobian}
    F_{\mu}(f)=\int \log J_{\mu} \ d\mu.
\end{equation}
It always satisfies that $F_\mu(f) \le \log d$ where $d$ is the degree of $f$. This bound plays a key role in the Jacobian of measures in the following sense.

\begin{prop}\label{prop. conv fold imply conv in meas jacobian}
     Let $f_n, f \in \End^1(M)$ be such that $f_n$ converges to $f$ in the $\mathcal{C}^1$-topology. If $\mu_n$ are $f_n$-invariant measures such that $F_{\mu_n}(f_n) \to \log d$, $d = \deg f$, then for any $\epsilon>0$
        \[
        \mu_n \{x \in \TT^2: \text{for some} \ y \in f_n^{-1}(x), \  |J_{\mu_n}(y)-d|\ge \epsilon\} \to 0 \ \text{as} \ n \to \infty. 
        \]
\end{prop}
    \begin{proof}
    Since $f_n\to f$ in the $\mathcal{C}^1$-topology, we may suppose that $\deg f_n = \deg f = d$ for all n. For each $n \in \nn$, consider the function (defined $\mu_n$-a.e.) 
    \begin{equation}\label{eq jensens ineq}
    \Psi_n(x) = -\sum\limits_{y \in f^{-1}(x)} \frac{1}{J_{\mu_n}(y)}\log \frac{1}{J_{\mu_n}(y)} \le \log d.
    \end{equation}
    From the relations between folding entropy and Jacobian \eqref{eq jacobian and continuous functions}, \eqref{eq folding entropy and jacobian}, we get
    \[
    F_{\mu_n}(f_n)=  \int \Psi_n(x) \ d \mu_n(x) \to \log d, \ \text{as} \ n \to \infty.
    \]
    
    Since the function
    \[
    \left\{(x_1,\cdots,x_d): \sum x_i=1 , \ x_i \in (0,1]  \right\} \ni (x_1,\cdots x_d) \mapsto -\sum x_i \log x_i,
    \]
    is uniformly continuous and achieves its maximum ($\log d$) if and only if $x_i = d^{-1}$ for all $i=1,\cdots,d$, it is enough to show that given $\epsilon>0$, $\mu_n\{x: \Psi_n(x)<\log d-\epsilon\}$ goes to zero as $n \to \infty$.
    
    Now, let $\epsilon_n\to 0$ be such that $F_{\mu_n}(f_n) = \log d- \epsilon_n$ and consider $\beta_n = \mu_n\{x \in \TT^2: \Psi_n(x)<\log d - \sqrt{\epsilon_n}\}$, then
    \[
    \log d -\epsilon_n = \int \Psi_n \ d\mu_n < \beta_n (\log d -\sqrt{\epsilon_n})+ (1-\beta_n) \log d.
    \]
    Hence $\beta_n \le \sqrt{\epsilon_n} \to 0$ as $n \to \infty$.
\end{proof}

\subsubsection{Equidistributed measures}\label{subsec equidistritued meas}
Consider $p$ the equally distributed Bernoulli measure on $\Sigma = \{1,\cdots ,d\}^\nn$. We say that the invariant measure $\mu$ is {\it equidistributed} if one of the following equivalent conditions holds:
\begin{enumerate}
    \item $J_{\mu}(x)=d$ for $\mu$-almost every $x$;
    \item $F_{\mu}(f)=\log d$;
    \item For any small open set $V\subset\TT^2$, given the homeomorphism $\Xi: V\times \Sigma \to \pi_{f}^{-1}(V)$ as in \eqref{eq ident fiber with Cantor}, we have 
    \begin{equation}\label{eq equidistributed meas}
        \Xi_*((\mu|_V)\times p) = \hat\mu|_{\pi_{f}^{-1}(V)}.
    \end{equation}
\end{enumerate}
In other words, the disintegrations $\mu_x$ of $\hat\mu$ on the fibers $\pi_{f}^{-1}(x)$ are Bernoulli equidistributed for $\mu$-almost every $x$.

If $f$ preserves the Lebesgue measure $Leb$ on $M$, then there exists a unique invariant lift $\hat{Leb}$ on $L_f$. However, in our case $Leb$ may not be invariant under $f$, so there is no dynamical lift of $Leb$ to $L_f$. We are interested in the equidistributed lift, which is defined as follows.

The equally distributed Bernoulli measure on the fibres of $\pi_{f}: L_f \to M$ may be obtained as
\[
p_x = \Xi_{*}(\delta_{x} \times p),
\]
We also define our reference measure $\Hat{\eta}$ on $L_f$ by
\begin{equation}\label{eq defn measure on Sol local leb times bernoulli}
\Hat{\eta} = \int p_x \ d\Leb(x).
\end{equation}
For any partition $\mathcal{P}$ of $M$ formed only by simply connected sets, we may define a partition $\Hat{\mathcal{P}}$ of $L_f$ given by $\Hat{\mathcal{P}} = \{\Xi(P \times \{\omega\}): P \in \mathcal{P}, \ \omega \in \Sigma\}$. Since the measures on the fibres are all equally distributed Bernoulli measures, the factor measure on $L_f/\Hat{\mathcal{P}}$ is equivalent to the Bernoulli measure $p$ on $\Sigma$. Moreover the conditional measures $\Hat{\eta}^\omega$ on $P_\omega= \Xi(P\times\{\omega\})$ is equivalent to $\Leb$ on $P \subset M$. More specifically, for every measurable set $\Hat{A} \subset L_f$
\begin{equation}\label{eq fubini on reference measure}
\Hat{\eta}(\Hat{A}) = \int \Leb(A^\omega) \ dp(\omega),
\end{equation}
where $A^\omega = \{x \in M: \Xi(x,\omega) \in \Hat{A}\}$. From \eqref{eq defn measure on Sol local leb times bernoulli} and \eqref{eq fubini on reference measure}, the reference measure $\Hat{\eta}$ is locally a product of the Lebesgue measure on the manifold with the equally distributed Bernoulli measure.

The following Corollary is a direct consequence of the relation between the disintegration of measures in $L_f$ with the Jacobian \eqref{eq disit lift and jacobian} along with Prop. \ref{prop limit meas is equidist} and \ref{prop. conv fold imply conv in meas jacobian} .

\begin{corol}\label{corol dist meas on fibers to p}
    Let $f_n, f \in \End^1(M)$ be such that $f_n$ converges to $f$ in the $\mathcal{C}^1$-topology. If $\mu_n$ are $f_n$-invariant measures such that $F_{\mu_n}(f_n) \to \log d$, $d = \deg f$, then for any continuous $\varphi:M \to \rr$ and any $\epsilon>0$
    \[
    \mu_n\left\{x \in \TT^2: \left|\int_{\pi_{f}^{-1}(x)} \varphi d\Hat{\mu}_{x,n}-\int_{\pi_{f}^{-1}(x)} \varphi \ dp_x  \right| \ge \epsilon\right\} \to 0, \ \text{as} \ n \to \infty,
    \]
    where $\Hat{\mu}_{x,n}$ is the disintegration of $\Hat{\mu}_n$ in $\pi_{f}^{-1}(x)$. In particular, any accumulation measure of $\{\mu_n\}_n$ is equidistributed.
\end{corol}

We also detail here another reference measure we'll be using throughout the paper. Given $\sigma:[0,1] \to M$ a $\mathcal{C}^r$ curve, consider $\Leb_\sigma$ the Lebesgue measure on $\sigma$ induced as a submanifold of $M$. We'll consider the measure $\Hat{\eta}_\sigma$ on $\pi_{f}^{-1}(\sigma)$ as the equidistributed lift of $\Leb_\sigma$, that is
\begin{equation}\label{eq defn Leb times bernoulli on curves}
    \Hat{\eta}_\sigma = \int p_x \ d \Leb_{\sigma}(x).
\end{equation}
It also satisfies for every measurable $\Hat{A} \subset L_f$
\begin{equation}\label{eq fubini on meas curves}
    \Hat{\eta}_\sigma(\Hat{A}) = \int \Leb_\sigma(A^\omega) \ dp(\omega).
\end{equation}

\subsubsection{Lyapunov exponents}
In this section we restrict ourselves to the 2 dimensional case and let $f \in \End^1(\TT^2)$ .For $(x,v) \in T^1\TT^2$, the {\it Lyapunov exponent of $f$ at $(x,v)$} is defined as 
\[
\chi(x,v) := \limsup\limits_{n\to +\infty} \frac{1}{n}\log\|Df^n_x\cdot v\| 
\]
and it measures the asymptotic rate of growth of $v$ along the orbits of $f$. 

From Oseledets Theorem \cite{oseledets1968multiplicative}, any $f$-invariant probability measure $\mu$ satisfies that for $\mu$-almost every point $x \in \TT^2$, and for every $v\neq 0 \in T_x\TT^2$ the limit $\chi(x,v) =  \lim\limits_{n\to +\infty} \frac{1}{n}\log\|Df^n_x\cdot v\|$ exists and is one of the numbers 
\[
\chi^+(x) := \lim\limits_{n \to +\infty} \frac{\log \|Df_x^n\|}{n}, \ \chi^-(x) := \lim\limits_{n \to +\infty} \frac{\log m(Df_x^n)}{n},
\]
called the Lyapunov exponents of $f$ at $x$. One also has
\[
\int (\chi^-(x)+\chi^+(x))\ d\mu(x) = \int \log |\det Df_x| \ d\mu(x).
\]

The Lyapunov exponents of $\mu$ are defined as 
\[
\chi^+(\mu) = \int \chi^+(x)\ d\mu(x), \ \ \ \chi^-(\mu) = \int \chi^-(x)\ d\mu(x).
\]

The differential cocycle $Df$ over $f$ defines a cocycle $D\Hat{f}$ over $\hat{f}$ acting on $TL_f:= L_f \times \rr^2$ given by 
\[
D\Hat{f}_{\Hat{x}} = Df_{\pi_f(\Hat{x})}: \rr^2 \to \rr^2,
\]
which with the usual cocycle notations satisfy for $n \ge 0$ and $\Hat{x} = (x_k)_{k \in \zz}$:
\begin{align*}
    &D\Hat{f}_{\Hat{x}}^{n} =  Df_{x_{n-1}}  \cdots  Df_{x_1}\  Df_{x_0} = Df^n_{x_0},\\
    &D\Hat{f}_{\Hat{x}}^{-n} = (Df_{x_{-n}})^{-1} \cdots (Df_{x_{-2}})^{-1} \ (Df_{x_{-1}})^{-1} = (Df^{n}_{x_{-n}})^{-1}, 
\end{align*}
Thus, we may consider the Lyapunov exponents associated with $D\Hat{f}$ given by
\[
\chi_{\Hat{f}}(\Hat{x},v) = \limsup\limits_{n \to +\infty} \frac{1}{n} \log \|D\Hat{f}^n_{\Hat{x}}\cdot v\|.
\]
It satisfies that for every $\Hat{x} \in L_f$ and $v \in T^1_{\pi_f(\Hat{x})}\TT^2$
\[
\chi_{\Hat{f}}(\Hat{x},v) = \chi_f(\pi_{f}(\Hat{x}),v).
\]
moreover, from Oseledets Theorem, given any $f$-invariant measure $\mu$, if $\Hat{\mu}$ is the unique $\Hat{f}$-invariant lift of $\mu$ to $L_f$ then for $\Hat{\mu}$-almost every $\Hat{x} \in L_f$:
\begin{align*}
&\chi^+_f(\pi_{f}(\Hat{x})) = \chi^+_{\Hat{f}}(\Hat{x}) = -\chi^-_{\Hat{f}^{-1}}(\Hat{x})\\
&\chi^-_f(\pi_{f}(\Hat{x})) = \chi^-_{\Hat{f}}(\Hat{x}) = -\chi^+_{\Hat{f}^{-1}}(\Hat{x}).
\end{align*}

If $\chi^-(x)\neq\chi^+(x)$, the Oseledets Theorem also gives the existence of the invariant one dimensional Oseledets sub-bundle $E^-$ along the orbit of $x$ such that
\begin{equation}
    \chi^-(x) =  \lim\limits_{n\to +\infty} \frac{1}{n}\log\|Df^n_x\cdot v^-\|,\ \forall v^-\in E^-(x)\setminus\{0\}.
\end{equation}

The other Oseledets sub-bundle $E^+$ is not well defined on $\TT^2$ because it depends on the backward trajectory, however it is well defined on $L_f$: if $\chi^-(\hat x)\neq\chi^+(\hat x)$ then there is an invariant one-dimensional Oseledets bundle $E^+$ along the orbit of $\hat x$ such that $\rr^2 = E^-\oplus E^+$ and
\begin{equation}
    \chi^+(\hat x) =  \lim\limits_{n\to \pm\infty} \frac{1}{n}\log\|D\hat f^n_{\hat x}\cdot v^+\|,\ \forall v^+\in E^+(\hat x)\setminus\{0\}.
\end{equation}

The measure $\mu$ is {\it hyperbolic} when $\chi^-(x)<0<\chi^+(x)$ for $\mu$-a.e. $x \in \TT^2$.

\subsection{Pesin Theory}\label{sect Pesin}

We consider $f\in End^r(M)$, $r>1$, and $\Hat{f}:L_f\to L_f$ its natural extension.

\begin{defn}  For $\Hat{x} \in  L_f$ and $x = \pi_{f}(\Hat{x})$.
\begin{enumerate}
    \item A $C^{1+}$ embedded manifold $W^u(\Hat{x}) \subset M$ is a local unstable manifold of $f$ at $\Hat{x}$ if there are numbers $\lambda>0$, $0<\epsilon< \lambda/200$, $0 \le C_1 \le 1 \le C_2$ such that for each $y \in W^u(\Hat{x})$ there is a unique $\Hat{y} \in  \pi_{f}^{-1}(y)$ satisfying for every $n \ge 0$
    \begin{equation}\label{eq unst manifold}
        d(x_{-n}, y_{-n}) \le C_1e^{-n\epsilon}, \ \text{and} \ d(x_{-n},y_{-n}) \le C_2e^{-n\lambda},
    \end{equation}
   the local unstable set $\Hat{W}^u(\Hat{x})\subset L_f$ of $\Hat{f}$ at $\Hat{x}$ is 
    \begin{equation}
        \Hat{W}^u(\Hat{x}) = \{\Hat{y} \in L_f: \Hat{y} \ \text{satisfies \eqref{eq unst manifold}}\}.
    \end{equation}
    The global unstable set $\mathcal{V}^u(\Hat{x}) \subset L_f$ is defined as
    \begin{equation}
        \mathcal{V}^u(\Hat{x}) = \bigcup\limits_{n\ge 0} \Hat{f}^n(\Hat{W}^u(\Hat{f}^{-n}(\Hat{x})))
    \end{equation}
    
    \item Analogously, a $C^{1+}$ embedded manifold $W^s(x) \subset M$ is a local stable manifold of $f$ at $x$ if there are numbers $\lambda>0$, $0<\epsilon< \lambda/200$, $0 \le C_1 \le 1 \le C_2$ such that each $y \in W^s(x)$ satisfies for every $n \ge 0$
     \begin{equation}
        d(f^n(x), f^n(y)) \le C_1e^{-n\epsilon}, \ \text{and} \ d(f^n(x),f^n(y)) \le C_2e^{-n\lambda}.
    \end{equation}
    We take $\Hat{W}^s(\Hat{x}) = \pi_{f}^{-1}(W^s(\Hat{x})) \subset L_f$ and define the global stable set $\mathcal{V}^s(\Hat{x}) \subset L_f$ as
    \begin{equation}
        \mathcal{V}^s(\Hat{x}) = \bigcup\limits_{n \ge 0}\Hat{f}^{-n}(\Hat{W}^s(\Hat{f}^n(\Hat{x})))
    \end{equation}
\end{enumerate}
\end{defn}

Given $\lambda,\epsilon,K>0$, a $(\lambda,\epsilon,K)$-Pesin block is a non-empty set $\Hat{\Lambda} \subset L_f$ for which there is a direct sum decomposition $T_{x_0}M = E^s(\Hat{x})\oplus E^u(\Hat{x})$ for every $\Hat{x} = (x_k)_k \in \bigcup_{n\ge 0}\Hat{f}^n(\Hat{\Lambda})$, such that for any $n \in \zz$, $k\ge 0$ and $\Hat{y} \in \Hat{\Lambda}$
\begin{equation}\label{eq defn pesin block inverse limits}
    \max\{ \|D\Hat{f}^k|_{E^s(\Hat{f}^n(\Hat{y}))}\|, \ \|D\Hat{f}^{-k}|_{E^u(\Hat{f}^n(\Hat{y}))}\|\} \le K \exp(-\lambda k+\epsilon|n|).
\end{equation}

We say that $\Hat{\Lambda}$ is a $(\lambda,\epsilon)$ Pesin block if it is a $(\lambda,\epsilon,K)$ Pesin block for some $K>0$.

For the following theorem we consider $\mu$ any $f$-invariant hyperbolic probability measure on $M$ and $\Hat{\mu}$ the unique $\Hat{f}$-invariant lift of $\mu$ to $L_f$.

\begin{thm}\label{ThmExistencePesinBlocks} (see \cite{ZhuS1998UnstableMan,Qian2009smooth,LW25})
The $(\lambda,\epsilon,K)$-Pesin blocks $\Hat{\Lambda}_{\lambda,\epsilon,K}$ satisfy 
\[
\Hat{\mu}\left(\bigcup_{\lambda,\epsilon,K>0} \Hat{\Lambda}_{\lambda,\epsilon;K}\right) = 1,
\]
and for any fixed Pesin block $\Hat{\Lambda}$:
\begin{enumerate}
    \item There exists a continuous family $\{W^u(\Hat{x}) \subset M:\Hat{x}\in\Hat{\Lambda}\}$ of local unstable manifolds  so that for every $\Hat{x} = (x_n)_n \in \Hat{\Lambda}$, it holds
    \begin{itemize}
        \item $T_{x_0}W^u(\Hat{x}) = E^u(\Hat{x})$ is the Oseledets subspace associated to the positive Lyapunov exponents at $\Hat{x}$, in particular $E^u(\Hat{x})$ depends continuously on $\Hat{x} \in \Hat{\Lambda}$,
        \item $f(W^u(\Hat{f}^{-1}(\Hat{x}))) \supset W^u(\Hat{x})$.
    \end{itemize}

    \item If $\Lambda = \pi_{f}(\Hat{\Lambda})$, then there exists a continuous family of local stable manifolds $\{W^s(x):x \in \Lambda \}$ so that for every $x \in \Lambda$ it holds:
    \begin{itemize}
        \item $T_xW^s(x) = E^s(x)$ is the Oseledets subspace associated to the negative Lyapunov exponents, in particular $E^s(x)$ depends continuously on $x \in \Lambda$,
        \item $f(W^s(x)) \subset W^s(f(x))$.
    \end{itemize}
\end{enumerate}
\end{thm}

\subsubsection{Absolute continuity}\label{subsec abs cont}
We also have that the holonomy along stable and unstable manifolds within Pesin blocks is absolutely continuous. Let us explain the absolute continuity of unstable holonomy with more details.

Given a simply connected open set $V \subset M$, we may identify $\pi_{f}^{-1}(V)$ with $V \times \Sigma$ by considering the homeomorphism $\Xi: V\times \Sigma \to \pi_{f}^{-1}(V)$ given by \eqref{eq ident fiber with Cantor}. We consider the partition of $\pi_{f}^{-1}(V)$ into leaves $\{ V_\omega := \Xi(V\times\{\omega\}): \omega \in \Sigma\}$ and we define the unstable holonomy on the leaves as follows. The unstable manifolds $W^u$ do not form a lamination of $M$ since they may intersect, however their lifts $\Hat{W}^u$ do form a lamination of $L_f$. Thus, for each $\omega \in \Sigma$ the restriction $\Hat{W}^u_\omega = \{\Hat{W}^u(\Hat{x}): \Hat{x} \in V_\omega\}$ do form a lamination of $V_\omega$. We may then consider the lamination $W^u_\omega$ on $V$ given by $W^u_\omega(x) = \pi_{f}(\Hat{W}^u(\Xi(x,\omega)))$.

By restricting ourselves to a Pesin block $\Hat{\Lambda} \subset L_f$ with positive $\Hat{\mu}$-measure, $\Hat{W}^u(\Hat{x})$ varies continuously with $\Hat{x} \in  \Hat{\Lambda} \cap V_\omega$, hence their length are uniformly away from zero. Setting $B_\omega = \bigcup_{\Hat{x} \in \Hat{\Lambda} \cap V_\omega} \pi_{f}(\Hat{W}^u(\Hat{x}))$, given $T_1, T_2 \subset V$ two submanifolds uniformly transverse to the lamination $W^u_\omega$, we may define the leaf unstable holonomy $h^u: T_1 \cap B_\omega \to T_2 \cap B_\omega$ as $h^u(x) = W^u_\omega(x) \cap T_2$. 

\begin{lema}\label{lema abs cont unst holon}
Given a simply connected open set $V \subset M$ and $\omega \in \Sigma$, let  $W^u_{\omega}$ be the unstable lamination of $V$ associated with $\omega$. Given two submanifolds $T_1, T_2 \subset V$ transversal to $W^u_\omega$ such that $\Leb_{T_1}(B_\omega)>0$, the unstable holonomy $h^u: T_1 \cap B^k_\omega \to T_2 \cap B^k_\omega$ is absolutely continuous with respect to the Lebesgue measures of the two transversals.

As a consequence, the conditional measures associated with the disintegration of the Lebesgue measure on $V$ along the elements of the lamination $W^u_\omega$ are absolutely continuous with respect to the Lebesgue measure on unstable manifolds.
\end{lema}

\subsubsection{Inverse SRB measures}

A partition $\xi$ of $\TT^2$ is subordinated to the stable manifolds if for $\mu$-almost every $x \in \TT^2$ the element $\xi(x)$ of the partition containing $x$ satisfies:
\begin{itemize}
    \item $\xi(x) \subset W^s(x)$ and,
    \item $\xi(x)$ contains an open neighborhood of $x$ inside $W^s(x)$.
\end{itemize}
When $\xi$ is a measurable partition, one may disintegrate the measure $\mu$ along the elements of the partition obtaining $\mu = \int_{\TT^2/\xi} \mu_x^{\xi} d \mu_{\xi}$, where $\mu_x^{\xi}$ are probability measures on $\xi(x)$ and $\mu_\xi$ is the factor measure on $\TT^2/\xi$.

A hyperbolic invariant measure $\mu$ has absolutely continuous conditional measures on stable manifolds if for every measurable partition $\xi$ of $\TT^2$ subordinated to the stable manifolds, $\mu^\xi_x \ll \Leb^s_x$ for $\mu$-a.e. $x \in \TT^2$, where $\Leb^s_x$ is the Lebesgue measure on $W^s(x)$ inherited from $\TT^2$ as a submanifold.

\begin{defn}
    The invariant measure $\mu$ of the endomorphism $f\in End^r(\TT^2)$ is an {\it inverse SRB measure} if it is hyperbolic and it has absolute continuous conditional measures on the stable manifolds.
\end{defn}

Liu \cite{liu2008invariant} shows that there always exists at least one measurable partition subordinated to the stable manifolds. Moreover, if we suppose that $\mu$ has a Hölder continuous Jacobian, then $\mu$ has absolutely continuous conditional measures on stable manifolds if and only if it satisfies the equality $h_{\mu}(f) = F_{\mu}(f) -\int_M\chi^-d\mu$ (see also \cite{LW25} for a more relaxed hypothesis). 

If $\mu$ is equidistributed then it has constant Jacobian $J_{\mu}(x)=\log d$ almost everywhere. In particular, it holds:

\begin{prop}
    Let $f\in End^r(\TT^2)$, $r>1$, and $\mu$ a hyperbolic invariant measure for $f$. Then $\mu$ is an equidistributed inverse SRB measure for $f$ if and only if the following formula holds:
    $$h_{\mu}(f)=\log d-\int_M\chi^-d\mu.$$
\end{prop}

\subsection{SPR Property}\label{sect spr}

The strong positive recurrence property (SPR) was defined by J. Buzzi, S. Crovisier and O. Sarig in \cite{buzzi2025strong} for diffeomorphisms on compact manifolds. The authors prove a series of consequences that such property yields for equilibrium states. We present here how to extend such ideas for non invertible local diffeomorphisms.

We consider $g:M \to M$ a $\mathcal{C}^1$ diffeomorphism on a compact manifold $M$. We recall some classical notions on Pesin Theory \cite{barreira2023introduction}. Given $\lambda,\epsilon>0$, a $(\lambda,\epsilon)-$Pesin block is a non-empty set $
\Lambda \subset M$ for which there are a direct sum decomposition $T_xM = E^s(x)\oplus E^u(x)$ for every $x \in \bigcup_{n\ge 0}g^n(\Lambda)$, and a uniform constant $K>0$ such that for any $n \in \zz$, $k\ge 0$ and $y \in \Lambda$
\begin{equation}\label{eq defn pesin block diffeo}
    \max\{ \|Dg^k|_{E^s(g^n(y))}\|, \ \|Dg^{-k}|_{E^u(g^n(y))}\|\} \le K \exp(-\lambda k+\epsilon|n|).
\end{equation}

Consider $X \subset M$ a $g$-invariant Borel set and a potential $\phi: X \to \rr \cup \{-\infty\}$ which is a Borel function with $\sup\phi <\infty$. Given a $g|_X$-invariant measure $\nu$, the pressure of $\phi$ relative to $\nu$ is given by
\begin{equation}
P(g,\nu,\phi) := h_\nu(g)+ \int \phi \ d\nu, 
\end{equation}
and the top pressure of $\phi$ on $X$ is
\[
P_{\text{top}}(g|_X,\phi) := \sup\{P(g,\nu,\phi): \nu \ \text{is} \ g|_X\text{-invariant} \}.
\]
For a compact invariant set $X$ and a continuous potential $\phi$, $P_{\text{top}}(g|_X,\phi)$ equals the topological pressure of $\phi$ on $X$ from Walters variational principle \cite{walters2000introduction}.

A $g$-invariant measure $\mu$ is an equilibrium measure of $\phi$ on $X$ if $P(g,\mu,\phi) = P_{\text{top}}(g|_X,\phi)$ or just an equilibrium measure if $X = M$. We are particularly interested in SRB and inverse SRB measures which are the equilibrium states of the geometric potentials
\begin{equation*}
\phi_u = -\log |\det Dg|_{E^u}|, \ \phi_s = \log |\det Dg|_{E^s}|,
\end{equation*}
and $\phi_\sigma(x) = -\infty$ if $E^\sigma$ is not well defined at $x$, $\sigma = u,s$.

\begin{defn}
    A diffeomorphism $g$ is Strongly Positive Recurrent (SPR) for a potential $\phi$ on an invariant Borel set $X \subset M$ if there exists $\lambda>0$ such that for each $\epsilon>0$ there are a Borel $(\lambda,\epsilon)$-Pesin block $\Lambda$ and numbers $P_0<P_{\text{top}}(g|_X,\phi)$, $\tau>0$ such that, for every $g|_X$-invariant measure $\nu$:
    \[
    P(g,\nu,\phi)>P_0 \ \Longrightarrow \ \nu(\Lambda)>\tau.
    \]
\end{defn}

We present some consequences of the SPR property proven in \cite{buzzi2025strong}, we refer on each result for precise statements.

\begin{thm}\label{thm conseq SPR diffeo} Suppose that $g \in \text{Diff}^{1+}(M)$ is SPR for some H\"older continuous (or geometric) potential $\phi$. Then any equilibrium measure of $\phi$ satisfies:
\begin{itemize}
    \item Exponential decay of correlations [Theorem 11.10, \cite{buzzi2025strong}].

    \item Large deviations of Birkhoff sums [Theorem 11.15, \cite{buzzi2025strong}]. 
    \item Almost sure invariance principle for Birkhoff sums [Theorem 11.19, \cite{buzzi2025strong}].
    \item Central limit theorem and convergence of moments [Corollary 11.14, \cite{buzzi2025strong}].
\end{itemize}
\end{thm}

There are other consequences of the SPR property such as asymptotic variance of Birkhoff sums [Theorem 11.13, \cite{buzzi2025strong}], among others. For a full list see [Subsection 11.15,\cite{buzzi2025strong}].

\begin{rmk}
    The authors actually prove such results for any quasi-H\"older potential (see Section 11, \cite{buzzi2025strong} for precise definition) which includes any H\"older continuous potential as well as the geometric potentials.
\end{rmk}

Now, for a typical point $x \in M$, consider
\[
\chi^1(x) \ge \chi^2(x) \ge \cdots \ge \chi^k(x), \ k = \dim M
\]
the Lyapunov exponents of $g$ at $x$ and consider $\Lambda^+(x) = \sum_{\chi^i \ge 0} \chi^i(x)$ and for a $g$-invariant measure $\mu$, let $\Lambda^+(\mu) = \int \Lambda^+(x) d \mu$. In \cite{buzzi2025strong}, the authors give sufficient conditions for the SPR property in terms of the Lyapunov exponents of $g$. For that, let $X \subset M$ an invariant Borel set and consider the following conditions.
\begin{enumerate}
    \item[\namedlabel{PH Prop.}{(PH)}] Pressure Hyperbolicity of $g$ for $\phi$ on $X$: There exists $\lambda>0$ as follows. Let $\{\mu_n\}_n$ be a sequence of $g|_X$-invariant ergodic measures converging to some $\mu$ in the weak$^*$-topology. If $P(g,\mu_n,\phi) \underset{n}{\longrightarrow} P_{\text{top}}(g|_X,\phi)$, then there exists a constant $i = i(\mu)$ such that $\chi^i(x)>\lambda>-\lambda> \chi^{i+1}(x)$, for $\mu$-almost every $x \in M$.
    \item[\namedlabel{PC Prop.}{(PC)}] Pressure continuity of $\Lambda^+$ for $\phi$ on $X$: Let $\{\mu_n\}_n$ be a sequence of $g|_X$-invariant ergodic measures converging to some $\mu$ in the weak$^*$-topology. If $P(g,\mu_n,\phi) \underset n \longrightarrow P_{\text{top}}(g|_X,\phi)$, then $\Lambda^+(\mu_n) \underset n \longrightarrow \Lambda^+(\mu)$.
\end{enumerate}

\begin{thm}[Theorem 3.1 and Remark 3.5,\cite{buzzi2025strong}] \label{thm SPR iff PH and PC diffeos} Let $g$ be a $\mathcal{C}^1$ diffeomorphism on a closed manifold $M$, $X \subset M$ a Borel invariant set and $\phi:X\to \rr \cup \{-\infty\}$ a potential. If $g$ is pressure hyperbolic and $\Lambda^+$ is pressure continuous for $\phi$ on $X$, then $g$ is SPR for $\phi$ on $X$. 
\end{thm}

\subsubsection{SPR Property for endomorphisms}

We consider now $f:M \to M$ a $\mathcal{C}^1$ local diffeomorphism on a compact manifold $M$. The natural way to define the SPR property is to consider the Pesin blocks as for diffeomorphisms, but since one point may have many unstable directions depending on its past orbits, it can only be defined on the inverse limit space $L_f$. Recall the definition of $(\lambda,\epsilon)$-Pesin blocks in \eqref{eq defn pesin block inverse limits}.

Given a potential $\phi: L_f \to \rr\cup \{-\infty\} $ which is a Borel function with $\sup\phi <\infty$ and an $f$-invariant measure $\nu$, the pressure of $\phi$ relative to $\nu$ is given by
\begin{equation}
P(f,\nu,\phi) := h_{\nu}(f)+ \int_{L_f} \phi \ d\Hat{\nu}, 
\end{equation}
where $\Hat{\nu}$ is the unique $\Hat{f}$-invariant lift of $\nu$ to $L_f$. The top pressure of $\phi$ is
\[
P_{\text{top}}(f,\phi) := \sup\{P(f,\nu,\phi): \nu \ \text{is} \ f\text{-invariant} \}.
\]

Again, an $f$-invariant measure $\mu$ is an equilibrium measure of $\phi$ if $P(f,\mu,\phi) = P_{\text{top}}(f,\phi)$. For example the equilibrium measures for $\phi \equiv 0$ are the measures of maximal entropy. As when considering the geometric potentials
\begin{equation}\label{eq defn geometric potentials}
\phi_s = \log |\det Df|_{E^s}|, \ \text{and} \ \phi_u = -\log |\det D\Hat{f}|_{E^u}|,
\end{equation}
and $\phi_\sigma(\Hat x) = -\infty$ if $E^\sigma$ is not well defined at $\Hat x$, $\sigma = s,u$, the equilibrium measures are respectively the inverse and forward SRB measures.

\begin{defn}
    $f$ is Strongly Positive Recurrent (SPR) for a potential $\phi$ if there exists $\lambda>0$ such that for each $\epsilon>0$ there are a Borel $(\lambda,\epsilon)$-Pesin block $\Hat{\Lambda}$ and numbers $P_0<P_{\text{top}}(f,\phi)$, $\tau>0$ such that, for every measure $f$-invariant measure $\nu$:
    \[
    P(f,\nu,\phi)>P_0 \ \Longrightarrow \ \Hat{\nu}(\Hat{\Lambda})>\tau.
    \]
\end{defn}

We remark that one can consider the SPR property when restricted to $\Hat f$-invariant sets as in the diffeomorphism case, but we only consider the global definition for our purposes. 

We would like to obtain the same results as Theorem \ref{thm conseq SPR diffeo} for endomorphisms as well as the characterization of the SPR property through the Lyapunov exponents as Theorem \ref{thm SPR iff PH and PC diffeos}. The obstruction for the conclusion here is that the natural extension $L_f$ is not a manifold. We consider then the smooth natural extension of $f$ commonly known as the realization of $f$.

\begin{defn}
    Consider $f: M \to M$ a local diffeomorphism. A realization of $f$ is a diffeomorphism $g: N\to N$ for some compact manifold $N$ for which there is an invariant attractor $X \subset N$ and a homeomorphism $h: L_f \to X$ such that $h \circ \Hat{f} = g|_X\circ h$.
\end{defn}

\begin{prop}\label{lema nice realization}
    Every local diffeomorphism $f:M \to M$ on a compact manifold $M$ has a realization $g:N \to N$ of the same regularity satisfying the following. Let $X \subset N$ and $h: L_f \to X$ the homeomorphism that conjugates $g|_X$ with $\Hat{f}$. There is a constant $0<\theta<1$ such that for every $x \in X$ there is an invariant splitting $T_xN = E \oplus F$, with $\dim E = \dim M$, such that:
    \begin{itemize}
        \item The Oseledets subspaces of $g$ at $x$, when well defined, are contained in $E_x \cup F_x$;
        \item $Dg|_F$ is a uniform contraction, that is, for every $x \in X$ and every unit vector $w \in F_x$, $\|Dg_x\cdot w\| = \theta$.
        \item $\Hat{x} \in L_f$ is Lyapunov regular if and only if $h(\hat{x})$ is as well. Moreover, the Lyapunov exponents of $h(\Hat{x})$ are exactly the exponents at $\Hat{x}$ in the $E$-direction and are all equal to $\log \theta$ in the $F$-direction.
    \end{itemize}
    Moreover, $\Hat{\Lambda}$ is a Pesin block for $f$ if and only if $h(\Hat{\Lambda})$ is a Pesin block for $g$. More specifically if $\Tilde{\Lambda}$ is a $(\lambda,\epsilon)$-Pesin block for $g$, then $h^{-1}(\Tilde{\Lambda})$ is a $(\lambda,\epsilon)$-Pesin block for $f$. Conversely, if $\Hat{\Lambda}$ is a $(\lambda,\epsilon)$-Pesin block for $f$ then $h(\Hat{\Lambda})$ is a $(\max\{\lambda,-\log \theta\},\epsilon)$-Pesin block for $g$.
\end{prop}
\begin{proof}
    Let $\iota : M \to D^m$ be a Whitney embedding of $M$ into some $m$-dimensional unitary disk $D^m \subset \rr^m$, with $\|D\iota\|_\infty<m(Df)$, where $m(Df) = \inf_x\|(Df_x)^{-1}\|^{-1}$. Since $f$ is a local diffeomorphism and $\iota$ is uniformly continuous, there is some $\delta>0$ such that if $\|\iota(x)-\iota(y)\|<\delta$ and $x\neq y$ then $f(x) \neq f(y)$. Fix some $\theta < \min\{\delta/2, m(Df),\|Df\|_{\infty}^{-1}\}$, and define $\tilde{g}: M \times D^m \to M \times D^m$ by
    \[
    \tilde{g}(x,z) = (f(x), \iota(x) + \theta z).
    \]
    Then $\tilde{g}$ is a local diffeomorphism of the same regularity as $f$ and, if $\theta$ is small enough, the image of $M \times D^m$ through $\tilde{g}$ is relatively compact in $M \times D^m$, hence the intersection $X = \bigcap_{n\ge 0}\tilde{g}^n(M\times D^m)$ is a $\tilde{g}$-invariant compact attractor.
    
    The choice of $\theta < \delta/2$ guarantees that $\tilde{g}$ is injective. Moreover,  given $\Hat{x} = (x_n)_n \in L_f$, then for all $k \ge 0$ the image $D_k(\Hat{x}):=\tilde{g}^k(\{x_{-k}\}\times D^m)$ is an $m$-dimensional disk of radius $\theta^k$ inside $\{x_0\}\times D^m$. The disks $\{D_k(\Hat{x})\}_{k \ge 0}$ satisfy $\ol{D_k(\Hat{x})} \subset D_{k+1}(\Hat{x})$, hence their intersection is a unique point which we define as $h(\Hat{x})$. One can check that $h: L_f \to X$ defined this way is a homeomorphism that conjugates $\tilde{g}|_X$ with $\Hat f$.

    We take $\alpha>\frac{\|D\iota\|_\infty}{(m(Df)-\theta)}$, note that we may assume $0<\alpha<1$ since we can choose $\iota$ as we want, and consider the cone $C_{\alpha}(x,z) = \{(v,w) \in T_xM \times \rr^m: \|w\| \le \alpha\|v\|\}$. Note that $D\tilde{g}_{(x,z)}\cdot (v,w) = (Df_x\cdot v,\  D\iota _x\cdot v+ \theta w)$, hence for any $(v,w) \in C_{\alpha}$ we have
    \begin{align*}\label{eq cone preserved realization}
        \|D\iota_x \cdot v +\theta w\| \le (\|D\iota\|_\infty+\theta \alpha) \|v\|  < \alpha \  m(Df) \|v\| \le \alpha\|Df_x \cdot v\|.
    \end{align*}
    Thus, $\ol{D\tilde{g} \cdot C_\alpha} \subset C_\alpha$. In particular, see e.g. [Section 3,\cite{crovisier2015introduction}], this implies that we have for every $\tilde{x} \in X$ a $D\tilde{g}$-invariant dominated splitting $T_{\tilde{x}}(M\times D^m) = E_{\tilde{x}}\oplus F_{\tilde{x}}$ (one can check that $F_{\tilde{x}} =\{0\}\times  T_{\tilde{x}}D^m$ and that $E \subset C_\alpha$ has the same dimension as $M$). Moreover $D\tilde{g}|_{F}$ is uniformly contracting at constant rate $\log \theta<0$. 
    
    Therefore, if $\tilde{x}$ is in some Pesin block satisfying \eqref{eq defn pesin block diffeo}, then necessarily there is some $D\tilde{g}$-invariant subspace $E^1_{\tilde{x}} \subset E_{\tilde{x}}$ (that could be trivial) such that $E^s_{\tilde{x}} = E^1_{\tilde{x}}\oplus F_{\tilde{x}}$ and $E^u_{\tilde{x}} \subset E_{\tilde{x}}$.
    
    Note that for $\tilde{x} = h((x_n)_n) \in X$ and $k > 0$,
    \[
    D\tilde{g}^k_{\tilde{x}} = \begin{pmatrix}
        Df^k_{x_0} & 0\\
        A^k(\tilde{x}) & \theta^k Id
    \end{pmatrix}, \ \ D\tilde{g}^{-k}_{\tilde{x}} = \begin{pmatrix}
        (Df_{x_{-k}}^k )^{-1} & 0 \\
        A^{-k}(\tilde{x}) & \theta^{-k} Id
    \end{pmatrix},
    \]
    where $A^k(\tilde{x}) = \sum_{j=0}^{k-1}\theta^{k-(j+1)}D\iota_{x_j} Df_{x_0}^j$ and $A^{-k}(\tilde{x}) = \sum_{j=1}^k-\theta^{-k+j-1}D\iota_{x_{-j+1}} (Df_{x_{-j}}^j)^{-1}$. From the choice of $\theta$, we get that for $v = (v_1,v_2) \in C_\alpha$, $\chi_{\tilde{g}}(\tilde{x},v) = \chi_{f}(x_0,v_1)$. 
    
    Taking $\pi_1:T(M\times D^m) \to TM$ the projection into the first coordinate, if $\tilde{x} = h((x_n)_n) \in X$ is in a $(\lambda,\epsilon)$-Pesin block for $\tilde g$, we take $E^u_{x_n} = \pi_1(E^u_{\tilde{g}^n(\tilde{x})})$ and $E^s_{x_n} = \pi_1(E^1_{\tilde{g}^n(\tilde{x})})$. Then
    \[
    \|Df^k_{x_n}|_{E^s_{x_n}} \| \le \|D\tilde{g}_{\tilde{g}^n(\tilde{x})}^k|_{E^1_{\tilde{g}^n(\tilde{x})}}\| \le C\exp(-\lambda k+\epsilon|n|), \ \ \ \forall k \ge 0, n \in \zz.
    \]
    Analogously, we obtain the same bounds for $\|(Df_{x_{n-k}}^k)^{-1}|_{E^u_{x_n}}\|$ to conclude that these direction satisfy \eqref{eq defn pesin block inverse limits} for the same $(\lambda,\epsilon)$. Hence if $\Tilde{\Lambda}$ is $(\lambda,\epsilon)$-Pesin block for $\tilde{g}$, then $h^{-1}(\Tilde{\Lambda})$ is a $(\lambda,\epsilon)$-Pesin block for $f$.

    Conversely if $\Hat{x} \in L_f$ is in some $(\lambda,\epsilon)$-Pesin block $\Hat{\Lambda}$ for $f$, then we consider $E^u_{\tilde{g}^n(\tilde{x})} = \pi_1^{-1}(E^u_{x_n})\cap E_{\tilde{x}}$, $E^1_{\tilde{g}^n(\tilde{x})} = \pi_1^{-1}(E^s_{x_n})\cap E_{\tilde{x}}$ and $E^s_{\tilde{g}^n(\tilde{x})} = E^1_{\tilde{x}} \oplus F_{\tilde{x}}$. For a unitary $(v,w) \in E^s_{\tilde{g}^n(\tilde{x})}$, if $v=0$, then $\|D\tilde{g}^k_{\tilde{g}^n(\tilde{x})}\cdot (0,w)\| = \theta^k$ (this will impose that $h(\Hat{\Lambda})$ shall be a $(\max\{\lambda,-\log \theta\},\epsilon)$-Pesin block for $\tilde{g}$). As if $v \neq 0$, writing $(v_k,w_k) = D\tilde{g}_{\tilde{g}^n(\tilde{x})}^k\cdot (v,w)$, then since $(v_k,w_k) \in C_\alpha$, we get
    \[
    \|w_k\| \le \|v_k\| \ = \|Df_{x_n}^k\cdot v\| \le C \exp(-\lambda k+\epsilon|n|).
    \]
    
    Thus, taking $\tilde{C}$ such that $\|(v,w)\| \le \Tilde{C} \max\{\|v\|,\|w\|\}$, we get $\|D\tilde{g}^k_{\tilde{g}^n(\tilde{x})}|_{E^u_{\tilde{g}^n(\tilde{x})}}\| \le \tilde{C} C\exp(-\lambda k+\epsilon|n|)$. Again, the same argument shows the desired result for $E^u_{\tilde{g}^n(\tilde{x})}$.

    Finally, by adding a source at infinity we may obtain a system $g: N \to N$ on some compact manifold of dimension $\dim N = m+\dim M$ for which the same results obtained for $\tilde{g}|_X$ hold for some invariant attractor in $N$.
\end{proof}

This way, the following result is a direct consequence of the properties of the constructed realization.

\begin{prop}
    A local diffeomorphism $f: M \to M$ is SPR for some potential $\phi: L_f \to \rr \cup \{-\infty\}$ if and only if its realization $g$ (given by Prop. \ref{lema nice realization})) is SPR for the potential $\phi \circ h^{-1}$ on $X$. Moreover, every consequence of the SPR property given by Theorem \ref{thm conseq SPR diffeo} for an equilibrium measure $\mu$ of $\phi \circ h^{-1}$ for $g|_X$ also hold for $h_*\mu$ which is an equilibrium measure for $\phi$.
\end{prop}

One can define the Pressure Hyperbolicity and Pressure continuity properties for an endomorphism $f$ exactly as made in \ref{PH Prop.} and \ref{PC Prop.} previously for diffeomorphisms. Since the Lyapunov exponents of the realization $g|_X$ are exactly the Lyapunov exponents of $f$ in one direction and are constant (equal to $-\theta<0$) at the other, we obtain from Theorem \ref{thm SPR iff PH and PC diffeos} and the previous Proposition that:

\begin{thm}\label{thm ph and pc imply SPR endos}
    If a local diffeomorphism $f: M \to M$ satisfies \ref{PH Prop.} and \ref{PC Prop.} for a potential $\phi:L_f \to \rr \cup \{-\infty\}$, then $f$ is SPR for $\phi$.
\end{thm}

\section{Inverse expanding on average endomorphisms}\label{sect:ieap}

We remember the definition in $\eqref{eq defn Ixvfn}$, for $f \in \End^1(\TT^2)$, $d = \deg f$. $(x,v) \in T^1\TT^2$ and $n \in \nn$, 
\[
I(x,v;f^n) = \sum\limits_{y \in f^{-n}(x)}\frac{\log\|(Df^n_y)^{-1}\cdot v\|}{d^n} = \int\limits_{\pi_{f}^{-1}(x)}\log \|D\Hat{f}^{-n}_{\Hat{x}}\cdot v\| \ dp_{x}(\Hat{x}).
\]

\begin{prop}\label{prop exp on average endos}
    We consider $f \in \End^1(\TT^2)$ with $d = \deg f>1$. The following conditions are equivalent:
    \begin{enumerate}
    \item It holds:
        \begin{equation}\label{eq hyp I(x) unif bounded}
            C(f) = \lim\limits_{n\to \infty} \frac{1}{n} \inf\limits_{(x,v) \in T^1\TT^2} I(x,v;f^n)  > 0.
        \end{equation}
    \item There are $N \in \nn$ and $C>0$ such that for every $(x,v) \in T^1\TT^2$:
        \begin{equation}\label{eq IxvfN greater than C}
            I(x,v;f^N)>C
        \end{equation}
    \item There is $C>0$ such that for every $x \in \TT^2$: 
        \begin{align}\label{eq I(x) uniformly bounded}
            I_f(x) &= \limsup\limits_{n\to \infty}\frac{1}{n} \inf\limits_{v \in T^1_x\TT^2}I(x,v;f^n) >C.
        \end{align}
    \end{enumerate}
\end{prop}

\begin{proof}
    To see that $1$ implies $2$, we take $C = C(f)-\epsilon>0$, then there is $N$ sufficiently large such that 
    \[
    \frac{1}{N} \inf\limits_{(x,v) \in T^1\TT^2} I(x,v;f^N)  > C
    \]
    
    For $2$ implies $3$, we utilize the following relation proven in Corollary $2.1$ of \cite{martin-a}. For every $(x,v) \in T^1\TT^2$, and every pair $n,m \in \nn$:
    \begin{equation}\label{eq Ixvfn as a tree}
        I(x,v;f^{nm}) = \sum\limits_{i=0}^{n-1}\sum\limits_{y \in f^{-im}(x)}\frac{1}{d^{im}}I(y,F^{-im}_yv;f^m),
    \end{equation}
    where $F^{-k}_yv = \frac{(Df^k_y)^{-1}\cdot v}{\|(Df^k_y)^{-1}\cdot v\|}$. Taking $m = N$, we obtain for every $n \in \nn$, and every $(x,v) \in T^1\TT^2$:
    \[
    \frac{1}{nN}I(x,v;f^{nN}) > \frac{1}{nN}\sum\limits_{i=0}^{n-1}\sum\limits_{y \in f^{-iN}(x)}\frac{1}{d^{iN}}C = \frac{1}{N}C.
    \]
    Hence $I_f(x) \ge \frac{C}{N}>0$ for every $x \in \TT^2$. Note that analogously we may obtain that $2$ implies $1$.

    Finally, to prove that $3$ implies $1$, for each $x \in \TT^2$, we consider $N_x \in \nn$  as the smallest natural satisfying
    \[
        \inf_{v \in T^1_x\TT^2} I(x,v;f^{N_x}) >C.
    \]
    From \eqref{eq Ixvfn as a tree}, for every $(x,v) \in T^1\TT^2$, and every $n \in \nn$, we may obtain
    
    \begin{align}\label{eq induction good trees}
        I(x,v;f^n) &= I(x,v;f^{N_x}) +\frac{1}{d^{N_x}} \sum\limits_{y \in f^{-N_x}(x)}I(y,F^{-N_x}_yv;f^{n-N_x})\nonumber\\
        &> C +\frac{1}{d^{N_x}} \sum\limits_{y \in f^{-N_x}(x)}I(y,F^{-N_x}_yv;f^{n-N_x})
    \end{align}
    Since $x \mapsto N_x$ is locally bounded, there are $m, M \in \nn$ such that $N_x \in [m,M]$ for every $x \in \TT^2$.

    We fix $n \in \nn$, and we define a good branch at time $n$, base $x$ and size $k$ as a maximal set of points (meaning that it is not contained in any other good branch at time $n$ and base $x$) $ T = \{x_0,\cdots,x_k\}$ satisfying:
    \begin{enumerate}
        \item $x_0 = x$ and $x_j \in f^{-N_{x_{j-1}}}(x_{j-1})$ for every $j = 1,\cdots, k$;
        \item Taking $N_T := \sum\limits_{j=0}^kN_{x_j}$, we have $n-M \le N_T \le n$.
    \end{enumerate}
    Note that for every good branch $T$ as defined, its size $k = k(T) \ge \lfloor\frac{n}{M}\rfloor$.
    
    We denote by $\mathcal{T}$ the set of all possible distinct good branches at time $n$ with base $x$. Continuing the process as in \eqref{eq induction good trees}, we may obtain:
    \[
    I(x,v;f^n) >2C + \frac{1}{d^{N_x}}\sum\limits_{x_1 \in f^{-N_x}(x)}\frac{1}{d^{N_{x_1}}}\sum\limits_{x_2 \in f^{-N_{x_1}}(x_1)}I(x_2,F_{x_2}^{-(N_x+N_{x_1})}v;f^{n-N_x-N_{x_2}}).
    \]
    And keeping inductively the process until we complete every possible good branch, we obtain from the bound on the size $k(T)$ of good branches
    \[
    I(x,v;f^n)>\left\lfloor\frac{n}{M}\right\rfloor C \ + \sum\limits_{T \in \mathcal{T}} \frac{1}{d^{N_T}}I(x_{k(T)},F_{x_k}^{-N_T}v,f^{n-N_T})
    \]

    We consider
    \[
    K := \inf\left\{I(y,w.f^k): 0 \le k \le M, \ (y,w) \in T^1\TT^2  \right\},
    \]
    and we suppose $K\le0$ (since $K>0$ implies item $2$). We note that $\#\mathcal{T} \le d^n$, since the case with most branches would be if $N_y=1$ for every $y$. Also since $n-M \le N_T$, $I(x_{k(T)},F_{x_k}^{-N_T}v,f^{n-N_T}) \ge K$ and $d^{-N_T} \le  d^{-(n-M)}$ for every $T \in \mathcal{T}$, hence
    \begin{equation} \label{eq bound Ixvfn using trees}
        I(x,v;f^n) \ge \left\lfloor \frac{n}{M} \right\rfloor C \ + \frac{d^n}{d^{n-M}} K = \left\lfloor \frac{n}{M} \right\rfloor C \ + d^M K.
    \end{equation}
    Since $\eqref{eq bound Ixvfn using trees}$ holds for every $(x,v) \in T^1\TT^2$ and for every $n \in \nn$, we conclude
    \[
    C(f) \ge \limsup\limits_{n \in \nn} \frac{1}{n} \left\lfloor \frac{n}{M} \right\rfloor C \ + \frac{d^M}{n} K = \frac{C}{M}>0.
    \]
\end{proof}

\begin{defn}
    We say that $f \in \End^1(\TT^2)$ is inverse expanding on average if one, and hence all the three equivalent conditions in Proposition \ref{prop exp on average endos} hold.
\end{defn}

\begin{rmk}
    From the characterization given by \eqref{eq IxvfN greater than C}, one can note that the inverse expanding on average property is $\mathcal{C}^1$-open in the spaces of endomorphisms of $\TT^2$. 
\end{rmk}

The inverse expanding on average property yields a uniform bound on the Lyapunov exponent of every equidistributed invariant measure. Even more, it gives the existence of the unstable direction $E^u(\hat x)$ for most points $\hat x \in \pi_{f}^{-1}(x)$ and every $x \in \TT^2$, independently of any invariant measure.

\begin{lema}\label{lema weakbundle}
    For $f \in \End^1(\TT^2)$, suppose that $C(f)>0$ and $\inf_{x \in \TT^2}|\det Df_x|>1$. There exist $\lambda>0$ such that every $x \in \TT^2$ satisfies that for $p_x$-almost every $\Hat{x} \in \pi_{f}^{-1}(x)$, there exists a one dimensional direction $E^u(\Hat{x}) \subset \rr^2$ verifying that:
    \begin{enumerate}
        \item If $v \in E^u(\Hat{x})$ then $\limsup\limits_{n\to \infty}\frac{1}{n}\log \|D\Hat{f}^{-n}_{\Hat{x}}\cdot v\| \le -C(f)$;
        \item If $v \in \rr^2 \setminus E^u(\Hat{x})$, then $\liminf\limits_{n \to \infty}\frac{1}{n}\log \|D\Hat{f}^{-n}_{\Hat{x}}\cdot v\| \ge \lambda$.
    \end{enumerate}
\end{lema}
\begin{proof}
    The property \eqref{eq IxvfN greater than C} is also called expanding on average property by DeWitt and Dolgopyat in \cite{dewitt2024expanding}, in the context of random diffeomorphisms. We may repeat verbatim the technique developed in Section 4 of their work to obtain temperedness of the sequence of norms $\|D\Hat{f}^{-n}_{\Hat{x}}\|$ for $p_x$-almost every point in $\pi_{f}^{-1}(x)$, see Propositions 4.6 and 4.7 in \cite{dewitt2024expanding}. This yields item (1).
    
    Item (2) is a direct consequence of item (1) with the hypothesis that $\inf_x |\det Df_x| >1$.
\end{proof}

The following Corollary is a direct consequence of the disintegration of the lift of equidistributed measures to $L_f$ \eqref{eq equidistributed meas} along with the previous lemma.

\begin{corol}\label{corol hyperbolicity of measures max fold}
    Let $f \in \End^1(\TT^2)$. If $C(f)>0$ and $\inf_x |\det Df_x|>1$, then there exists $\lambda>0$ such that every equidistributed $f$-invariant measure is $\lambda$-hyperbolic, meaning that $\chi^-(x) \le -\lambda<\lambda\le\chi^+(x)$ for $\mu$-a.e. $x \in \TT^2$.
\end{corol}

\begin{rmk}
    As in Proposition 2.2 in \cite{martin-a}, the previous Corollary may be obtained through a direct calculation without the need of Lemma \ref{lema weakbundle}. The importance of such Lemma lies only in the existence of unstable direction for every point.
\end{rmk}

\subsection{Size of unstable manifolds}

Our main goal in this subsection is to obtain uniform bounds for the sizes of unstable manifolds.

\begin{defn}
    A curve $\sigma: [-1,1] \to \TT^2$ is a $L$-Lipschitz graph centered at $x=\sigma(0)$ if writing $T_x\TT^2 = \rr\sigma'(0)\oplus \rr \sigma'(0)^\perp$, there exists a $L$-Lipshitz function $\varphi: J \subset \rr \to \rr$ such that $\sigma([-1,1]) = \exp_x\{(t,\varphi(t)): t \in J\}$.

    We denote the length of a curve $\sigma$ as $|\sigma|$.
\end{defn}

 We have the following Proposition, the point here is that every point on $\TT^2$ has many large unstable manifolds, which allows us to study such objects regardless of the existence of invariant measures.
 
\begin{prop}\label{prop large unstable manifolds}
    Suppose that $f \in \End^2(\TT^2)$ satisfies $C(f)>0$ and $\inf_{x\in M} |\det Df_x| > 1$. Then, given $\delta >0$, there exists $\ell_0>0$ and a $\mathcal{C}^2$-neighborhood $\mathcal{U}_1$ of $f$ such that for every $g \in \mathcal{U}$ and every $x \in \TT^2$:
    \[
    p_x\{\Hat{x} \in \pi_{f}^{-1}(x): W^u_g(\Hat{x}) \ \text{is a} \ 1\text{-Lipschitz graph with} \ |W^u_g(\Hat{x})|>\ell_0\}> 1- \delta,
    \]
    where $W^u_g(\Hat{x})$ is the unstable manifold at $\Hat{x} \in L_f$ of the corresponding lift of $g$ to $L_f$. 
\end{prop}

\begin{rmk}
    The tools developed in this section are mainly utilized in order to obtain uniqueness on a neighborhood of maps $f \in \mathcal{F}_1$ in Theorems \ref{thm main} and \ref{main thm uniqueness}. We remark that this property of large unstable manifolds given by Prop. \ref{prop large unstable manifolds} is the only one that imposes the neighborhood we obtain to be $\mathcal{C}^2$ and not $\mathcal{C}^1$.
\end{rmk}

This is a version of [Theorem 7.2, \cite{lima2024measuresmaximalentropynonuniformly}, where Lima-Obata-Poletti proved such affirmation in a similar context (see also Dolgopyat-Krikorian \cite{DK07}, Zhang \cite{zh19}, Dolgopyat-DeWitt \cite{dewitt2024expanding} in the context of random dynamics).

We point out that their result and ours have a significant difference which is crucial in our context: we need that, for every $x \in \TT^2$, there are large unstable manifolds on a set of uniformly positive $p_x$-measure in $\pi_{f}^{-1}(x)$, while the \cite{lima2024measuresmaximalentropynonuniformly} result only guarantee such condition for typical points $x \in \TT^2$ of the preserved measure.

This restriction only appears in their work in order to guarantee the existence of the unstable directions given by Oseledets Theorem, however we have the $E^u$ direction given by Lemma \ref{lema weakbundle}, and this is enough to get the results we need. For that reason we may argue exactly as they do and we expose only an idea of their argument. A direct adaptation of [Theorem 7.1, \cite{lima2024measuresmaximalentropynonuniformly}] is the following Lemma.

\begin{lema}\label{lema unst direction distributed on fibres}
Suppose that $f \in \End^1(\TT^2)$ satisfies $C(f)>0$ and $\inf_{x\in M} |\det Df_x| > 1$. There exist constants $A, \beta>0$, and a $\mathcal C^1$ neighborhood $\mathcal U_0$ of $f$ such that for any $g\in\mathcal U_0$, $x\in\TT^2$, $E \in \mathbb{P}T_x\TT^2$, and any $\gamma>0$,
    \begin{equation}\label{eq unst direction distributed on fibres} 
        p_x\{\Hat{x} \in \pi_{g}^{-1}(x): \angle (E, E^u_g({\Hat{x}}))<\gamma\}\le A \gamma^\beta.
    \end{equation}
\end{lema}

Such lemma states that condition \eqref{eq hyp I(x) unif bounded} implies that for every $x \in \TT^2$, the unstable directions $\{E^u_g(\Hat{x}): \Hat{x} \in \pi_{f}^{-1}(x)\}$ given by Lemma \ref{lema weakbundle} are varying on points $\Hat{x} \in \pi_{f}^{-1}(x)$. Informally, if the unstable directions of points on a set of positive $p_x$-measure were concentrated on some direction $E \in \PP T_x\TT^2$, we could obtain a point $z \in g^{-n}(x)$ such that the averaged Lyapunov exponent $I(z)$ would necessarily be negative since for $v \in D\Hat{g}^{-n}_{\Hat{x}}\cdot E$ there are many points $\Hat{z} \in \pi_{f}^{-1}(z)$ (in the sense of the $p_z$ measure) that are contracting the vector $v$ when iterating backwards, as it is the unstable direction of such points.

\begin{proof}[Proof of Prop. \ref{prop large unstable manifolds}]
We fix a constant $0<\chi <C(f)$ and take $\mathcal{U}_1\subset\mathcal U_0$ a $\mathcal{C}^2$-neighborhood of $f$ such that for every $g \in \mathcal{U}_1$, the number $C(g)$ given by \eqref{eq hyp I(x) unif bounded} satisfies $C(g) > \chi$. We fix $g \in \mathcal{U}_1$. We follow the proof of [Theorem 7.2 \cite{lima2024measuresmaximalentropynonuniformly}].

They begin by defining a fibered version of a Pesin set, namely for a suitable $\lambda>0$, which does not depend on $g \in \mathcal{U}$ but only on $\chi$ (see Lemmas 7.3 and 7.5 in \cite{lima2024measuresmaximalentropynonuniformly}), for $C,\epsilon>0$ and $E \in \mathbb{P}T_x\TT^2$, let $\Lambda^-_{\epsilon,C,E,g}(x)$ be the set of points $\Hat{x} \in \pi_{f}^{-1}(x)$ satisfying for every $n,k \ge 0$:
\begin{itemize}
    \item $\|D\Hat{g}^{-n}_{\Hat{f}^{-k}(\Hat{x})}|_{E^u_g(\Hat{g}^{-k}(\Hat{x}))}\| \le C e^{-n\lambda+k\epsilon}$;
    \item $\|D\Hat{g}^{-n}_{\Hat{f}^{-k}(\Hat{x})}|_{E_k(\Hat{x})}\| \ge C^{-1} e^{n\lambda-k\epsilon}$, where $E_k(\Hat{x}) = D\Hat{g}^{-k}_{\Hat{x}}\cdot E$;
    \item The angle $\angle(E_k(\Hat{x}),E^u_g(\Hat{g}^{-k}(\Hat{x}))) \ge C^{-1}e^{-k\epsilon}$.
\end{itemize}

They follow by showing that condition \eqref{eq hyp I(x) unif bounded} and Lemma \ref{lema unst direction distributed on fibres} are enough to obtain that for every $\delta>0$, there are $C,\epsilon>0$ such that for any $x \in \TT^2$ it holds that $p_x(\Lambda^-_{\epsilon,C,E,g}(x))>1-\delta$, for every $E  \in \mathbb{P}T_x\TT^2$. This fact allows the use of graph transform methods [\cite{pesin1976families}, Theorem 2.1.1] to obtain that every $\Hat{x} \in \Lambda^-_{\epsilon,C,E,g}(x)$ has a local unstable manifold that is a $1$-Lipschitz graph. Its size depends only on $\epsilon,C,E$ and on the Hölder regularity of $Dg$ which is controlled in $\mathcal{U}_1$.
\end{proof}

\subsection{Size of stable manifolds}

In this subsection we will show that we also have a uniform bound for the sizes of stable manifolds in the following sense. 

\begin{prop}\label{prop large stable manifolds}
    Suppose that $f \in \End^{1+}(\TT^2)$ satisfies $C(f)>0$ and $\inf_{x} |\det Df_x|>1$. Given $\ell_0>0$ sufficiently small (depending only on $f$), there exists a $\mathcal{C}^1$-open neighborhood $\mathcal{U}_2 \subset \End^{1+}(\TT^2)$ of $f$ such that for every $g \in \mathcal{U}$ and for every equidistributed $g$-invariant probability measure $\mu$:
    \[
    \mu(\{x \in \TT^2: W^s_g(x)\  \text{is a $1$-Lipschitz graph with }|W^s_g(x)|>\ell_0\})>0. 
    \]
\end{prop}

We begin by the following lemma.

\begin{lema}\label{lema cones UEA}
    Suppose that $f \in \End^{1+}(\TT^2)$ satisfies $C(f)> 0$ and $\inf_x|\det Df_x|>1$. There exist constants $\epsilon,R>0$, $N \in \nn$, $\lambda>1$ and a $\mathcal{C}^1$-open neighborhood $\mathcal{U}_2$ of $f$ such that every $g \in \mathcal{U}_2$ satisfies the following.
    
    For every $x \in \TT^2$, and any unitary $v \in \rr^2$, there is a cone $C(v,\epsilon) = \{(w_1,w_2) \in \rr v \oplus \rr v^\perp : |w_2| \le \epsilon|w_1|\}$ and some $y \in g^{-N}(x)$ such that, for every $z \in B(y,R)$ it holds:
    \begin{itemize}
        \item $(Dg^N_z)^{-1}C(v,\epsilon) \subset C(v_y,\epsilon)$, where $v_y = \frac{(Dg_{y}^N)^{-1}v}{\|(Dg_{y}^N)^{-1}v\|}$ and
        \item For every $w \in C(v,\epsilon)$, $\|(Dg_z^N)^{-1}w\| \ge \lambda\|w\|$.
    \end{itemize}
\end{lema}

\begin{proof}
    Since $C(f)>0$, from item 2 of Prop. \ref{prop exp on average endos}, there are $N \in \nn$ and $C>0$ such that for every $(x,v) \in T^1\TT^2$, 
    $I(x,v.f^N)>C$. It is clear then that for every $(x,v) \in T^1\TT^2$, there exists at least one point $y \in f^{-N}(x)$ with 
    \[
    \|(Df_y^N)^{-1}\cdot v\| > \lambda, \ \ \lambda=e^C>1.
    \]

    We write $ v_y= \frac{(Dg_{y}^N)^{-1}v}{\|(Dg_{y}^N)^{-1}v\|}$ and consider $(Df_y^N)^{-1}: \rr v \oplus \rr v^\perp \to \rr v_y \oplus \rr v_y^\perp$ as
    \[
    (Df_y^N)^{-1}=\begin{pmatrix}
        a &b \\
        0 &c
    \end{pmatrix}.
    \]
    Taking $K = \min_x|\det Df^N_x|>1$, we have $|ac| = \frac{1}{K}<1$, also $\|(Df_y^N)^{-1}\cdot v\| = |a| >\lambda>1$ and $|b|<\|Df^N\|_\infty$.

    We consider $0<\epsilon<\min\left\{ \frac{\lambda-1}{\|Df\|_\infty}, \frac{\lambda^2K-1}{K\|Df\|_\infty^2 } \right\}<\min\left\{\frac{|a|-1}{|b|},\frac{a^2K-1}{abK}  \right\}$ and the cone $C(v,\epsilon)$ as enunciated. For every $w=(w_1,w_2) \in C(v,\epsilon)$, $(Df_y^N)^{-1}\cdot w= (aw_1+bw_2,cw_2)$ and from the previous observations we have 
    $|cw_2|\le\frac{\epsilon}{|a| K}|w_1|$ and $|aw_1+bw_2|\ge |aw_1|-|bw_2| \ge\left(|a|-\epsilon |b|\right)|w_1|$. Hence
    \[
    \frac{|cw_2|}{|aw_1+bw_2|} \le \frac{\epsilon}{aK(a-\epsilon b)}<\epsilon.
    \]
    Hence $(Df_y^{N})^{-1} C(v,\epsilon) \subset C(v_y, \epsilon)$.
    
    From the uniform continuity of $Df^N$, we may conclude that for uniform constants $1<\Tilde{\lambda}\le\lambda$ and $R>0$, it holds that for any $(x,v) \in T^1\TT^2$ and for $y \in  f^{-N}(x)$ as before, it holds for every $z \in B(y,R)$ that
    \begin{itemize}
        \item $(Df_z^N)^{-1}C(v,\epsilon) \subset C(v_y,\epsilon)$, and
        \item For every $w \in C(v,\epsilon)$, $\|(Df_z^N)^{-1}w\| > \Tilde{\lambda}\|w\|$.
    \end{itemize}
    
    Since the results obtained are $\mathcal{C}^1$-open properties, we conclude that the same holds for every $g$ in a $\mathcal{C}^1$-open neighborhood $\mathcal{U}_2$ of $f$.
\end{proof}

\begin{corol}\label{corol expanding pre-images of curves}
    Given $f \in \End^{1+}(\TT^2)$ with $C(f)>0$ and $\inf_x|\det Df_x|>1$, let $R>0$ and $\mathcal{U}_2$ be, respectively, the radius and the $\mathcal{C}^1$ neighborhood of $f$ given by Lemma \ref{lema cones UEA}. Then, for every $g \in \mathcal{U}_2$ and every $\mathcal{C}^1$ curve $\sigma:[-1,1] \to \TT^2$, there is some $n \in \nn$ and other curve $\gamma:[-1,1] \to \TT^2$ such that $g^N(\gamma) \subset \sigma$ and $\gamma$ is a $1$-Lipschitz graph with $|\gamma| >R/2$.
\end{corol}

\begin{proof}
    Let $g \in \mathcal{U}_2$ and $\sigma$ be given and consider $x_0 = \sigma(0)$ and $v_0 = \sigma'(0)$. Let $x_1 \in g^{-N}(x_0)$ be at pre-image of $x_0$ that satisfies the conclusion of Lemma \ref{lema cones UEA} and consider $v_1 = \frac{(Dg_{x_1}^N)^{-1}v_0}{\|(Dg_{x_1}^N)^{-1}v_0\|}$. Inductively consider for $k \in \nn$, $(x_k,v_k) \in T^1\TT^2$ constructed from the previous one as indicated for $(x_1,v_1)$.

    Consider also the uniform $\epsilon>0$ given by Lemma \ref{lema cones UEA}. Since $\sigma$ is $\mathcal{C}^1$, there is some piece $\Tilde{\sigma}:[-a,a] \to \TT^2$ which is an $\epsilon$-Lipschitz graph and for which $\Tilde{\sigma}'(t) \in C(v_0,\epsilon)$ for all $t \in [-a,a]$. The conclusions of Lemma \ref{lema cones UEA} allow us to apply graph transform methods [Theorem 2.1.1, \cite{pesin1976families}] to $\Tilde{\sigma}$ to obtain eventually some pre-image $\gamma:[-1,1] \to \TT^2$ with $\gamma(0)= x_k$, $f^{kN}(\gamma) \subset \Tilde{\sigma}$ such that $\gamma$ is a $1$-Lipschitz graph with $|\gamma|>R/2$.
\end{proof}

We are now ready to prove Proposition \ref{prop large stable manifolds}.

\begin{proof}[Proof of Prop. \ref{prop large stable manifolds}]
    Let $\mathcal{U}_2$ and $R>0$ be, respectively the $\mathcal{C}^1$-neighborhood of $f$ and the radius given by Lemma \ref{lema cones UEA} and fix $\ell_0 = R/2$. For $g \in \mathcal{U}_2$, consider the set 
    \[
    A= \{x \in \TT^2: W^s_g(x) \ \text{is a 1-Lipschitz graph with } |W^s_g(x)|>\ell_0\}.
    \]
    From Corollary \ref{corol expanding pre-images of curves}, for every Lyapunov regular $x \in  \TT^2$ (which have a well defined stable manifold), there exist some $n \in \nn$ such that $g^{-n}(x) \cap A \neq \emptyset$. Hence, for every $g$-invariant probability measure $\mu$, we have
    \begin{equation}\label{eq mu gnA}
    \mu \left( \bigcup_{n\ge0} g^n(A) \right) >0.
    \end{equation}
    Now, if $\mu$ is equidistributed it is, in particular, a non-singular measure, meaning that for every measurable set $E\subset \TT^2$, $\mu(E) >0$ if and only if $\mu(g^n(E))>0$ for every $n \ge 0$. This, along with \eqref{eq mu gnA} implies that $\mu(A)>0$ for every equidistributed $g$-invariant measure. 
\end{proof}

\subsection{The conservative case}

We recall the families
\[
\mathcal F_0=\{f\in\text{End}^2(\TT^2):\inf_x I_f(x)>0, \ \text{and}  \det(Df_x)=\deg(f)\}
\]
of conservative expanding on average endomorphisms with constant Jacobian, and
\[
\mathcal F_1=\{f\in\mathcal F_0:\ \pm 1 \text{ is not an eigenvalue of the linear part of } f\}
\]
for which from Theorem \ref{thm previous results for F_1} every $f \in \mathcal{F}_1$ is stably ergodic. In particular, from Subsection \ref{subsec equidistritued meas}, the unique $\Hat{f}$-invariant lift of the Lebesgue measure to $L_f$ is equidistributed, thus coincide with the reference measure $\Hat{\eta}$ given by the relations \eqref{eq defn measure on Sol local leb times bernoulli} and \eqref{eq fubini on reference measure}.

\begin{prop}\label{prop dense preimages of curves}
    Consider $f \in \mathcal{F}_1$. For any $\mathcal{C}^1$ curve $\sigma:[-1,1] \to \TT^2$ the set $\bigcup\limits_{n \ge 0} f^{-n}(\sigma)$ is dense in $\TT^2$.
\end{prop}

\begin{proof}
    We prove by contradiction. Suppose then that there is a $\mathcal{C}^1$-curve $\sigma:[-1.1] \to \TT^2$ and an open set $U \subset M$ such that $U \cap \bigcup_{n\ge 0} f^{-n}(\sigma) = \emptyset$. Fix some $\delta>0$ small and let $\ell_0>0$ be given by Prop. \ref{prop large unstable manifolds}. We consider the following set in $L_f$:
    \[ 
        Sat^u_{\ell_o}(\sigma) = \bigcup\limits_{\Hat{x} \in \Hat{L}}W^u(\Hat{x}),
    \]
    where $\Hat{L} = \{\Hat{x} \in \pi_{f}^{-1}(\sigma): W^u(\Hat{x}) \ \text{is a} \ 1\text{-Lipschitz graph with} \ |W^u(\Hat{x})|>\ell_0\}$. We claim that $\Hat{\eta}(Sat^u_{\ell_0}(\sigma)) >0$. 
    
    Indeed, we consider $V$ a simply connected open neighborhood of $\sigma$ with diameter greater than $\ell_0$, and we identify $\pi_{f}^{-1}(V) = V \times \Sigma$ through a homeomorphism $\Xi: V \times \Sigma \to \pi_{f}^{-1}(V)$ as made in Subsection \ref{subsec Solenoid} and write $V_{\omega} = \Xi(V,\omega)$, we may cut the curve $\sigma$ if necessary. We consider $\Leb_{\sigma}$ the Lebesgue measure on $\sigma$ inherited from $\TT^2$ as a submanifold. Then we consider $\Hat{\eta}_\sigma$ on $L_f$ as the unique equidistributed lift of $\Leb_{\sigma}$ as in \eqref{eq defn Leb times bernoulli on curves}.
    
    From Prop. \ref{prop large unstable manifolds}, $p_x(\Hat{L})>1-\delta$ for every $x \in \sigma$, thus $\Hat{\eta}_\sigma(\Hat{L})>0$. Hence, considering for each $\omega \in \Sigma$ the set $L_\omega= \pi_{f}(\Hat{L} \cap V_\omega)$, we obtain $\Leb_{\sigma}(L_{\omega})>0$ for every $\omega \in B \subset \Sigma$ with $p(B)>0$, where $p$ is the equally distributed Bernoulli measure on $\Sigma$. Finally, from Lemma \eqref{lema unst direction distributed on fibres}, we may obtain $\Tilde{B} \subset B$ with $p(\Tilde{B})>0$ such that $\sigma$ is transversal to $\Hat{W}^u_{\omega} = \{W^u(\Hat{x}): \Hat{x} \in V_\omega\}$ for every $\omega \in \Tilde{B}$. We also take $S_\omega = \pi_{f}(Sat^u_{\ell_0}(\sigma) \cap V_\omega)$. From Lemma \ref{lema abs cont unst holon}, $Leb(S_\omega)>0$, for every $\omega \in \Tilde{B}$ which from \eqref{eq fubini on reference measure} implies that $\Hat{\eta}(Sat^u_{\ell_0}(\sigma))>0$.

    By considering a Pesin block $\Hat{\Lambda} \subset L_f$ with $\Hat{\eta}(\Hat{\Lambda} \cap Sat^u_{\ell_0}(\sigma))>0$, there are constants $C>0$ and $0<\lambda<1$ such that for every $\Hat{z} = (z_k)_k \in \Hat{\Lambda} \cap Sat^u_{\ell_0}(\sigma)$ 
    \[
    d(z_{-n}, y_{-n}) \le C \lambda^{-n}\ell_0, \ \text{for every} \ \Hat{y} = (y_k)_k \in W^u(\Hat{z}), \ \text{and every} \ n \ge0. 
    \]
    Hence by taking $\ell_0$ sufficiently small and, if necessary, considering $\Tilde{U} \subset U$ a smaller open set, we have that $\pi_{f}^{-1}(\Tilde{U}) \cap \bigcup_{n\ge 0} \Hat{f}^{-n}(\Hat{\Lambda} \cap Sat^u_{\ell_0}(\sigma)) = \emptyset$. Since $\Hat{\eta}(\pi_{f}^{-1}(\Tilde{U}))>0$, we obtain that $\Hat{\eta}\left(\bigcup_{n \ge 0} \Hat{f}^{-n}(\Hat{\Lambda} \cap Sat^u_{\ell_0}(\sigma))\right) < 1$ contradicting the fact that $\Hat{\eta}$ is ergodic.
\end{proof}

\begin{corol}\label{corol dense preimages of curves}
    Let $f \in \mathcal{F}_1$ be given. Then, given constants $\epsilon,\ell>0$, there exists a $\mathcal{C}^1$-open neighborhood $\mathcal{U}_3 \subset End^{1}(\TT^2)$ and $N \in \nn$ such that every $g \in \mathcal{U}_3$ satisfies the following. For every $\mathcal{C}^{1}$ curve $\sigma:[-1,1] \to \TT^2$ which is a $1$-Lipschitz graph with $|\sigma|\ge \ell$, the union $\bigcup_{n= 0}^Ng^{-n}(\sigma)$ is $\epsilon$-dense in $\TT^2$.
\end{corol}
\begin{proof}
    
    We fix $\epsilon,\ell>0$ and consider the set
    \[
    \Gamma = \{\sigma: [-1,1] \to \TT^2: \sigma \ \text{is a 1-Lipschitz graph with }\ell \le |\sigma|\le \ell+1 \}.
    \]
    From Arzelà-Ascoli's Theorem this space is compact in the compact-open topology. 

    For $\sigma \in \Gamma$, let $N(\sigma) = \min\{n \in \nn: \bigcup_{j=0}^{n}f^{-j}(\sigma) \ \text{is} \ \epsilon/2\text{-dense}\}$, from Prop. \ref{prop dense preimages of curves} the function is well defined. It is a continuous function, hence we may take $N=\max\{N(\sigma): \sigma \in \Gamma\}$. We take the $\mathcal{C}^1$-neighborhood $\mathcal{U}_3$ of $f$ given by the maps $g \in \End^1(\TT^2)$ satisfying for every $x \in M$ and every $k=1,\cdots,N$
    \[
    d(f^{-k}(x),g^{-k}(x)) < \epsilon/2.
    \]
    
    This way, for each curve $\sigma \in \Gamma$ and $x \in \TT^2$, there is $k \le N$ such that $d(x,f^{-k}(\sigma)) <\epsilon/2$, hence $d(x,g^{-k}(\sigma)) < \epsilon$. In particular this implies that for every $g \in \mathcal{U}$, and every $\sigma \in \Gamma$, $\bigcup_{j=0}^{N}g^{-j}(\sigma)$ is $\epsilon$-dense. 
\end{proof}

\section{Other tools: F{\o}lner sequences, Yomdin Theory}\label{sect:tools}

In this section, we present the tools given by D. Burguet in \cite{burguet2024srb} in order to prove existence of SRB measures and we show how to extend them to the non-invertible scenario.

\subsection{F{\o}lner sequences}

Let us start by introducing some necessary concepts on F{\o}lner sequences, they are of extreme importance on Burguet's argument since the geometric times will be expressed by a suitable F{\o}lner sequence on which we will consider the iterates of a reference measure. We give a brief review on the main concepts, for details see Section 2 on \cite{burguet2024srb}.

We let $\mathcal{P}_{\nn}$ and $\mathcal{P}_n$ be the power sets of $\nn$ and $\{0,1,2, \cdots, n-1\}$ respectively. For $E \in \mathcal{P}_{\nn}$, its boundary $\partial E$ is given by the integers $n \in E$ such that either $n+1 \notin E$ or $n-1 \notin E$. Given $E \in \mathcal{P}_\nn$ and $n \in \nn$, we consider 
\[
d_n(E) := \frac{\# E \cap [0,n-1]}{n},
\]
then the upper and lower asymptotic densities of $E$ are given respectively by
\[
\ol{d}(E) := \limsup\limits_{n \to \infty}d_n(E), \ \ \underline{d}(E) := \liminf\limits_{n \to \infty} d_n(E).
\]

We let $\mathcal{N} \subset \mathcal{P}_{\mathbb{N}}$ be the set of infinite subsets of $\mathbb{N}$. For a sequence $\mathcal{n} \in \mathcal{N}$ and some $E \in \mathcal{P}_{\mathbb{N}}$ we may also define the upper and lower asymptotic densities of $E$ along a subsequence $\mathcal{n}$ by $\ol{d}^{\mathcal{n}}(E) := \limsup_{n \in \mathcal{n}}d_n(E)$ and  $\underline{d}^{\mathcal{n}}(E): = \liminf_{n \in \mathcal{n}}d_n(E)$.

Finally, we may consider for $\mathcal{n} \in \mathcal{N}$, the generalized power set of $\mathcal{n}$ as $\mathcal{Q_n}:= \prod_{n \in \mathcal{n}} \mathcal{P}_n$. For $\mathcal{F} = (F_n)_{n \in \mathcal{n}} \in \mathcal{Q_n}$, we let
\[
\ol{d}^{\mathcal{n}}(\mathcal{F}) = \limsup\limits_{n \in \mathcal{n}} d_n(F_n), \ \ \underline{d}^{\mathcal{n}}(\mathcal{F}) = \liminf\limits_{n \in \mathcal{n}}d_n(F_n).
\]
To any of the asymptotic densities, we write it as $d^{\mathcal{n}}$ when the corresponding $d_n$ are converging.

We say that $F \in \mathcal{P}_{\nn}$ is F{\o}lner along a subsequence $\mathcal{n} \in \mathcal{N}$ if its boundary $\partial F$ has zero upper asymptotic density along $\mathcal{n}$. More generally, $\mathcal{F} = (F_n)_{n \in \mathcal{n}} \in \mathcal{Q_n}$ is F{\o}lner when $d^{\mathcal{n}}(\partial \mathcal{F}) = 0$, where $\partial\mathcal{F} = (\partial F_n)_{n \in \mathcal{n}}$. In our context, we will only work with F{\o}lner sequences $\mathcal{F}$ with positive lower asymptotic density and in this scenario this property heuristically means that asymptotically the sequence has only few boundary elements. This yields the following.

\begin{lema}\label{lema limit of measures along folner sequences}[Lemma 3,\cite{burguet2024srb}]
Let $X$ be a compact metric space, $T:X \to X$ a continuous map and $\mathcal{F} = (F_n)_{n \in \mathcal{n}} \in \mathcal{Q_n}$ for some $\mathcal{n} \in \mathcal{N}$ a F{\o}lner sequence with positive lower asymptotic density $\underline{d}^{\mathcal{n}}(\mathcal F)>0$. Then for any family $\{\mu_n\}_{n \in \mathcal{n}}$ of Borel probability measures on $X$, any weak$^*$-limit of
\[
\mu_{n}^{F_{n}} := \frac{1}{\#F_{n}}\sum\limits_{j \in F_{n}} T^j_* \mu_{n}
\]
is a $T$-invariant Borel probability measure.
\end{lema}

Given $E \in \mathcal{P}_\nn$ and $N \in \nn$, the $N$-filled sequence related to $E$ is the sequence $\text{Fill}^N(E)$ given by the union over all interval of integers $[a,b) \cap \nn$ with $a,b \in E$ such that $0<b-a\le N$. The sequence $\text{Fill}^N(E)$ is hence constructed by filling all gaps of $E$ with size at most $N$. Given $E \in \mathcal{P}_\nn$, a subsequence $\mathcal{n} \in \mathcal{N}$ and some $\mathcal{F} = (F_n)_{n \in \mathcal{n}} \in \mathcal{Q_n}$, we say that $E$ is dense in $\mathcal{F}$ if
\[
d^{\mathcal{n}} (\mathcal{F}\setminus \text{Fill}^N(E))  \longrightarrow 0, \ \text{as} \ N \to \infty,
\]
where $\mathcal{F}\setminus E := (F_n\setminus E )_{n \in \mathcal{n}} \in \mathcal{Q_n}$.

At last, let $(X,\mathcal{B},\eta)$ be a probability space. A map $E: X \to \mathcal{P}_\nn$ is said to be measurable when for every $n \in \nn$ the set $\{x \in X: n \in E(x)\}$ is measurable. The following lemma is of extreme importance in the construction we aim to make, it will allow us to obtain sets with subexponential decay of measure where every point has its geometrical times on a a unique good F{\o}lner sequence. Its proof is based on a Borel-Cantelli argument, for details see Section 2.4 on \cite{burguet2024srb}.

\begin{lema}\label{lema folner sequence}[Lemma 2, \cite{burguet2024srb}]
    Assume that $E: X \to \mathcal{P}_\nn$ is measurable and has uniformly positive upper asymptotic density on a set of positive measure, i.e. there exists $\beta>0$ such that
    \[
    \ol{d}(E(x)) > \beta >0, \ \text{for every} \ x \in A \ \text{with} \ \mu(A)>0.
    \]
    Then, there exists $\mathcal{n} \in \mathcal{N}$, a sequence $\mathcal F = (F_n)_{n \in \mathcal{n }} \in \mathcal{Q_n}$ and $(A_n)_{n \in \mathcal{n}}$ measurable subsets of $A$, such that:
    \begin{enumerate}
        \item $\mathcal F$ is F{\o}lner and, for all $n \in \mathcal{n}$, $\partial F_{n} \subset E(x)$ for every $x \in A_{n}$.
        \item $\eta(A_{n}) \ge \frac{1}{n^2}e^{-n\delta_n}$ for every $n \in \mathcal{n}$, with $\delta_n \to 0$ as $\mathcal{n} \ni n \to +\infty$.
        \item $\mathcal F$ has positive lower asymptotic density. More specifically
        \[
        \underline{d}^{\mathcal{n}}(\mathcal F) \ge \liminf\limits_{n \in \mathcal{n}} \inf\limits_{x \in A_{n}} d_n(E(x)\cap F_n) \ge \beta.
        \]
        \item $E$ is dense in $\mathcal F$ uniformly in $A_{n}$, that is,
        \[
        \limsup\limits_{n \in \mathcal{n}} \sup\limits_{x \in A_{n}} d_n(F_n \setminus \text{Fill}^N(E(x)))  \longrightarrow 0, \ \text{as} \ N \to \infty.
        \]
    \end{enumerate}
\end{lema}

\subsection{Empirical measures associated to F{\o}lner sequences}

We consider an additive cocycle $\psi^n = \sum\limits_{j=0}^{n-1}\psi \circ T^{j}$ associated to a continuous function $\psi: X \to \rr$. We also consider a measurable $E: Y \to \mathcal{P}_\nn$ with positive asymptotic density, for $Y$ a Borel subset of $X$, and let $\mathcal{n} \in \mathcal{N}$, $\mathcal F = (F_n)_{n \in \mathcal{n}} \in \mathcal{Q_n}$ a F{\o}lner sequence and $(A_{n})_{n \in \mathcal{n}}$ subsets of $Y$ which are associated to $E$ given by Lemma \ref{lema folner sequence}.

\begin{defn}\label{defn folner gibbs property}(F{\o}lner-Gibbs property) 
    The measure $\eta$ is said to satisfy the F{\o}lner-Gibbs property with respect to the additive cocycle $(\psi^n)_n$ and the F{\o}lner sequence $\mathcal{F} = (F_n)_{n \in \mathcal{n}}$ when there exist $\epsilon>0$ such that given any partition $\mathcal{P}$ of $X$ with diameter less than $\epsilon$, there exists $N \in \nn$ such that for every integer $\mathcal{n} \ni n>N$ and every $x \in A_{n}$
    \begin{equation}
    \frac{1}{\eta(P^{F_{n}}(x) \cap A_{n})} \ge \exp{\bigg( \sum\limits_{j \in F_{n}}\psi \circ T^j(x) \bigg)}.
    \end{equation}
\end{defn}

We consider the probability measures $(\mu_{n})_{n \in \mathcal{n}}$ on $X$ given by
\[
\mu_{n}(B) = \frac{\eta(A_{n} \cap B)}{\eta(A_{n})}, \ \text{for every Borel} \ B \subset X.
\]
We are interested on studying limit measures of the sequence
\[
\mu_{n}^{F_{n}} = \frac{1}{\#F_{n}}\sum\limits_{j \in F_{n}} T^{j}_*\mu_{n}
\]
The F{\o}lner-Gibbs property is a version of the classical Gibbs property restricted to a F{\o}lner sequence. Burguet shows that this is sufficient to relate the entropy of limit measures associated to $\mu_{n}^{F_n}$ with the associated cocycle as follows.

\begin{prop} \label{prop folner gibbs property}[Prop. 3, \cite{burguet2024srb}]
    If $\eta$ satisfies the F{\o}lner-Gibbs property with respect to the additive cocycle $(\psi^n)_n$ and the F{\o}lner sequence $\mathcal F$, then any weak$^*$ limit $\mu$ of the sequence $(\mu_{n}^{F_{n}})_{n \in \mathcal{n}}$  satisfies
    \[
    h_\mu(T) \ge \int \psi d\mu.
    \]
\end{prop}

At last, the following Lemma will allows us to control the Lyapunov exponents of such limit measures.

\begin{lema}\label{lema prop of large sets}[Lemma 4, \cite{burguet2024srb}]
    Suppose that there is $a>0$ such that for any $x \in X$ it holds $\psi^{m-k}(T^k(x)) > (m-k)a$ for all integers $m>k \in E(x)$. Then, for any weak$^*$ limit $\mu$ of the sequence $(\mu_{n}^{F_{n}})_{n \in \mathcal{n}}$ satisfies
    \[
    \Tilde{\psi}(x) := \lim\limits_{n\to +\infty}\frac{1}{n}\psi^n(x) \ge a, \ \text{for} \ \mu\text{-a.e.} \ x \in X.
    \]
\end{lema}

\subsection{Geometric times}\label{subsec geom times}

A $\mathcal{C}^r$ curve $\gamma: [-1,1] \to \TT^2$ is said to be bounded if 
\begin{equation}\label{eq bounded prop of curves}
\max\limits_{s=2,\cdots,r} \|d^s\gamma\|_{\infty} \le \frac{1}{6}\|d\gamma\|_{\infty}.
\end{equation}
We remark some important properties on bounded curves:
\begin{itemize}
    \item A bounded curve has bounded distortion, i.e.
    \begin{equation}\label{eq distortion bounded curves}
    \text{for every} \ t,s \in [-1,1], \ \frac{\|d\gamma(t)\|}{\|d\gamma(s)\|} \le \frac{3}{2}.
    \end{equation}
    \item The projective component of $\gamma$ also oscillates slowly, that is, by fixing local charts covering $\gamma$, we have
    \[
    \text{for every}\ t,s \in [-1,1], \  \angle (d\gamma(t),d\gamma(s)) \le \frac{\pi}{6}
    \]
\end{itemize}

A $\mathcal{C}^r$ curve is strongly $\epsilon$-bounded when it is bounded and $\|d\gamma\|_{\infty} \le \epsilon$. Burguet remarks that such a map satisfies $\|\gamma\|_r = \max_{1 \le s \le r}\|d^s\gamma\|_{\infty} \le \epsilon$, which is the standard $\mathcal{C}^r$ upper bound required for the reparametrizations on the usual Yomdin's theory. Hence the bounded property \eqref{eq bounded prop of curves} is the main difference introduced which allows the control of the distortion along the curve. 

There is yet another difference between the Burguet's and the Yomdin's approach. Here we also want to control the size of our reparametrizations from below, when possible. This way, he introduces the notion of geometric times, which we adapt to the non-invertible case.

We consider a $\mathcal{C}^r$ non-invertible local diffeomorphism $g: \TT^2 \to \TT^2$, and a curve $\sigma:[-1,1] \to \TT^2$. For $x = \sigma(s)$, we take $v_x = \sigma'(s)/\|\sigma'(s)\| \in T_x\TT^2$ an unitary vector in the direction tangent to $\sigma$ at $x$. We also consider the map $\Hat{g}:L_g \to L_g$ and some fixed identification of $\pi_{g}^{-1}(\sigma)$ with $\sigma \times \Sigma$ through a homeomorphism $\Xi: \sigma \times \Sigma \to \pi_{g}^{-1}(\sigma) $ as in Subsection \ref{subsec Solenoid}. As before, for $\Hat{x} = (x_n)_n \in \pi_{g}^{-1}(\sigma)$, we let $v_{\Hat{x}} = v_{x_0} \in  T_x\TT^2$.

\begin{defn}(Strongly $(n,\epsilon)$-bounded curves)
For each $\omega \in \Sigma$, we consider the curve $\Hat{\sigma}_\omega:[-1,1] \to L_g$ as $\Hat{\sigma}_\omega(s) = \Xi(\sigma(s),\omega)$, then $\Hat{\sigma}_\omega$ is strongly $(n,\epsilon)$-bounded if for every $0 \le k \le n$ the curve $\sigma_{\omega}^k := \pi_{g} \circ \Hat{g}^{-k} \circ \Hat{\sigma}_\omega$ is strongly $\epsilon$-bounded.
\end{defn}
Basically this property depends on the pre-images of $\sigma$ by $g$ up to time $n$ associated with $\omega$, in particular if $\Hat{\sigma}_\omega$ is strongly $(n,\epsilon)$-bounded, then for every $\omega'\in \Sigma$ with $\omega'_i = \omega_i$ for all $1 \le i \le n$, the curve $\Hat{\sigma}_{\omega'}$ is also strongly $(n,\epsilon)$-bounded.

\begin{defn}\label{defn geometric times}(Geometric Times) \ 
    Fix a $\mathcal{C}^r$ curve $\sigma: [-1,1] \to \TT^2$ and an identification $\Xi: \sigma \times \Sigma \to \pi_{g}^{-1}(\sigma)$. For $\Hat{x} = \Xi(x,\omega) \in \pi_{g}^{-1}(\sigma)$, given $\alpha,\epsilon > 0$, a positive integer $n$ is called an $(\alpha,\epsilon)$-geometric time of $\Hat{x}$ if there exists an affine map $\theta_n^\omega:[-1,1]\to [-1,1]$ such that:
    \begin{itemize}
        \item $\Hat{\sigma}_\omega \circ \theta_n^\omega(0) = \Hat{x}$.
        \item $\Hat{\sigma}_\omega \circ \theta_n^\omega $ is strongly $(n,\epsilon)$-bounded.
        \item $\||d(\sigma_\omega^n\circ \theta_n^\omega)(0)\| \ge \frac{3}{2}\alpha\epsilon$, where $\sigma^n_\omega = \pi_{g}\circ \Hat{g}^{-n} \circ \Hat{\sigma}_\omega$.
    \end{itemize}
\end{defn}

\subsection{Reparametrization Lemma}

We present the following lemma with purpose to illustrate the idea behind the upcoming division of reparametrizations into red and blue ones (in item (2) of the Reparametrization Lemma).

\begin{lema}\label{lema red reparametrizations}[Lemma 6, \cite{burguet2024srb}]
    Let $\gamma:[-1,1]\to \TT^2$ be a $\mathcal{C}^r$ bounded curve with $\|d\gamma\|_{\infty} \ge \epsilon$. Then, there is a family $L:= \ol{L} \cupdot \underline{L}$ of affine maps $\theta:[-1,1]\to [-1,1]$ such that:
    \begin{itemize}
        \item for every $\theta \in L$, the curve $\gamma \circ \theta$ is strongly $\epsilon$-bounded and $\|d(\gamma \circ \theta)(0)\| \ge \frac{\epsilon}{6}$.
        \item $[-1,1]$ is the union of $\bigcup_{\theta \in \ol{L}} \theta([-1/3,1/3])$ and $\bigcup_{\theta \in \underline{L}}\theta([-1,1])$.
        \item $\# \underline{L} \le 2$ and $\#\ol{L}\le 6\left( \frac{\|d\gamma\|_{\infty}}{\epsilon}+1 \right)$.
        \item For every $x \in \gamma$, $\#\{\theta \in L : (\gamma \circ \theta) \cap B(x,\epsilon) \neq \emptyset\} \le 100\|d\gamma\|_\infty$.
    \end{itemize}
\end{lema}

We now give a brief overview on Burguet's reparametrization lemma. For textual cohesion purposes we already present it on the non-invertible scenario and we give the proper remarks on how it can be extended to this case. Such construction in the diffeomorphism case can be found on Section 4.2 \cite{burguet2024srb}.

To construct the reparametrizations, we shall divide a curve $\sigma:[-1,1] \to \TT^2$ as follows. Fixed $n \ge 1$, we define $k_n,\tilde{k_n}: \pi_{g}^{-1}(\sigma) \to \zz$ as:
\[
k_n(\Hat{x}) := \lfloor\log \|D\Hat{g}^{-1}_{\Hat{g}^{-(n-1)}(\Hat{x})} \| \rfloor, \ \ \ \tilde{k}_n(\Hat{x}) = \lfloor\log \|D\Hat{g}^{-1}_{\Hat{g}^{-(n-1)}(\Hat{x})} \cdot v^{(n-1)}_{\Hat{x}}\|\rfloor,
\]
where $\lfloor \cdot \rfloor$ is the floor function and $ v^n_{\Hat{x}} = \frac{D\Hat{g}^{-n}_{\Hat{x}}\cdot v_{\Hat{x}}}{\|D\Hat{g}^{-n}_{\Hat{x}}\cdot v_{\Hat{x}}\|}$. 

We consider the function $\boldsymbol{k}_n: \pi_{g}^{-1}(\sigma) \to \zz^{2n}$, given by
\begin{equation}\label{eq defn k n omega}
\boldsymbol{k}_n(\Hat{x}) = (k_1(\Hat{x}), \tilde{k}_1(\Hat{x}), k_2(\Hat{x}), \cdots, \tilde{k}_{n-1}(\Hat{x}), k_{n}(\Hat{x}), \tilde{k}_{n}(\Hat{x})).
\end{equation}
Finally, for a $2n$-tuple $\boldsymbol{k} = (k_1,\tilde{k}_1,\cdots,k_{n},\tilde{k}_{n}) \in \zz^{2n}$, we consider the following subset of $\pi_{g}^{-1}(\sigma)$:
\[
\mathcal{H}(\boldsymbol{k}):= \{\Hat{x} \in \pi_{g}^{-1}(\sigma): \boldsymbol{k}_n(\Hat{x}) = \boldsymbol{k}\}.
\]

We restate now the Burguet's Reparametrization Lemma \cite{burguet2024srb} for the endomorphism case. We define a reparametrization tree as a collection $\mathcal{R}$ of affine maps $\theta: [-1,1] \to [-1,1]$.

\begin{lema}[Reparametrization Lemma]
    Let $g \in \text{End}^r(\TT^2)$, $r>1$, $d = \deg g$ be fixed. We may fix a scale $\epsilon_0 >0$ small enough (depending only on $g$ and $r$) with the following property. There exist a universal constant $C_r$ (depending only on the regularity $r$) such that for any $\mathcal{C}^r$ strongly $\epsilon$-bounded curve $\sigma:[-1,1]\to \TT^2$  with $0<\epsilon<\epsilon_0$, we may construct a reparametrization tree $\mathcal{R} = \bigcupdot_{n \ge 0} \bigcupdot_{\tau \in \{1,\cdots,d\}^n}\mathcal{R}_n^\tau$ satisfying that for every $n \in  \nn$, $\tau = (\tau_1,\cdots,\tau_n) \in \{1,\cdots,d\}^n$, it holds:

    \begin{enumerate}
        \item For every $\omega \in C(\tau)=\{\omega \in \Sigma: \omega_i = \tau_1 \ \forall i =1,\cdots, n\}$ and every $\theta_n^\tau  \in \mathcal{R}_n^\tau$, the curve $\Hat{\sigma}_\omega \circ \theta_n^\tau$ is strongly $(n,\epsilon)$-bounded.
        
        \item  Each $\mathcal{R}_n^{\tau}$ can be decomposed into the disjoint union of \textbf{blue} ($\underline{\mathcal{R}_n^\tau}$), and \textbf{red} ($\ol{\mathcal{R}_n^\tau}$) reparametrizations satisfying that for every red one $\theta_n^\tau \in \ol{\mathcal{R}_n^\tau}$
        \begin{equation}\label{eq expanding prop red rep}
        \||d(\pi_g \circ \Hat{g}^{-n} \circ \Hat{\sigma}_\omega \circ \theta_n^\tau)(0)\| \ge \frac{\epsilon}{6}, \ \text{for some  (hence any)} \ \omega \in C(\tau).
        \end{equation}
        
        \item Each $\mathcal{R}_n^\tau$ can also be divided as $\bigcupdot_{\boldsymbol{k} \in \zz^{2n}}\mathcal{R}_n^\tau(\boldsymbol{k})$ satisfying that for every $\omega \in C(\tau)$, the set $(\Hat{\sigma}_\omega)^{-1}(\mathcal{H}(\boldsymbol{k}))$ is contained in the set:
        \[
        \bigcup\limits_{\theta \in \ol{\mathcal{R}_n^\tau}(\boldsymbol{k})}\theta([-1/3,1/3]) \cup \bigcupdot\limits_{\theta \in \underline{\mathcal{R}_n^\tau}(\boldsymbol{k})}\theta([-1,1]).
        \]
        
        \item For each $\theta_n^\tau \in \mathcal{R}_n^\tau(\boldsymbol{k})$, $\boldsymbol{k} = (k_1,\tilde{k}_1,\cdots,k_{n},\tilde{k}_{n}) \in \zz^{2n}$, there is an affine contraction $\varphi :[-1,1]\to [-1,1]$ , with contraction rate smaller than $1/100$, such that $\theta_n^\tau = \theta_{n-1}^{\tilde{\tau}} \circ \varphi$ for some $\theta_{n-1}^{\tilde{\tau}} \in \mathcal{R}_{n-1}^{\tilde{\tau}}(\tilde{\boldsymbol{k}})$, where $\tilde{\tau} = (\tau_1,\cdots,\tau_{n-1})$ and $\tilde{\boldsymbol{k}}= (k_1,\tilde{k}_1,\cdots,k_{n-1},\tilde{k}_{n-1})$.
        
        \item If $\theta_n^\tau = \theta_{n-1}^{\tilde{\tau}}\circ \varphi$ with $\theta_{n-1}^{\tilde{\tau}} \in \ol{\mathcal{R}_{n-1}^{\tilde{\tau}}}$ being a red reparametrization, then $\theta_n^\tau([-1,1]) \subset \theta_{n-1}^{\tilde{\tau}}([-1/3,1/3])$.
        
        \item We may control inductively the cardinality of $\mathcal{R}_n^\tau(\boldsymbol{k})$ as follows. For a fixed element $\theta_{n-1}^{\tilde{\tau}} \in \mathcal{R}_{n-1}^{\tilde{\tau}}(\tilde{\boldsymbol{k}})$, we have
        \begin{align*}
        &\#\{\theta \in \mathcal{R}_n^\tau(\boldsymbol{k}) \cap \ol{\mathcal{R}_n^\tau}: \theta = \theta_{n-1}^{\tilde{\tau}} \circ \varphi, \ \text{for some} \ \varphi \ \text{as in item (4)}\} \le C_re^{\max\left\{k_{n}, \frac{k_{n}-\tilde{k}_{n}}{r-1}\right\}}, \\
        &\#\{\theta \in \mathcal{R}_n^\tau(\boldsymbol{k}) \cap \underline{\mathcal{R}_n^\tau}: \theta = \theta_{n-1}^{\tilde{\tau}} \circ \varphi, \ \text{for some} \ \varphi \ \text{as in item (4)} \} \le C_re^{\frac{k_{n}-\tilde{k}_{n}}{r-1}}.
        \end{align*}
    \end{enumerate}
\end{lema}
\begin{proof}
    Fix the scale $\epsilon_0>0$ such that:
    \begin{itemize}
        \item For every $x \in \TT^2$, the exponential map $\exp_x:T_x\TT^2 \to \TT^2$ is injective on $B(0,\epsilon_0)$;
        \item For every $x \in \TT^2$, if $V_x = \exp_x(B(0,\epsilon_0))$, then every $h: V_x \to \TT^2$, inverse branch of $g$, is a well defined $\mathcal{C}^r$ immersion. 
        \item For every $x \in \TT^2$ and every $h:V_x \to \TT^2$ inverse branch of $g$, the map $h^x_{2\epsilon} = h \circ \exp_x (2\epsilon\cdot): \{v \in T_x\TT^2:\|v\| \le 1\} \to \TT^2$ satisfies $\|d^sh^x_{2\epsilon}\| \le 3 \epsilon \|D_xh\|$ for all $s = 1,\cdots, r$ and all $\epsilon<\epsilon_0$.
    \end{itemize}

    We fix $\sigma:[-1,1] \to \TT^2$ a strongly $\epsilon$-bounded curve for some $\epsilon< \epsilon_0$ and construct the reparametrizations inductively on $n \ge 0$. For $n=0$, set $\mathcal{R}_0 = \underline{\mathcal{R}}_0 = \{Id\}$. 
    
    Assume that $\mathcal{R}_{n}^\tau$ is constructed for every $\tau \in \{1,\cdots,d\}^n$. Let $\tau = (\tau_1,\cdots, \tau_{n+1}) \in \{1,\cdots,d\}^{n+1}$ be fixed and set $\tilde{\tau} = (\tau_1,\cdots,\tau_n) \in \{1,\cdots,d\}^n$, similarly let $$\boldsymbol{k} = (k_1,\tilde{k}_1,\cdots,k_{n+1},\tilde{k}_{n+1}) \in \zz^{2(n+1)}$$ be fixed and set $\tilde{\boldsymbol{k}} = (k_1,\tilde{k}_1,\cdots, k_n,\tilde{k}_n) \in \zz^{2n}$. For each $\theta_n^{\tilde{\tau}} \in \mathcal{R}_n^{\tilde{\tau}}(\tilde{\boldsymbol{k}})$, we will define the elements $\theta_{n+1}^\tau \in \mathcal{R}_{n+1}^\tau(\boldsymbol{k})$ that satisfy items $2$ and $3$ of the Lemma. We fix $\theta_{n+1}^\tau \in \mathcal{R}_{n+1}^\tau(\boldsymbol{k})$ and write
    \[
    \Hat{\theta}_n^\tau = \left\{\begin{array}{ll}
        \theta_n^\tau(\frac{1}{3}\cdot), \ &\text{if} \ \theta_n^\tau \in \ol{\mathcal{R}_n^\tau}  \\
        \theta_n^\tau, &\text{if} \ \theta_n^\tau \in \underline{\mathcal{R}_n^\tau}.
    \end{array}\right.
    \]
    We fix any $\omega \in C(\tilde{\tau})$ , we are interested in the curves:
    \[
    \Hat{\sigma}_\omega \circ \Hat{\theta}_n^{\tilde{\tau}}, \ \text{and} \ \sigma_\omega^n \circ \theta_n^{\tilde{\tau}} = \pi_g\circ \Hat{g}^{-n} \circ \Hat{\sigma}_\omega \circ \Hat{\theta}_n^{\tilde{\tau}}.
    \]
    For clarification, the latter curve is the pre-image of the first by the inverse branch of $g$ corresponding to $\tilde{\tau}$, thus this construction does not depend on the choice of $\omega \in C(\tau)$.

    We may now follow the steps on Burguet's Reparametrization Lemma, we give an outline of these steps and refer the reader to \cite{burguet2024srb} for details. Our purpose is to define the contractions $\varphi:[-1,1] \to [-1,1]$ such that the reparametrizations of the type $\Hat{\theta}_n^{\tilde{\tau}}\circ \varphi$ satisfy our conclusions as the elements of $\mathcal{R}_{n+1}^\tau$.

    At steps 1-3, by studying the inverse branch on a neighborhood of $\sigma_\omega^n \circ \theta_n^{\tilde{\tau}}$ associated to the pre-image given by $\tau$, we obtain a constant $C_r>0$ (depending only on $r$) such that for any affine contraction $\Tilde{\varphi}:[-1,1]\to [-1,1]$ with contraction rate $b = (C_r e^{k_{n+1}-\tilde{k}_{n+1}-4})^{-\frac{1}{r-1}}$ (which can be supposed smaller than $1/100$ by increasing $C_r$ if necessary), if $\Hat{\theta}_n^{\tilde{\tau}}\circ \tilde{\varphi}([-1,1]) \cap H^\omega(\boldsymbol{k}) \neq \emptyset$ then the curve $\sigma_{\omega}^{n+1} \circ \Hat{\theta}_n^{\tilde{\tau}} \circ \tilde{\varphi}= \pi_{g}\circ \Hat{g}^{-1} \circ \Hat{g}^{-n} \circ \Hat{\sigma}_\omega \circ \Hat{\theta}_n^{\tilde{\tau}} \circ \tilde{\varphi}$ is bounded. 
    
    We may cover $[-1,1]$ with no more than $b^{-1}+1$ of such affine maps $\Tilde{\varphi}$, this term is the source of the factor $e^{\frac{k_n-\tilde{k}_n}{r-1}}$ on the last item of the Reparametrization Lemma. Although theses curves are bounded, they may not satisfy $\|d(\sigma^{n+1}_\omega \circ \Hat{\theta}_n^{\tilde{\tau}} \circ \tilde{\varphi})\|_{\infty} <\epsilon$. What is left then is to obtain the strongly $\epsilon$-bounded property. This is obtained in the last step as follows:
    \begin{itemize}
        \item If $\sigma^{n+1}_\omega \circ \Hat{\theta}_n^{\tilde{\tau}} \circ \tilde{\varphi}$ is strongly $\epsilon$-bounded, there is nothing to do. We define $\varphi = \Tilde{\varphi}$, $\theta_{n+1}^\tau = \Hat{\theta}_{n}^{\tilde{\tau}} \circ \varphi$ and label it as a blue reparametrization.
        \item If not, then $\sigma^{n+1}_\omega \circ \Hat{\theta}_n^{\tilde{\tau}} \circ \tilde{\varphi}$ is a bounded curve with $\|d(\sigma^{n+1}_\omega \circ \Hat{\theta}_n^{\tilde{\tau}} \circ \tilde{\varphi})\|_{\infty}\ge \epsilon$. We then apply Lemma \ref{lema red reparametrizations} to obtain the family $L = \ol{L} \cup \underline{L}$ of affine maps $\theta:[-1,1]\to [-1,1]$ such that $\sigma^{n+1}_\omega \circ \Hat{\theta}_n^{\tilde{\tau}} \circ \tilde{\varphi} \circ \theta$ is strongly $\epsilon$-bounded. We take $\varphi = \tilde{\varphi} \circ \theta$, $\theta_{n+1}^\tau = \Hat{\theta}_n^{\tilde{\tau}} \circ \varphi$, for $\theta \in \underline{L}$ we label it a blue reparametrization and for $\theta \in \ol{L}$ we label it a red reparametrization. The red property required \eqref{eq expanding prop red rep} is achieved from the first item of Lemma \ref{lema red reparametrizations}.
    \end{itemize}

    Finally, the control of the cardinality of $\mathcal{R}_{n+1}^\tau$ is obtained from the fact that we cover the interval with no more than $b^{-1}+1$ maps $\tilde{\varphi}$ along with the control of $\#L$ given by Lemma \ref{lema red reparametrizations}.
\end{proof}

As a consequence, we also have a local version of the Reparametrization Lemma. More precisely, let $G: \mathbb{P}L_g \to \mathbb{P}L_g$ be the associated projective cocycle. For a curve $\sigma$ in $\TT^2$ and for each $\omega \in \Sigma$, we may reparametrize only the intersection of $\Hat{\sigma}_\omega$ with the projection of a $\Hat{G}^{-1}$-dynamical ball as follows.

For a given a curve $\sigma:[-1,1] \to \TT^2$ and $\omega \in \Sigma$, let $\Hat{\sigma}_\omega$ be the curve in $L_g$ associated to the $\omega$ pre-orbit of $\sigma$ by $\Hat{\sigma}_\omega(s) = \Xi(\sigma(s),\omega)$. For each $\Hat{y} \in \Hat{\sigma}_\omega$, let $(\Hat{y},v_y) \in \mathbb{P}L_g$ be the line tangent to $\sigma$ at $y = \pi_{g}(\Hat{y})$. For each $x \in \sigma$, $\epsilon>0$ and $n \in \nn$, we define:
\[
B^{\Hat{G}^{-1}}_{\omega}(x,\epsilon,n) = \{\Hat{y} \in \Hat{\sigma}_{\omega}: \Hat{d}(\Hat{G}^{-j}(\Hat{y},v_y), \Hat{G}^{-j}(\Hat{x},v_x))<\epsilon, \ \text{for every} \ 0 \le j \le n\}.
\]
This is the intersection of $\sigma$ with the projection to $\TT^2$ of the $\Hat{G}^{-1}$-dynamical ball at $(\Hat{x}_\omega,v_x)$. Note that for every $\omega'$ with $\omega'_i = \omega_i$ for every $1 \le i \le n$, we have 
\[
\pi_{g}(B^{\Hat{G}^{-1}}_{\omega'}(x,\epsilon,n)) = \pi_{g}(B^{\Hat{G}^{-1}}_{\omega}(x,\epsilon,n)),
\]
thus we may define for $\tau \in \{1,\cdots,d\}^n$:
\[
B^{\Hat{G}}_{\tau}(x,\epsilon) = \pi_{g}(B^{\Hat{G}^{-1}}_{\omega}(x,\epsilon,n)), 
\]
for some $\omega \in C(\tau) = \{\omega \in \Sigma: \omega_i = \tau_i \ \forall \ i = 1,\cdots, n\}$. 

Define also the function $\mathcal{w}:\mathbb{P}L_g \to \rr$ as 
\[
\mathcal{w}(\Hat{x},v) = \log \|Dg_x\| - \log \|Dg_x\cdot v\|,
\]
and for $n \in \nn$, let $(\mathcal{w}^n)_n$ be the associated $\Hat{G}^{-1}$ additive cocycle defined by 
\[
\mathcal{w}^n(\Hat{\xi}) = \sum\limits_{j=0}^{n-1} \mathcal{w} \circ \Hat{G}^{-j}(\Hat{\xi}).
\]

We consider $\epsilon<\epsilon_0$ (the scale given by the Reparametrization Lemma) and we also assume that for $(x,v), (y,w) \in \mathbb{P}L_g$:
\begin{equation*}
d((x,v),(y,w))<\epsilon \  \Rightarrow  \left\{\begin{array}{ll}
     &|\log \|Dg_x\cdot v\| -\log \|Dg_y\cdot w\|\ | <1,\ \text{and}\\
     &|\log\|Dg_x\|-\log\|Dg_y\|\ |<1. 
\end{array}\right.
\end{equation*}

\begin{corol}\label{corol local reparam}
    There is a constant $C_r>0$ (depending only on r) such that for any $\epsilon<\epsilon_0(g)$ given by the Reparametrization Lemma, for any strong $\epsilon$-bounded curve $\sigma: [-1,1] \to \TT^2$ and $x \in \sigma$. For every $n \in \nn$ and every $\tau \in \{1,\cdots,d\}^n$, if $\omega \in C(\tau)$, then
    \[
    \#\{\theta \in \mathcal{R}_n^\tau: (\sigma \circ \theta) \cap B^{\Hat G}_\tau(x,\epsilon) \neq \emptyset\} \le C_r^n \exp \left(\frac{\mathcal{w}^n(\Hat{x}_\omega,v_x)}{r-1}\right),
    \]
    where $\Hat{x}_{\omega} = \Xi(x,\omega)$ is the unique point in $\pi_{g}^{-1}(x)\cap \Hat{\sigma}_\omega$.
\end{corol}

For a clarification, we would like to apply the reparametrization Lemma to $g = f^q$, $q \ge 1$ and $f \in \text{End}^r(\TT^2)$, thus we return now to our notation of $f \in \text{End}^r(\TT^2)$ as in Theorem \ref{main thm burguet}. As a consequence of this local reparametrization Lemma, we may estimate the number of strongly $(\omega,n,\epsilon)$-bounded curves reparametrizing the intersection of a given strongly $\epsilon$-bounded curve  with an $\Hat{F}^{-1}$-dynamical ball of length $n$ and radius $\epsilon$. This estimate is an instrument for the proof of the F{\o}lner-Gibbs property (Proposition \ref{prop folner gibbs prop with q}). 

For a given $q \in \nn$, we define $\mathcal{w}_q: \mathbb{P}L_f\to \rr$ by
\[
\mathcal{w}_q(\Hat{x},v) := \frac{1}{q}(\log \|D\Hat{f}_{\Hat{x}}^{-q}\| - \log \|D\Hat{f}_{\Hat{x}}^{-q}\cdot v\|),
\]
and we denote by $(\mathcal{w}_q^n)_n$ the associated $\Hat{F}^{-1}$ additive cocycle given by
\[
\mathcal{w}_q^n(\Hat{\xi}) = \sum\limits_{j=0}^{n-1}\mathcal{w}_q(\Hat{F}^{-j}(\Hat{\xi})).
\]

\begin{lema}\label{lema rep of dynamical balls}
    Given $q \in \nn$, there exist $\epsilon_q>0$ and $B_q>0$ such that for any strongly $\epsilon_q$-bounded curve $\sigma:[-1,1]\to \TT^2$, for any $x \in \sigma$, any $\omega \in \Sigma$ and any $n \in \nn$ there exist a family $\Theta_{q,n}^\omega$ of affine maps of $[-1,1]$ such that
    \begin{itemize}
        \item $B^{\Hat{F}^{-1}}_{\omega}(x, \epsilon,n) \subset \bigcup\limits_{\theta \in \Theta_{q,n}^\omega} \Hat{\sigma}_\omega \circ \theta$;
        \item $\sigma \circ \theta$ is strongly $(\omega,n,\epsilon_q)$-bounded (with respect to $f$) for any $\theta \in \Theta_{q,n}^\omega$;
        \item $\#\Theta_{q,n}^\omega \le B_q C_r^{n/q}\exp{\left(\frac{\mathcal{w}_q^n(\Hat{x}_\omega,v_x)}{r-1}\right)}$, where $\Hat{x}_\omega = \Xi(x,\omega)$ is the unique point in $\pi_{f}^{-1}(x)\cap \Hat{\sigma}_\omega$ and $C_r$ is the constant given by the Reparametrization Lemma.
    \end{itemize}    
\end{lema}

For the ideas behind Lemma \ref{lema rep of dynamical balls} and Corollary \ref{corol local reparam} we refer to \cite{burguet2024srb} (respectively Lemma 8 and Corollary 3), we remark that there is no major difference in the argument required for the endomorphism case but to consider inverse branches as we did in the Reparametrization Lemma.

\subsection{The Geometric set}

We now return to our context, we let $f:\TT^2 \to \TT^2$ be a $\mathcal{C}^r$ non-invertible local diffeomorphism as in Theorem \ref{main thm burguet} and let $d$ be the number of pre-images of $f$. For some positive integer $p \in \nn $, we apply the Reparametrization Lemma to $g = f^p$, we fix $\sigma:[-1,1] \to \TT^2$ a strongly $\epsilon$-bounded curve for $\epsilon<\epsilon_0$ (the scale given by the Lemma). 

There is an important observation to be made here. The inverse limit space constructed depend on the map $f$. Rigorously, when we apply the Reparametrization Lemma to $f^p$ we would obtain $\mathcal{R}^p$ a reparametrization tree of $\sigma$ with respect to $L_{f^p}$, that is,
\[
\mathcal{R}^p = \bigcupdot\limits_{m\ge 0} \bigcupdot\limits_{\tilde{\tau} \in \{1,\cdots,d^p\}^m} \mathcal{\Tilde{R}}_m^{\tilde{\tau}}.
\]
  
For better understanding we would like to rewrite such reparametrization tree as follows. Write $\pi_{f}:L_f \to \TT^2$ and $\pi_{f^p}: L_{f^p} \to \TT^2$ as the canonical projections, and let $\Xi:\sigma \times \{1,\cdots,d\}^\nn \to \pi_{f}^{-1}(\sigma)$ and $\Xi_p:\sigma \times \{1,\cdots,d^p\}^\nn \to (\pi_{f^p})^{-1}(\sigma)$ the corresponding identifications as in Subsection \ref{subsec Solenoid}. For each $\tilde{\tau} \in \{1,\cdots,d^p\}^m$, there is a unique $\tau \in \{1,\cdots,d\}^{mp}$ such that for every $\omega \in C(\tau):=\{\omega \in \{1,\cdots,d\}^\nn: \omega_i = \tau_i, \forall i=1,\cdots m\}$ and every $\tilde{\omega} \in C(\tilde{\tau}) \subset \{1,\cdots,d^p\}^\nn$ where:
\[
\pi_{f^p}\circ (\Hat{f^p})^{-m} \circ \Xi_p(\sigma(s),\tilde{\omega}) = \pi_{f}\circ \Hat{f}^{-mp}\circ \Xi(\sigma(s),\omega).
\]
Then we may write $\mathcal{\Tilde{R}}_m^{\tilde{\tau}}:=\mathcal{R}_{mp}^\tau$ and
\begin{equation}\label{eq relation rep fp and rep f}
\mathcal{R}^p = \bigcupdot\limits_{m\ge 0} \bigcupdot\limits_{\tau \in \{1,\cdots,d\}^{mp}}\mathcal{R}_{mp}^{\tau}.
\end{equation}

\begin{defn} (Geometric Set)
    For $\Hat{x} \in \pi_{f}^{-1}(\sigma)$, let $\omega = (\omega_1,\omega_2,\cdots) \in \{1,\cdots,d\}^\nn$ to be the unique sequence such that $\Hat{x} = \Xi(x,\omega)$ and define the set $E_p(\Hat{x}) \subset p\nn$ of integers $mp$ for which there is a red reparametrization $\theta \in \ol{\mathcal{R}_{mp}^\tau}(\boldsymbol{k})$ with $\pi_{f}(\Hat{x}) \in \sigma \circ \theta([-1/3,1/3])$, or equivalently $\Hat{x} \in \Hat{\sigma}_\omega\circ \theta([-1/3,1/3])$, where $\tau = (\omega_1,\cdots,\omega_{mp})$ and $\boldsymbol{k} = \boldsymbol{k}^\omega_{m}(\Hat{x})$ is as in \eqref{eq defn k n omega} for $g = f^p$.
\end{defn}

Although the construction of such geometric set for a point $\Hat{x} \in \pi_{f}^{-1}(\sigma)$ depends on the reparametrization tree given by the Reparametrization Lemma, there is a more intrinsic property of the curve around $\Hat{x}$ that we would like to explore. More specifically, we may relate the geometric set of $\Hat{x}$ with its geometric times.

\begin{lema}
    There are $\alpha_p>0$ and $\epsilon_p>0$ (depending only on $r$, $f$ and $p$) such that any $n \in E_p(\Hat{x})$ is a $(\alpha_p,\epsilon_p)$-geometric time for $\Hat{x}$ (with respect to $\Hat{f}$).
\end{lema}
\begin{proof}
     Fix $\Hat{x} \in \pi_{f}^{-1}(\sigma)$  and let $\omega = \omega(\Hat{x})$ be such that $\Hat{x}=\Xi(x,\omega)$. Let $x = \pi_{f}(\Hat{x}) = \sigma\circ \theta (b)$, for $b \in [-1/3,1/3]$, $\theta \in \ol{\mathcal{R}_{mp}^\tau}(\boldsymbol{k})$, $\tau = (\omega_1,\cdots,\omega_{mp})$ and $\boldsymbol{k} = \boldsymbol{k}^\omega_m(\Hat{x})$. We write the identification $\sigma^n_\omega = \pi_{f} \circ \Hat{f}^{-n} \circ \Xi(\sigma,\omega)$ (so that there is no confusion with what could be an identification made with $\Xi_p)$.
     
     From the first item of the reparametrization Lemma and the identity \eqref{eq relation rep fp and rep f}, we have that for every $0 \le k \le m$ the curve $\sigma^{kp}_\omega \circ \theta$ is strongly $\epsilon$-bounded.

    If we define $\tilde{\theta}(s) = \theta(b+\frac{2}{3}s)$, we have $\|d(\sigma \circ \tilde{\theta})(s)\| = \frac{2}{3}\|d(\sigma \circ \theta)(b+\frac{2}{3}s)\|<\epsilon$, and for every $2 \le s \le r$
    \begin{align*}
    \|d^s(\sigma \circ \tilde{\theta})\|_{\infty} &\le \left(\frac{2}{3}\right)^{s}\|d^s(\sigma \circ \theta)\|_{\infty} \le   \frac{1}{6} \left(\frac{2}{3}\right)^{s} \|d(\sigma\circ\theta)\|_{\infty}, \ \text{since} \ \sigma \circ \theta \ \text{is bounded}, \\
    &\le \frac{1}{6} \left(\frac{2}{3}\right)^{s}  \frac{3}{2} \|d(\sigma\circ \theta)(0)\| \le \frac{1}{6} \left(\frac{2}{3}\right)^{s-1}  \|d(\sigma\circ \theta)(0)\|, \ \text{from \eqref{eq distortion bounded curves}}, \\
    &\le \frac{1}{6}\left(\frac{2}{3}\right)^{s-2}\|d(\sigma \circ \tilde{\theta})(-3b/2)\| \le \frac{1}{6}\|d(\sigma\circ\tilde{\theta})\|_{\infty}.
    \end{align*}
    Hence, $\sigma\circ\tilde{\theta}$ is a strongly $\epsilon$-bounded curve with $\sigma\circ\tilde{\theta}(0) = x$. Analogously, we may obtain that $\sigma^{kp}_\omega \circ \tilde{\theta}$ is strongly $\epsilon$-bounded for every $0 \le k \le m$. Moreover, from the property \eqref{eq expanding prop red rep} of red reparametrizations, we also have for every $0 \le k \le m$
    \begin{align*}
        \|d(\sigma^{kp}_\omega \circ \tilde{\theta})(0)\| = \frac{2}{3}\|d(\sigma^{kp}_\omega \circ \theta)(b)\| \ge \frac{4}{9}\|d(\sigma^{kp}_\omega \circ \theta)(0)\|\ge \frac{4}{9}\frac{\epsilon}{6} = \frac{2}{27}\epsilon.
    \end{align*}
    
We would like to obtain similar results for every $0 \le n \le mp$. For $0<a<1$ we let $\Tilde{\theta}_a(s) = \Tilde{\theta}(as)$, we fix $0\le n < mp$ and let $\ell \in \{0,\cdots, m-1\}$ be such that $\ell p \le n <(\ell+1)p$. By utilizing the same approach as in the Reparametrization Lemma with inverse branches, we may show that there is a constant $C_p>0$ (depending only on $r$, $f$ and $p$) such that for every $1 \le s \le r$
\[
\|d^s(\sigma^n_\omega \circ \tilde{\theta}_a)\|_{\infty} \le C_p a^s \|d^s(\sigma^{\ell p}_\omega \circ \tilde{\theta})\|_{\infty}.
\]
As $\sigma^{\ell p}_\omega \circ \tilde{\theta}$ is bounded and $f^{n-\ell p} \circ \sigma^{n}_\omega = \sigma^{\ell p}_\omega$, we obtain
\begin{align*}
    \|d^s(\sigma^n_\omega \circ \tilde{\theta}_a)\|_{\infty} &\le \frac{3}{2} C_p a^s \|d(\sigma^{\ell p}_\omega \circ \tilde{\theta})(0)\| \\
    &\le \frac{3}{2} C_p a^s \|Df\|_{\infty}^p  \|d(\sigma^n_\omega \circ \tilde{\theta})(0)\|\\
    &\le \frac{3}{2} C_p a^{s-1} \|Df\|_{\infty}^p  \|d(\sigma^n_\omega \circ \tilde{\theta}_a)\|_{\infty}.
\end{align*}

Then we may fix $a = \frac{1}{9C_p\|Df\|_{\infty}^p}$ so that $\sigma\circ \Tilde{\theta}_a$ is strongly $(\omega,mp,\epsilon_p)$-bounded (with respect to $f$), with $\epsilon_p = \frac{\epsilon}{9C_p}$. Finally,
\[
    \|d(\sigma^{mp}_\omega \circ \tilde{\theta}_a)(0)\| = a \|d(\sigma^{mp}_\omega \circ \tilde{\theta})(0)\| \ge \frac{2}{27} a \epsilon = \frac{3}{2}\alpha_p\epsilon_p,
\]
for $\alpha_p = \frac{4}{81\|Df\|^p_{\infty}}$. Hence, $mp$ is a $(\alpha_p,\epsilon_p)$-geometric time for $\Hat{x}$.
\end{proof}

Finally, we show the exponential growth of vectors along geometric times given by Lemma \ref{lema prop of large sets}. We consider the function $\phi:\mathbb{P}L_f \to \rr$ given by $\phi(\Hat{x},v) = \log\|D\Hat{f}^{-1}_{\Hat{x}}\cdot v\|$ and the associated $\Hat{F}^{-1}$-cocycle
\[
\phi^n(\Hat{x},v) = \sum\limits_{j=0}^{n-1}\phi \circ \Hat{F}^{-j}(\Hat{x},v).
\]

\begin{lema}\label{lema Ep is large}
    For every $p \in \nn$, there is a constant $a_p>0$ such that for every $\Hat{x} \in \pi_{f}^{-1}(\sigma)$ and every $mp>kp \in E_p(\Hat{x})$:
    \[
    \phi^{mp-kp}(\Hat{F}^{-m}(\Hat{x},v_{\Hat{x}})) > (mp-kp)a_p.
    \]
\end{lema}
\begin{proof}
   Fix $\Hat{x} \in \pi_{f}^{-1}(\sigma)$  and let $\omega = \omega(\Hat{x})$ be such that $\Hat{x}=\Xi(x,\omega)$, we fix $p \in \nn$ and $mp>kp \in E_p(\Hat{x})$. Let $x = \pi_{f}(\Hat{x}) = \sigma \circ \theta_m (t)$, for $t \in [-1/3,1/3]$, $\theta_m \in \ol{\mathcal{R}_{mp}^\tau}(\boldsymbol{k})$, $\tau = (\omega_1,\cdots,\omega_{mp})$, and $\boldsymbol{k} = \boldsymbol{k}^\omega_m(\Hat{x})$. Analogously we let $x = \sigma\circ\theta_k(s)$, for $s \in [-1/3,1/3]$ and $\theta_k \in \ol{\mathcal{R}_{kp}^{\tilde{\tau}}}(\tilde{\boldsymbol{k}})$. Then, writing $(\Hat{x},v_x) \in \mathbb{P}L_f$ the tangent direction to $\sigma$ at $x = \pi_{f}^{-1}(\Hat{x})$, we have
   \begin{align*}
        e^{\phi^{mp-kp}(\Hat{F}^{-kp}(\Hat{x},v_x))} &= \frac{\|d(\sigma^{mp}_\omega \circ \theta_m)(t)\|}{\|d(\sigma^{kp}_\omega \circ \theta_m)(t)\|} \\
        &\ge \frac{2}{3}\frac{\|d(\sigma^{mp}_\omega \circ \theta_m)(0)\|}{\|d(\sigma^{kp}_\omega \circ \theta_m)(t)\|}, \ &\text{since} \ \sigma^{mp}_\omega \circ \theta_m \ \text{is bounded}, \\
        &\ge \frac{2}{3}\frac{\|d(\sigma^{mp}_\omega \circ \theta_m)(0)\|}{\|d(\sigma^{kp}_\omega \circ \theta_k)(s)\|} 100^{m-k}, \ &\text{from item} \ (4) \ \text{of the Reparametrization Lemma},\\
        &\ge \frac{2}{3\epsilon}\|d(\sigma^{mp}_\omega \circ \theta_m)(0)\|100^{m-k}, \ &\text{since} \ \sigma^{kp}_\omega \circ \theta_k \ \text{is strongly} \ \epsilon\text{-bounded},\\
        &\ge \frac{1}{9}100^{m-k} \ge 10^{m-k}, \ &\text{by property \eqref{eq expanding prop red rep} of} \ \theta_m \in \ol{\mathcal{R}}.
   \end{align*}

    Hence, $\phi^{mp-kp}(\Hat{F}^{-kp}(\Hat{x},v_x)) \ge (mp-kp)a_p$, for $a_p = \frac{\log 10}{p}$.
\end{proof}

We end this section by extending a key result on Burguet's construction (more specifically Proposition 4 on \cite{burguet2024srb}). It states that if a curve $\sigma$ has sufficient large asymptotic rate of expansion on its tangent direction, then every point on a subset of positive Lebesgue measure has uniformly large asymptotic density of geometric times. 

Recalling the notations on Subsection \ref{subsec equidistritued meas}, the main difference between our setting and Burguet's is that we want this property on a set of positive $\Hat{\eta}_{\sigma}$ measure on $\pi_{f}^{-1}(\sigma)$. From Burguet's result, it is true on each curve $\Hat{\sigma}_\omega$ on $L_f$ when considering the Lebesgue measure on the curve, this is the reason we require that $\Hat{\eta}_{\sigma}$ disintegrates along these curves as lifts of the Lebesgue measure $\Leb_{\sigma}$.

\begin{prop}\label{prop key estimates}
    Let $b > \frac{R^-(f)}{r}$. For $p$ large enough, there exists $\beta_p>0$ such that:
    \[
    \limsup\limits_{n \to +\infty}\frac{1}{n} \log \Hat{\eta}_{\sigma}(\{\Hat{x} \in \pi_{f}^{-1}(\sigma): d_n(E_p(\Hat{x})) < \beta_p \ \text{and} \ \|D\Hat{f}_{\Hat{x}}^{-n}\cdot v_x\| \ge e^{nb}\}) < 0.
    \]
\end{prop}
\begin{proof}
    We let $\mathcal{E}_n = \{\Hat{x} \in \pi_{f}^{-1}(\sigma): d_n(E_p(\Hat{x})) < \beta_p \ \text{and} \ \|D\Hat{f}_{\Hat{x}}^{-n}\cdot v_x\| \ge e^{nb}\}$. For a fixed $p \in \nn$, let $\mathcal{R}^p$ be the reparametrization tree of $\sigma$ used to construct $E_p$, divided as in \eqref{eq relation rep fp and rep f}, let $A_f = \log\|Df\|_{\infty}+\log\|Df^{-1}\|_{\infty}+1$.  We fix $n = mp \in p\nn$, $\omega \in \Sigma$ and $\tau = (\omega_1,\cdots,\omega_n) \in \{1,\cdots,d\}^n$. By following exactly the same arguments as Burguet (see Proposition 4 \cite{burguet2024srb}) on $\Hat{\sigma}_\omega$, we obtain:
    \begin{enumerate}
        \item For $\theta \in \mathcal{R}_{mp}^\tau$, with $(\Hat{\sigma}_\omega \circ \theta) \cap \mathcal{E}_n \neq \emptyset$, the length of $\sigma \circ \theta$ is less than $3\epsilon e^{-nb}$,
        \item The number of sequences $\boldsymbol{k} \in \zz^{2m}$ in the division of $\mathcal{R}_{mp}^\tau$ as in \eqref{eq defn k n omega} is bounded by $(2pA_f+1)^{2m}$,
        \item For a fixed sequence $\boldsymbol{k}$, the number of elements $\theta_m \in \mathcal{R}_{mp}^\tau(\boldsymbol{k})$ such that $(\Hat{\sigma}_\omega \circ \theta_m) \cap \mathcal{H}(\boldsymbol{k}) \cap \mathcal{E}_n \neq \emptyset$ is bounded from above by $2^m C_r^me^m\|(Df^p){-1}\|_{\infty}^{\frac{m}{r}}\|Df^{-1}\|_{\infty}^{\beta_p p^2 m}$.
    \end{enumerate}

    This yields for every $\omega \in \Sigma$:
    \begin{align*}
        Leb_{\sigma}(\pi_{f}(\mathcal{E}_n \cap \Hat{\sigma}_\omega)) &\le \sum\limits_{\boldsymbol{k} \in \zz^{2m}} \Leb_{\sigma}(\pi_{f}(\mathcal{H}(\boldsymbol{k}) \cap \mathcal{E}_n \cap \Hat{\sigma}_\omega))\\
        &\le \sum\limits_{\boldsymbol{k} \in \zz^{2m}} \sum\limits_{\theta \in \mathcal{R}_{mp}^\tau (\boldsymbol{k})} |\sigma\circ \theta| \\
        &\le (2pA_f+1)^{2m}\cdot (2^mC_r^me^m\|(Df^p)^{-1}\|_{\infty}^{\frac{m}{r}}\|Df^{-1}\|_{\infty}^{\beta_p p^2 m})\cdot (3\epsilon e^{-nb})
    \end{align*}
    
    Finally, we may use the relation \eqref{eq fubini on meas curves} to write
    \begin{equation*}
    \Hat{\eta}_{\sigma}(\mathcal{E}_n) = \int_{\Sigma} \Leb_{\sigma}(\pi_{f}(\mathcal{E}_n\cap \Hat{\sigma}_\omega)) \ dp(\omega).
    \end{equation*}
    Hence 
    \begin{align*}
    \limsup\limits_{n \to +\infty} \frac{1}{n} \log \Hat{\eta}_{\sigma}(\mathcal{E}_n) &= \limsup\limits_{n \to +\infty}\frac{1}{n} \log \int_{\Sigma} \Leb_{\sigma}(\pi_{f}(\mathcal{E}_n\cap \Hat{\sigma}_\omega)) \ dp(\omega)\\
    &\le \frac{2}{p}\log(2pA_f+1) + \frac{\log(2C_r e)}{p} + \frac{\log(\|(Df^p)^{-1}\|_{\infty})}{p r}+p \beta_p A_f-b.
    \end{align*}
    Taking p sufficiently large, the right hand side is bounded by $\frac{R^-(f)}{r}+p\beta_pA_f -b $. Since $b>\frac{R^-(f)}{r}$, we we may choose first p large enough, then $\beta_p$ small enough such that this expression is negative. 
\end{proof}

\section{Proof of Theorem \ref{main thm burguet}}\label{sect:B}

Let $f: \TT^2 \to \TT^2$ be a non-invertible $\mathcal{C}^r$-local diffeomorphism with $d = \deg(f)$. We let $R^-(f) = \lim\limits_{n\to \infty} \frac{1}{n} \log^+ \sup_{x \in \TT^2} \|(Df_x^n)^{-1}\|$ and suppose that 
\[
I(x) = \limsup\limits_{n \ge 0} \frac{1}{n} \inf\limits_{v \in T^1_x\TT^2} \frac{1}{d^n}\sum\limits_{y \in  f^{-n}(x)}\log \|(Df_y^n)^{-1}\cdot v\| > b >  \frac{R^-(f)}{r}
\]
on a set of positive Lebesgue measure on $\TT^2$.

Take $p$ large enough (so that Proposition \ref{prop key estimates} holds), then $\epsilon = \epsilon_0(\Hat{f}^{-p})$ as in the Reparametrization Lemma. From now on, $\beta = \beta_p$ given by Prop. \ref{prop key estimates} is fixed.  By a standard argument using Fubini's Theorem, there is a $\mathcal{C}^r$ smooth embedded curve $\sigma: [-1,1] \to \TT^2$, which can be assumed to be strongly $\epsilon$-bounded, and a subset $A$ of $\sigma$ with $\Leb_{\sigma}(A)>0$ such that $I(x)>b$ for every $x \in A$. Consequently, taking $v_x = \frac{\sigma'(s_0)}{\|\sigma'(s_0)\|}$, where $\sigma(s_0) = x$, we get
\begin{equation}\label{eq limsup I x, vx larger than b}
\limsup\limits_{n \to +\infty}\frac{1}{n}I(x,v_x, f^n) > b, \ \text{for every} \ x \in A.
\end{equation}
We can also obtain $\sigma$ in a way that the countable set of periodic sinks do not intersect it.

With the notations of Subsection \ref{subsec equidistritued meas}, we take the measure $\Hat{\eta}_{\sigma}$ on $\pi_{f}^{-1}(\sigma)$ obtained from the Lebesgue measure on $\sigma$ as 
\begin{equation*}
\Hat{\eta}_{\sigma} = \int p_x d\Leb_{\sigma}(x),
\end{equation*}
where $p_x$ is the unique measure on $\pi_{f}^{-1}(x)$ satisfying for every $z \in f^{-n}(x)$ and every $n \ge 0$ that $p_x ( \{(x_k)_k: x_n = z\} ) = d^{-n}$.

\begin{claim}
    There exists a set $\Hat{G} \subset \pi_{f}^{-1}(\sigma)$ with $\Hat{\eta}_\sigma(\Hat{G})>0$ such that
\begin{equation}\label{eq large l.e. on G}
\chi_{\Hat{f}^{-1}}(\Hat{x},v_{\Hat{x}}) = \limsup\limits_{n \to +\infty}\frac{1}{n} \log \|D\Hat{f}_{\Hat{x}}^{-n}\cdot v_{\Hat{x}}\| > b.
\end{equation}
\end{claim}

\begin{proof}
From \eqref{eq limsup I x, vx larger than b},
\begin{equation}\label{set essent sup}
    \int_A\limsup_n \frac{1}{n}\int_{\pi_{f}^{-1}(x)} \log\|D\Hat{f}^{-n}_{\Hat{x}}\cdot v_{\Hat{x}} \|\ dp_x(\Hat{x})\ d\Leb_{\sigma}(x) > b\ \Leb_{\sigma}(A).
    \end{equation}
    For all large $n$ we have $\frac{1}{n} \log\|D\Hat{f}^{-n}_{\Hat{x}}\cdot v_{\Hat{x}} \| \le \frac{1}{n}\log^+ \sup\limits_{x \in M}\|(Df^n_x)^{-1}\| \le R^-(f)+\epsilon$, for some $\epsilon>0$. Hence, for every $x \in M$, we may use Fatou's Lemma on the sequence of functions $\pi_{f}^{-1}(x) \ni  \Hat{x} \mapsto R^-(f)+\epsilon - \frac{1}{n} \log\|D\Hat{f}^{-n}_{\Hat{x}}\cdot v_{\Hat{x}}\|$ to obtain that
    \[
    \int_{\pi_{f}^{-1}(x)}\limsup\limits_n\frac{1}{n}\log\|D\Hat{f}^{-n}_{\Hat{x}}\cdot v_{\Hat{x}} \|\ dp_x(\Hat{x}) \ge \limsup_n \frac{1}{n}\int_{\pi_{f}^{-1}(x)} \log\|D\Hat{f}^{-n}_{\Hat{x}}\cdot v_{\Hat{x}} \|\ dp_x(\Hat{x}).
    \]
    
    Combining with \eqref{set essent sup}, we get
    \[
    \int_A\int_{\pi_{f}^{-1}(x)}\limsup\limits_n\frac{1}{n}\log\|D\Hat{f}^{-n}_{\Hat{x}}\cdot v_{\Hat{x}} \|\ dp_x(\Hat{x})\ d\Leb_{\sigma}(x) > b \ \Leb_{\sigma}(A).
    \]
    Thus from the definition of $\Hat{\eta}_\sigma$, $\int_{\pi_{f}^{-1}(A)} \chi_{\Hat{f}^{-1}}(\Hat{x},v_{\Hat{x}}) \ d\Hat{\eta}_{\sigma} > b \ \Leb_{\sigma}(A) = b \ \Hat{\eta}_\sigma(\pi_{f}^{-1}(A))$, which implies the existence of the set $\Hat{G} \subset \pi_{f}^{-1}(A)$ as we want.
\end{proof}

We fix $E = E_p: \pi_{f}^{-1}(\sigma) \to \mathcal{P}_\nn$ the geometric set. We then apply Proposition \ref{prop key estimates} to obtain that
\[
\sum\limits_{n} \Hat{\eta}_{\sigma}(\{\Hat{x} \in \Hat{G}: d_n(E(\Hat{x}))< \beta \ \text{and} \ \|D\Hat{f}_{\Hat{x}}^{-n}\cdot v_x\| \ge e^{nb}\}) <+\infty.
\]
Hence, by taking a smaller subset $\Hat{A} \subset \Hat{G}$, still with $\Hat{\eta}_{\sigma}(\Hat{A})>0$, there exist $N \in \nn$ such that for every $n \ge N$ it holds that
\begin{equation*}
\text{for every} \ \Hat{x} \in \Hat{A}, \ \|D\Hat{f}^{-n}_{\Hat{x}} \cdot v_{\Hat{x}}\|>e^{nb} \Longrightarrow d_n(E(x)) \ge \beta .
\end{equation*}
This along with \eqref{eq large l.e. on G} imply that $\ol{d}(E(\Hat{x}))\ge \beta$ for every $\Hat{x} \in \Hat{A}$.

We then apply our results on the projective action $\Hat{F}: \mathbb{P}L_f \to \mathbb{P}L_f$ given by $\Hat{F} (\Hat{x},[v]) = (\Hat{f}(\Hat{x}), [D\Hat{f}_{\Hat{x}} \cdot v])$, where we consider the following objects:

\begin{itemize}
    \item The derivative cocycle $\Phi = (\phi_k)_k$ given by $\phi_k(\Hat{x},v) = \log \|D\Hat{f}_{\Hat{x}}^{-k} \cdot v\|$.
    \item The measure $\boldsymbol{\eta} = (\boldsymbol{\pi}_{\sigma})_* \Hat{\eta}_{\sigma}$ on $\mathbb{P}L_f$, where $\boldsymbol{\pi}_{\sigma}: \pi_{f}^{-1}(\sigma) \to \mathbb{P}L_f$ is given by $\boldsymbol{\pi}_{\sigma}(\Hat{x}) = (\Hat{x},v_x)$.
    \item The set $\boldsymbol{\Hat{A}} = \boldsymbol{\pi}_{\sigma}(\Hat{A}) \subset \mathbb{P}L_f$ which satisfies $\boldsymbol{\eta}(\boldsymbol{\Hat{A}}) = \Hat{\eta}_{\sigma}(\Hat{A}) > 0$.
    \item The geometric sequence $E:\Hat{A} \to \mathcal{P}_{\nn}$ which can be seen as $E:\boldsymbol{\pi}_{\sigma}(\Hat{A}) \to \mathcal{P}_{\nn}$ since $\boldsymbol{\pi}_{\sigma}$ is one to one. It satisfies $\ol{d}(E(\Hat{\xi})) \ge \beta$, for every $\Hat{\xi} \in \boldsymbol{\Hat{A}}$. 
\end{itemize}

We then apply Lemma \ref{lema folner sequence} to obtain a subsequence $\mathcal{n} \in \mathcal{N}$, a F{\o}lner sequence $\mathcal F = (F_{n})_{n \in \mathcal{n}}$ and $(\boldsymbol{\Hat{A}}_{n})_{n \in \mathcal{n}}$ subsets of $\boldsymbol{\Hat{A}}$ with $\boldsymbol{\eta}(\boldsymbol{\Hat{A}}_{n})$ decreasing sub-exponentially and with the property that $\partial F_{n} \subset E(\Hat{\xi})$ for every $\Hat{\xi} \in \boldsymbol{A}_{n}$ and every $n \in \mathcal{n}$. We denote by $\boldsymbol{\mu}_{n}$ the probability measure induced by $\boldsymbol{\eta}$ on $\boldsymbol{\Hat{A}}_{n}$, i.e. $\boldsymbol{\mu}_{n}(B) = \frac{\boldsymbol{\eta}(\boldsymbol{\Hat{A}}_{n} \cap B )}{\boldsymbol{\eta}(\boldsymbol{\Hat{A}}_{n})}$.

From Lemma \ref{lema limit of measures along folner sequences}, any $\boldsymbol{\mu}$ weak$^*$ limit of $\boldsymbol{\mu}_n^{F_{n}}:= \frac{1}{\#F_{n}} \sum\limits_{j \in F_{n}} \Hat{F}_*^{-j}\boldsymbol{\mu}_n$, is $\Hat{F}$-invariant, thus supported by the Oseledets bundle. Hence $\boldsymbol{\mu}$-a.e. $\Hat{\xi} = (\Hat{x},v) \in \mathbb{P}L_f$ satisfies 
\[
\Tilde{\phi}(\Hat{\xi}) = \lim\limits_{n\to +\infty} \frac{\phi_n}{n} = \chi_{\Hat{f}^{-1}}(\Hat{x},v) \in \{-\chi^-(\Hat{x}), -\chi^+(\Hat{x})\},
\]
where $\chi^-(\Hat{x}) \le \chi^+(\Hat{x})$ are the Lyapunov exponents of $\Hat{f}$ at $\Hat{x}$.
By Lemmas \ref{lema Ep is large} and \ref{lema prop of large sets}, $\Tilde{\phi}(\Hat{\xi}) > a >0$ for $\boldsymbol{\mu}$-almost every $\Hat{\xi} \in \mathbb{P}L_f$, we claim that $\Tilde{\phi}(\xi) = - \chi^-(\boldsymbol{\pi}(\Hat{\xi}))$ for $\boldsymbol{\mu}$-a.e. $\Hat{\xi}$.

In effect, if that was not true, $\Hat{\mu} = \boldsymbol{\pi}_*(\boldsymbol{\mu})$ would have an ergodic component with two negative Lyapunov exponents. Such measure is necessarily supported on a periodic sink $S$ for which there is an open neighborhood $U$ with $f(U) \subset U$ and $\sigma \cap U = \emptyset$. In particular, $f^{-n}(\sigma) \cap U = \emptyset$ for all $n \ge 0$, hence $\boldsymbol{\pi}_* \boldsymbol{\mu}_n^{F_{n}} (U) = 0$. Taking the limit in $n$ we get $\Hat{\mu}(U) = \Hat{\mu}(S) =0$. Therefore,
\begin{equation}\label{eq measure on unstable bundle}
    \Tilde{\phi}(\Hat{\xi}) = -\chi^-(\Hat{x})> a >0, \ \text{for}\ \boldsymbol{\mu}\text{-a.e.} \ \Hat{\xi} = (\Hat{x},v) \in \mathbb{P}L_f.
\end{equation}

Now, for each $q \in \nn$, we define the function $\psi^q: \mathbb{P}L_f \to \rr$ by
\[
\psi_q  = \phi - \frac{\mathcal{w}_q}{r-1},
\]
where $\phi(\Hat{x},v) = \log \|D\Hat{f}_{\Hat{x}}^{-1}\cdot v\}$ and 
\[
\mathcal{w}_q(\Hat{x},v) := \frac{1}{q}(\log \|D\Hat{f}_{\Hat{x}}^{-q}\| - \log \|D\Hat{f}_{\Hat{x}}^{-q}\cdot v\|),
\]
the number $\mathcal{w}_q^n(\Hat{x},v_x)$ is an exponential upper bound for the number of reparametrizations required to cover an $\Hat{F}^{-1}$ dynamical ball around $(\Hat{x},v_x)$ as in Lemma \ref{lema rep of dynamical balls}. Then, we have the following Gibbs property for $\boldsymbol{\mu}$.

\begin{prop}\label{prop folner gibbs prop with q}
    There exist an infinite sequence of positive numbers $(\delta_q)_q$ with $\delta_q \to 0$ as $q \to \infty$ such that for every $q \in \nn$, $\boldsymbol{\eta}$ satisfies the F{\o}lner-Gibbs property (Definition \ref{defn folner gibbs property}) with respect to the $\Hat{F}^{-1}$-cocycle associated to $\psi_q + \log d -\delta_q$ and the F{\o}lner sequence $\mathcal{F}$.
\end{prop}
\begin{proof}
    We start by noticing that, from \eqref{eq fubini on meas curves}, for every measurable set $B \subset \mathbb{P}L_f$:
    \[
    \boldsymbol{\eta}(B \cap \boldsymbol{\Hat{A}}_n) = \Hat{\eta}_{\sigma}(\boldsymbol{\pi}_{\sigma}^{-1}(B) \cap \Hat{A}_n) = \int_{\Sigma} \Leb_{\sigma}(\pi_{f}(\boldsymbol{\pi}_{\sigma}^{-1}(B) \cap \Hat{A}_n \cap \Hat{\sigma}_{\omega})) \ dp(\omega).
    \]
    Define for each $\omega \in \Sigma$, the probability measure $\boldsymbol{\eta}^\omega$ on $\boldsymbol{\pi}_\sigma (\Hat{\sigma}_\omega)$ by $\boldsymbol{\eta}^\omega(E) = \Leb_{\sigma}(\pi_{f}(\boldsymbol{\pi}_{\sigma}^{-1}(E) \cap \Hat{\sigma}_\omega))$, we have
    \begin{equation}
        \boldsymbol{\eta}(B \cap \boldsymbol{\Hat{A}}_n) = \int_{\Sigma} \boldsymbol{\eta}^{\omega}(B \cap \boldsymbol{\Hat{A}}_n) \ dp(\omega).
    \end{equation}

    Now, from Lemma \ref{lema rep of dynamical balls}, we may repeat Burguet's arguments in Subsection 5.4 \cite{burguet2024srb} to obtain that for $\delta>0$ arbitrarily small, there is $q \in \nn$ (depending only on $r$ and $\delta$), $\epsilon_q>0$ and $N \in \nn$ (depending on $r$, $\delta$, $f$ and $q$) such that for every partition $\mathcal{P}$ of $\mathbb{P}L_f$ with diameter less than $\epsilon_q$, every $\mathcal{n} \ni n \ge N$, every $\Hat{\xi} \in \boldsymbol{\Hat{A}}_n$ and every $\omega \in \Sigma$:
    \[
    \boldsymbol{\eta}^\omega(\mathcal{P}^{F_n}(\Hat{\xi}) \cap \boldsymbol{\Hat{A}}_n) \le \exp\left(-\sum\limits_{j \in F_n}\psi_q \circ \Hat{F}^{-j}(\Hat{\xi}) +\delta \#F_n \right).
    \]
    Note that this estimate does not depend on $\omega \in \Sigma$.
    
    Now, given $\Hat{\xi} \in \boldsymbol{\Hat{A}}_n$, let $\omega \in \Sigma$ be such that $\boldsymbol{\pi}(\Hat{\xi}) \in \Hat{\sigma}_{\omega}$. By reducing $\epsilon_q$ if necessary, we may assume that for every $n \in \mathcal{n}$, $\mathcal{P}^{F_n}(\Hat{\xi}) \cap \boldsymbol{\pi}_{\sigma}(\Hat{\sigma}_{\omega'}) = \emptyset$ for every $\omega'$ outside the cylinder $C^{N(n)}(\omega) = \{\omega' \in \Sigma: \omega'_i = \omega_i \ \forall \ 0 \le i \le N(n)\}$, where $N(n) = \max\{k \in F_n\}$. Thus, since $p(C^{N(n)}(\tau)) = d^{-N(n)} \le d^{-\#F_n}$, we obtain:
    \begin{align*}
    \boldsymbol{\eta}(\mathcal{P}^{F_n}(\Hat{\xi}) \cap \boldsymbol{\Hat{A}}_n) &= \int_{C^{N(n)}(\tau)} \boldsymbol{\eta}^\omega(\mathcal{P}^{F_n}(\Hat{\xi}) \cap \boldsymbol{\Hat{A}}_n) \ dp(\omega) \\ 
    &\le \exp\left(-\sum\limits_{j \in F_n}\psi_q \circ \Hat{F}^{-j}(\Hat{\xi}) +\#F_n\cdot \delta - \#F_n \cdot \log d\right).
    \end{align*}
\end{proof}

Finally, for any fixed $\boldsymbol{\mu}$ weak$^*$ limit of $\boldsymbol{\mu}_n^{F_n}$, from Propositions \ref{prop folner gibbs property} and \ref{prop folner gibbs prop with q}, we obtain that for every $q \in \nn$:
\begin{align*}
    h_{\boldsymbol{\mu}}(\Hat{F}) &\ge \int \psi_q \ d\boldsymbol{\mu} + \log d - \delta_q\\
    &\ge \int \phi d\boldsymbol{\mu} - \frac{1}{r-1}\left(\frac{1}{q} \int \log\|D\Hat{f}_{\Hat{x}}^q\| d\boldsymbol{\mu} + \frac{1}{q}\int \log\|D\Hat{f}_{\Hat{x}}^{-q}\cdot v\| \ d\boldsymbol{\mu}\right) + \log d - \delta_q
\end{align*}
We take $\Hat{\mu} = \boldsymbol{\pi}_*\boldsymbol{\mu}$ on $L_f$, from \eqref{eq measure on unstable bundle} and the subadditive ergodic theorem it satisfies
\[
\int \phi d \boldsymbol{\mu} = \lim_q \frac{1}{q}\int \log\|D\Hat{f}_{\Hat{x}}^{-q}\cdot v_x\| \ d\boldsymbol{\mu}= -\chi^{-}(\Hat{\mu}).
\]
moreover, by a straightforward application of the Furstenberg-Kersten theorem
\[
\lim\limits_{q\to +\infty} \frac{1}{q} \int \log \|D\Hat{f}_{\Hat{x}}^{-q}\| \ d\boldsymbol{\mu} = -\chi^-(\Hat{\mu}).
\]
Since $h_{\boldsymbol{\mu}}(\Hat{F}) = h_{\Hat{\mu}}(\Hat{f})$, we obtain by taking the limit on $q \to \infty$
\[
h_{\Hat{\mu}}(\Hat{f}) \ge -\chi^-(\Hat{\mu}) + \log d.
\]

Finally, by taking $\mu = (\pi_{f})_*\Hat{\mu}$ on $\TT^2$, since $h_{\Hat{\mu}}(\Hat{f}) = h_\mu(f)$, $\chi^-(\Hat{\mu}) = \chi^-(\mu)$ and $F_{\mu}(f) \le \log d$, we obtain
\[
h_\mu(f) \ge |\chi^-(\mu)| + \log d \ge |\chi^-(\mu)| + F_{\mu}(f) \ge h_\mu(f),
\]
where the last inequality is given from Ruelle's inequalities on non-invertible systems \eqref{eq Ruelle ineq endo}. We conclude $F_{\mu}(f) = \log d$ and  
\begin{equation}\label{eq srb}
h_\mu(f) = |\chi^-(\mu)| + \log d.
\end{equation}
Note that any ergodic component of $\mu$ satisfies \eqref{eq srb}.

\section{Proof of Theorem \ref{main thm uniqueness}}\label{sect:A}

In this section we will prove Theorem \ref{main thm uniqueness}. We divide the proof into two parts, first we prove the uniqueness of the equidistributed inverse SRB measures, for that we require the system to be close to a conservative one. Then we prove the full basin property under more general assumptions.

\subsection{Uniqueness}
In order to study the uniqueness of such measures, we will study the homoclinic classes of hyperbolic measures, we follow Lima-Poletti-Obata \cite{lima2024measuresmaximalentropynonuniformly}. Let $\mu_1,\mu_2$ be $f$-invariant ergodic hyperbolic probability measures on $\TT^2$ and consider $\Hat{\mu}_1, \Hat{\mu}_2$ their unique $\Hat{f}$-invariant lifts to $L_f$.

\begin{defn}(Homoclinic relation of measures)
We say that $\mu_1 \preceq \mu_2$ if there are measurable sets $A_1, A_2 \subset L_f$ with $\Hat{\mu}_i(A_i)>0$ such that $\mathcal{V}^u(\Hat{x}) \pitchfork \mathcal{V}^s(\Hat{y}) \neq \emptyset$ for every $(\Hat{x},\Hat{y}) \in A_1 \times A_2$. We say that $\mu_1$ and $\mu_2$ are homoclinically related, and write $\mu_1 \overset{h}{\sim} \mu_2$, if $\mu_1 \preceq \mu_2$ and $\mu_2 \preceq \mu_1$.
\end{defn}

One may verify that the homoclinic relation of measures is indeed an equivalence relation on the family of $f$-invariant hyperbolic ergodic probability measures [Proposition 4.3, \cite{lima2024measuresmaximalentropynonuniformly}].

\begin{thm}\label{thm uniqueness inverse SRB max fold on homoc class}
    Let $f \in End^r(\TT^2)$, $r>1$. In each homoclinic class of an $f$-invariant ergodic hyperbolic probability measure there exists at most one inverse SRB measure that maximizes the folding entropy. 
\end{thm}

\begin{proof}
    We consider $\text{NUH}\subset L_f$ the set of Oseledets regular points $\Hat{x} \in L_f$ admitting a splitting $T {L_{f}}_{\Hat{x}} = E^-_{\Hat{x}} \oplus E^+_{\Hat{x}}$ into the Oseledets subspaces associated with a negative and a positive Lyapunov exponent respectively, then for every hyperbolic $f$-invariant probability measure $\mu$, we have $\Hat{\mu}(\text{NUH})=1$. We consider also the measurable potential $\Hat{\varphi}:\text{NUH} \to \rr$ given by $\Hat{\varphi}(\Hat{x}) = -\log \|D\Hat{f}|_{E^-_{\Hat{x}}}\|+\log d$. From Ruelle's inequality for endomorphisms \cite{liu2003ruelle} and the relation $F_{\mu}(f) \le \log d$, every $f$-invariant hyperbolic measure $\mu$ satisfies
    \[
    h_{\Hat{\mu}}(\Hat{f}) \le \int \Hat{\varphi} \ d\Hat{\mu}.
    \]
    Hence, any equidistributed inverse SRB measure is an equilibrium state for the potential $\Hat{\varphi}$. We claim that in each homoclinic class there is at most one equilibrium state of $\Hat{\varphi}$.  
    
    Theorem 4.1 in \cite{lima2024measuresmaximalentropynonuniformly} applied in our context states that given $f \in End^r(\TT^2)$, for every ergodic and hyperbolic $f$-invariant probability measure $\mu$, given $\chi>0$, there is a locally compact countable topological Markov shift $(\Tilde{\Sigma},\Tilde{\sigma})$ and a Hölder continuous map $\Tilde{\pi}: \Tilde{\Sigma} \to L_f$ such that $\Tilde{\pi} \circ \Tilde{\sigma} = \Hat{f} \circ \Tilde{\pi}$, satisfying, among other things, that:
    \begin{enumerate}
        \item $\Tilde{\Sigma}$ is irreducible;
        \item Any $\chi$-hyperbolic $f$-invariant probability measure $\nu$ homoclinically related to $\mu$ satisfies that $\Hat{\nu}(\Tilde{\pi}(\Tilde{\Sigma})) = 1$ and that there exists a $\sigma$-invariant probability measure $\Tilde{\nu}$ on $\Tilde{\Sigma}$ such that $\Hat{\nu} = \Tilde{\pi}_*(\Tilde{\nu})$ and $h_{\tilde{\nu}}(\Tilde{\sigma}) = h_{\Hat{\nu}}(\Hat{f})$;
        \item If $\Tilde{\nu}$ is a given $\Tilde{\sigma}$-invariant probability measure on $\Tilde{\Sigma}$, then $\Hat{\nu}= \Tilde{\pi}_*(\Tilde{\nu})$ is an $\Hat{f}$-invariant hyperbolic probability measure homoclinically related to $\Hat{\mu}$ satisfying $h_{\Tilde{\nu}}(\Tilde{\sigma}) = h_{\Hat{\nu}}(\Hat{f})$;
        \item For every $\Hat{x} \in \Tilde{\pi}(\Tilde{\Sigma})$, there is a splitting $T{L_f}_{\Hat{x}} = E^-_{\Hat{x}} \oplus E^+_{\Hat{x}}$ such that:
        \begin{itemize}
            \item $\limsup\limits_{n\to +\infty}\frac{1}{n}\log \|D\Hat{f}^n_{\Hat{x}}|_{E^-_{\Hat{x}}}\|\le -\frac{\chi}{2}$;
            \item $\limsup\limits_{n\to +\infty}\frac{1}{n}\log \|D\Hat{f}^{-n}_{\Hat{x}}|_{E^+_{\Hat{x}}}\|\le -\frac{\chi}{2}$.
        \end{itemize}
        Moreover, the maps $p \in \Tilde{\Sigma} \mapsto E^{+/-}_{\Tilde{\pi}(p)}$ are Hölder continuous.
    \end{enumerate}

    From item 4 above, the potential $\Tilde{\varphi}:\Tilde{\Sigma}\to \rr$ given by $\Tilde{\varphi}(p) = \Hat{\varphi} \circ \Tilde{\pi}(p)$ is Hölder continuous. From item 1, the potential $\Tilde{\varphi}$ has at most one equilibrium state \cite{buzzi2003uniqueness}. From items 2 and 3, we conclude that $\Hat{\varphi}$ has at most one equilibrium state in each homoclinic class of a hyperbolic measure.
\end{proof}

We are now ready to prove uniqueness, we recall the spaces
\begin{equation*}
    \mathcal F_0=\{f\in\text{End}^{1+}(\TT^2): \inf_x I_f(x)>0, \ \text{and} \det(Df_x)=\deg(f) \}
\end{equation*}
and
\begin{equation*}
    \mathcal F_1=\{f\in\mathcal F_0:\ \pm 1 \text{ is not an eigenvalue of the linear part of } f\}
\end{equation*}

\begin{thm}\label{thm uniq proved}
    For every $f \in \mathcal{F}_1$, there exists $\mathcal{U} \subset \End^{1+}(\TT^2)$ a $\mathcal{C}^2$-neighborhood of $f$ such that every $g \in \mathcal{U}$ has at most one equidistributed inverse SRB measure.  
\end{thm}

\begin{proof}

We fix $f \in \mathcal{F}_1$ and a constant $\delta>0$. We apply Proposition \ref{prop large unstable manifolds} to obtain a constant $\ell_1>0$ and $\mathcal{U}_1$, a $\mathcal{C}^2$-neighborhood of $f$. We then choose $0<\ell_0<\ell_1$ and $0<\epsilon_0<\ell_0/4$ sufficiently small so that we can apply Proposition \ref{prop large stable manifolds} and Corollary \ref{corol dense preimages of curves} to obtain respectively $\mathcal{U}_2$, $\mathcal{U}_3$, $\mathcal{C}^1$-neighborhoods of $f$ and $N \in \nn$. Finally, we consider $\mathcal{U} = \mathcal{U}_1 \cap \mathcal{U}_2 \cap \mathcal{U}_2$, for every $g \in \mathcal U$ we have
    \begin{enumerate}
        \item For every $x \in \TT^2$
        \[
        p_x\{\Hat{x} \in \pi_{g}^{-1}(x): W^u_g(\Hat{x}) \ \text{is a $1$-Lipschitz graph with} \ |W^u_g(\Hat{x})|>\ell_0\}>1-\delta.
        \]
        
        \item For any equidistributed inverse SRB measure $\mu$ for $g$
        \[
        \mu\{x \in \TT^2: W^s_g(\Hat{x}) \ \text{is a $1$-Lipschitz graph with} \ |W^s_g(\hat{x})|>\ell_0\}>0
        \]
        
        \item For every $\mathcal{C}^2$ curve $\sigma:I \subset \rr \to \TT^2$ which is a $1$-Lipschitz graph with $|\sigma| \ge \ell_0$, the union $\bigcup_{n=0}^N g^{-n}(\sigma)$ is $\epsilon_0$-dense in $\TT^2$. 
    \end{enumerate}

We fix $g \in \mathcal{U}$ and let $\mu_1$, $\mu_2$ be two equidistributed $g$-invariant inverse SRB measures (the existence of such measures is guaranteed by Theorem \ref{main thm burguet}), we'll prove that $\mu_1$ and $\mu_2$ are homoclinically related, which from Theorem \ref{thm uniqueness inverse SRB max fold on homoc class} implies $\mu_1=\mu_2$. We show that $\mu_1\preceq \mu_2$, the converse is analogous. Consider the set  $A \subset \TT^2$ as 
    \[
    A = \{x \in \TT^2: W^s_g(x) \ \text{is a} \ L\text{-Lipschitz graph with} \ |W^s_g(x)|>\ell_0\},
    \]
from property (2) of $\mathcal{U}$, $\mu_1(A), \mu_2(A)>0$.
    
From property (3) of $\mathcal{U}$, given $x \in A$, the set $\bigcup_{j=0}^{N}g^{-j}(W^s_g(x))$ is $\epsilon_0$-dense. Hence, for any $z \in \TT^2$, there is $k \in \{0,1,\cdots,N\}$ and a point $\Tilde{x} \in f^{-k}(x)$ such that $W^s_g(\Tilde{x}) \cap B(z,\epsilon_0) \neq \emptyset$. We may then find sets $A_i \subset A$ with $\mu_i(A_i)>0$ such that
    \begin{itemize}
        \item $\text{diam}(A_2)<\epsilon_0$;
        \item for any point $x \in A_1$, $W^s_g(x) \cap B(A_2,\epsilon) \neq \emptyset$;
        \item $A_i$ is contained in a Pesin block of $\mu_i$ (where stable manifolds vary continuously);
        \item $A_i$ is contained in the support of $\mu_i$.
    \end{itemize}

    Given $x \in A_1$ and $y \in W^s_g(x) \cap B(A_2,\epsilon_0)$, Property (1) of $\mathcal{U}$ along with Lemma \eqref{lema unst direction distributed on fibres} gives variation of unstable directions on $\pi_{g}^{-1}(y)$ and good control of their sizes. Moreover, since $A_2\ni z\mapsto W^s(g)$ is continuous and since $\epsilon_0<\ell_0/4$, we may find a set $\Hat{B} \subset \pi_{g}^{-1}(y)$ with $p_{y}(\Hat{B})>0$ satisfying that $W^u_g(\Hat{y})\pitchfork W^s_g(z) \neq \emptyset$ for every $\Hat{y} \in \Hat{B}$ and every $z \in A_2$.
    
    Since for any $x \in A_1$, $\Leb^s(W^s_g(x)\cap B(A_2,\epsilon_0))>0$, $z \mapsto W^s_g(z)$ varies continuously in $A_2$ and both $\mu_1$ and $\mu_2$ disintegrate along stable manifolds with absolutely continuous measures, we obtain a set $\Tilde{A}_1 \subset  W^s_g(A_1) \cap B(A_2,\epsilon_0)$ with $\mu_1(A_1)>0$ satisfying that for every $y \in \Tilde{A}_1$
    \[
    p_{y}\{\Hat{y} \in \pi_{g}^{-1}(y): W^u_g(\Hat{y})\pitchfork W^s_g(z) \neq \emptyset \ \text{for every} \ z \in A_2\}>0
    \] 

    Finally, taking $\Hat{\mu}_i$ the $\Hat{g}$-invariant lift of $\mu_i$ to $L_g$, $i=1,2$. Since $\Hat{\mu}_1$ is equidistributed and the stable direction does not depend on the point $\Hat{z} \in \pi_{g}^{-1}(z)$, we may obtain a set $\Hat{A}_1 \subset \pi_{g}^{-1}(\Tilde{A}_1)$ with $\Hat{\mu}_1(\Hat{A}_1)>0$ satisfying that $W^u_g(\Hat{y})\pitchfork W^s_g(\Hat{z}) \neq \emptyset$ for every $\Hat{x} \in \Hat{A}_1$ and every $\Hat{y} \in \Hat{A}_2 = \pi_{g}^{-1}(A_2)$. Since $\Hat{\mu}_2(\Hat{A}_2) = \mu_2(A_2)>0$ we conclude $\mu_1 \preceq \mu_2$.
\end{proof}

\subsection{The full basin property.}

\begin{prop}
    Let $f \in \text{End}^r(\TT^2)$, $r>1$, with $C(f)>\frac{R^-(f)}{r}$. Consider $\{\mu_j\}_{j \in J}$ the set of $f$-invariant equidistributed inverse SRB measures. Then $\Leb$-almost every $x \in \TT^2$ satisfies that $p_x$-almost every $\Hat{x} \in \pi_{f}^{-1}(x)$ is in the basin $B^-(\mu_j)$ for some $j \in J$.
\end{prop}
\begin{proof}
   
    We prove that $\Hat{W}^u(\Hat{\Gamma}) \subset L_f$ has total $\Hat{\eta}$-measure, where $\Hat{\eta}$ is the equidistributed lift of the Lebesgue measure to $L_f$ \eqref{eq defn measure on Sol local leb times bernoulli} and $\Hat{\Gamma}\subset L_f$ is any $\Hat{f}$-invariant subset of $\bigcup_{j\in J}B^-(\mu_j)$ satisfying $\Hat{\mu}_j(\Hat{\Gamma}) = 1$ for all $j \in J$.

    Take $\Hat{\Gamma}$ as above and assume that there exists a set $\Hat{B}\subset L_f$ with $\Hat{\eta}(\Hat{B})>0$ and $\Hat{B} \cap \Hat{W}^u(\Gamma) = \emptyset$. Consider $B = \pi_{f}(\Hat{B})$, then $\Leb(B)>0$. Hence, from Lemma \ref{lema weakbundle}, we may find a $\mathcal{C}^r$ smooth curve $\sigma$ such that, for every point $x$ in a subset $A \subset \sigma\cap B$ with positive $\Leb_\sigma$-measure, we have
    \begin{equation}\label{eq l.e. curve}
    \chi_{\Hat{f}^{-1}}(\Hat{x},v_{\Hat{x}}) = \limsup\limits_{n \to +\infty}\frac{1}{n} \log \|D\Hat{f}_{\Hat{x}}^{-n}\cdot v_{\Hat{x}}\| \ge C(f)> \frac{R^-(f)}{r}, \ \ \ v_{\Hat{x}} = d\sigma_{\pi_{f}(\Hat{x})},
    \end{equation}
    for $p_x$-a.e. $\Hat{x} \in \pi_{f}^{-1}(A)$. In particular, we may take $\Hat{A} \subset \Hat{B} \cap \pi_{f}^{-1}(A)$ with $\Hat{\eta}_{\sigma}(A)>0$ such that \eqref{eq l.e. curve} holds for every $\Hat{x} \in \Hat{A}$. We may then follow the approach from Section \ref{sect:B}, one can define the geometric set $E$ on a subset $\Hat{A}' \subset \Hat{A}$ with $\Hat{\eta}_{\sigma}(\Hat{A}')>0$, and we let $\tau,\beta,\alpha$ and $\epsilon$ be the parameters associated with $E$ in a way that:
    \begin{itemize}
        \item E is $\tau$-large with respect to the cocycle $\Phi$;
        \item $\overline{d}(E(\Hat{x}))\ge \beta >0$ for every $\Hat{x} \in \Hat{A}'$;
        \item If $\Hat{\sigma}_\omega$ is the unique lift of $\sigma$ to $L_f$ containing $\Hat{x}$, we denote by $H_k(\Hat{x})\subset \Hat{\sigma}_\omega$ the maximal disk containing $\Hat{x}$ satisfying the requirements of the geometric time $k$ of $\Hat{x}$, see Definition \ref{defn geometric times}. We also denote by $D_k(\Hat{x}) = \Hat{f}^{-n}H_k(\Hat{x})$. Then $D_k(\Hat{x})$ has semi-length larger than $\alpha\epsilon$.
    \end{itemize}

    Let $\Hat{A}'' \subset \Hat{A}'$ be the subset of density points of $\Hat{A}'$ with respect to $\Leb_{\sigma}$, meaning that for $\Hat{x} \in \Hat{A}''$, if $\Hat{\sigma}_\omega$ is the unique lift of $\sigma$ to $L_f$ containing $\Hat{x}$ as made in Subsection \ref{subsec geom times}, and $\Hat{V}_\epsilon(\Hat{x}): = \Hat{\sigma}_\omega \cap B(\Hat{x},\epsilon)$, then:
    \[
    \lim\limits_{\epsilon \to 0}\frac{\Leb_{\Hat{\sigma}_\omega}(V_\epsilon(\Hat{x})\cap \Hat{A}')}{\Leb_{\Hat{\sigma}_\omega}(V_\epsilon(\Hat{x}))} =1
    \]
    In particular:
    \begin{equation}
        \frac{\Leb_{\Hat{\sigma}_\omega}(H_k(\Hat{x})\cap \Hat{A}')}{\Leb_{\Hat{\sigma}_\omega}(H_k(\Hat{x}))} \overset{k \to \infty}{\longrightarrow}1, \ \forall \ \Hat{x} \in \Hat{A}''.
    \end{equation}
    From \eqref{eq fubini on meas curves}, $\Hat{\eta}_\sigma(\Hat{A}'')>0$, and we may fix a subset $\Hat{C} \subset \Hat{A}''$ with $\Hat{\eta}_{\sigma}(C)>0$ such that the above convergence is uniform in $\Hat{C}$.

    From this set $\Hat{C}$ and the geometric set $E$ on $\Hat{C}$ we may proceed as in Section \ref{sect:B} to build a Fölner sequence $(F_n)_n$ and subsets $\Hat{C}_n \subset \Hat{C}$ such that, by considering $\Hat{\mu}_n = \frac{\Hat{\eta}_{\sigma}(\Hat{C}_n \cap \cdot)}{\Hat{\eta}_{\sigma}(\Hat{C}_n)}$, any weak$^*$-limit measure $\Hat{\mu}$ of $\{\Hat{\mu}_n^{F_n}\}$ is an equidistributed inverse SRB measure. In particular, $\Hat{\mu}(\Hat{\Gamma})=1$. We will show that $\Hat{C}\cap \Hat{W}^u(\Hat{\Gamma}) \neq \emptyset$ and since $\Hat{C} \subset \Hat{B}$, this contradicts the initial hypothesis that $\Hat{B} \cap \Hat{W}^u(\Hat{\Gamma})= \emptyset$. 

    We fix $\Hat{\mu}$ obtained as mentioned and consider $\Hat{\Lambda}$ be a Pesin block with $\Hat{\mu}(\Hat{\Lambda})>1-\beta/2$, see Theorem \ref{ThmExistencePesinBlocks}. We fix $\theta, \ell$ respectively as the minimal angle between $E^u$ and $E^s$, and as the minimal length of the local stable and unstable manifolds on $\Hat{\Lambda}$. From the absolute continuity of $\Hat{\mu}$ on stable manifolds, we have that, for a fixed $\gamma>0$
    \[
    \Hat{\mu}(\{\Hat{y} \in L_f: \Leb_{\Hat{W}^s_{\gamma} (\Hat{y})}(\Hat{\Gamma}\cap \Hat{\Lambda})=0\})\le 1-\Hat{\mu}(\Hat{\Lambda})<\beta/2.
    \]
    Hence, for some $c>0$:
    \[
    \Hat{\mu}(\{\Hat{y} \in L_f: \Leb_{\Hat{W}^s_{\gamma} (\Hat{y})}(\Hat{\Gamma}\cap \Hat{\Lambda})>c\})>1-\beta/2,
    \]
    and, taking $\Hat{G} = \{\Hat{y} \in L_f \in \Hat{\Gamma} \cap \Hat{\Lambda}:\Leb_{\Hat{W}^s_{\gamma} (\Hat{y})}(\Hat{\Gamma}\cap \Hat{\Lambda})>c\}$, we get $\Hat{\mu}(\Hat{G})>1-\beta$.

    To finish the proof, we assume the following claim, which we prove in the sequel. 
    
    \begin{claim}\label{claim unst holonomy}
        Given $\delta <<\min(\theta,\ell,\alpha\epsilon)$, there is $\gamma<\delta$ such that for infinitely many $n \in  \nn$ there are points $\Hat{x}_n \in \Hat{C}_n$ and $\Hat{y}_n \in \Hat{G}$ with $n \in E(\Hat{x}_n)$ such that the unstable holonomy $h^u: \Hat{\Lambda}\cap \Hat{\Gamma}\cap \Hat{W}^s_{\gamma}(\Hat{y}_n) \to D_n^\delta(\Hat{x}_n)$ is well defined, where $D_n^\delta(\Hat{x}_n):= D_n(\Hat{x}_n) \cap B(\Hat{f}^{-n}(\Hat{x}), \delta)$.
    \end{claim}

    The unstable holonomy is absolutely continuous, see subsection \ref{subsec abs cont}, and its Jacobian is bounded from below on $\Hat{\Lambda}$ by a constant depending only on the Pesin block. Since $\Leb_{\Hat{W}^s_\gamma(\Hat{y}_n)}(\Hat{\Lambda}\cap \Hat{\Gamma})>c$, we get for some constant $c'>0$ independent of $n$ that
    \begin{equation}\label{eq basin on disks}
        \Leb_{D_n(\Hat{x}_n)}(\Hat{W}^u(\Hat{\Lambda}\cap \Hat{\Gamma}))>c'.
    \end{equation}

    Since the distortion of $D\Hat{f}^{-n}$ on $H_n(\Hat{x}_n)$ is bounded by 3 \eqref{eq distortion bounded curves} and by the definition of $\Hat{C}$, we get
    \[
    \frac{\Leb_{D_n(\Hat{x}_n)}(D_n(\Hat{x}_n) \setminus \Hat{f}^{-n}(\Hat{A}))}{\Leb_{D_n(\Hat{x}_n)}(D_n(\Hat{x}_n))} \le 9\frac{\Leb_{H_n(\Hat{x}_n)}(H_n(\Hat{x}_n)\setminus\Hat{A})}{\Leb_{H_n(\Hat{x}_n)}(H_n(\Hat{x}_n))} \overset{n\to \infty}{\longrightarrow}0.
    \]
    And since $\Leb_{D_n(\Hat{x}_n)}(D_n(\Hat{x}_n)) \le \alpha\epsilon$, we get
    \begin{equation}\label{eq meas basin on disks to zero}
        \Leb_{D_n(\Hat{x}_n)}(D_n(\Hat{x}_n) \setminus \Hat{f}^{-n}(\Hat{A})) \overset{n\to \infty}{\longrightarrow}0.
    \end{equation}
    Hence, from \eqref{eq basin on disks} and \eqref{eq meas basin on disks to zero}, for $n$ large enough $\Hat{f}^{-n}(\Hat{A}) \cap \Hat{W}^u(\Hat{\Lambda} \cap \Hat{\Gamma})\neq \emptyset$, from the invariance of the unstable manifolds, we get $\Hat{A} \cap \Hat{W}^u(\Hat{\Lambda} \cap \Hat{\Gamma})\neq \emptyset$ and since $\Hat{A} \subset \Hat{B}$, this contradicts the initial hypothesis. 
\end{proof}

\begin{proof}[Proof of Claim \ref{claim unst holonomy}]
    For $\Hat{x} \in \pi_{f}^{-1}(\sigma)$ and $\Hat{y} \in \Hat{\Lambda}$, we denote by
    \[
    \boldsymbol{\Hat{x}}_{\sigma} = (\Hat{x},v_{\Hat{x}}),\  \boldsymbol{\Hat{y}}_s = (\Hat{y},v^s_{\Hat{y}}) \in \mathbb{P}L_f,
    \]
    where $v_{\Hat{x}}$ is the element representing the line tangent to $\sigma$ at $x =\pi_{f}(\Hat{x})$ and $v^s_{\Hat{y}}$ is the element representing the $E^s(\Hat{y})$ direction. We let $\boldsymbol{\Hat{\mu}}_n = (\boldsymbol{\pi}_{\sigma})_*\Hat{\mu}_n$, where $\boldsymbol{\pi}_{\sigma}(\Hat{x}) = \boldsymbol{\Hat{x}}_\sigma$ and $\boldsymbol{\Hat{\mu}}$ the weak$^*$-limit of $\boldsymbol{\mu}_n$ that projects to $\Hat{\mu}$.
    
    We consider
    \[
    \boldsymbol{\zeta}_n:=\int\frac{1}{\#F_n}\sum\limits_{k \in E(\Hat{x})\cap F_n}\delta_{\Hat{F}^{-k}(\boldsymbol{\Hat{x}}_\sigma)} \ d\boldsymbol{\mu}_n(\boldsymbol{\Hat{x}}_\sigma).
    \]
    By item 3 of Lemma \ref{lema folner sequence}, $\liminf_n\inf_{\Hat{x} \in \Hat{C}_n}\frac{\#F_n\cap E(\Hat{x})}{n}\ge \beta$, hence there is a weak$^*$-limit $\boldsymbol{\zeta} = \lim_n \boldsymbol{\zeta}_{n}$ with $\boldsymbol{\zeta} \le \boldsymbol{\Hat{\mu}}$ and $\boldsymbol{\zeta}(\mathbb{P}L_f)\ge \beta$. By considering $\boldsymbol{\Hat{G}}_s = \{\boldsymbol{y}_s= (\Hat{y},v_s): \Hat{y \in \Hat{G}}\}$, we get $\boldsymbol{\zeta}(\boldsymbol{\Hat{G}}_s^c)\le \boldsymbol{\Hat{\mu}}(\boldsymbol{\Hat{G}}_s^c)=\Hat{\mu}(\Hat{G})<\beta$, hence $\boldsymbol{\zeta}(\boldsymbol{\Hat{G}}_s)>0$.

    Now, since $L_f$ is separable, so is $\mathbb{P}L_f$ and we may consider $\{(\Hat{y}_j,v^s_{\Hat{y}})\}_j \subset \boldsymbol{\Hat{G}}_s$ a countable and dense subset of $\boldsymbol{\Hat{G}}_s$. Let $y_j = \pi_{f}(\Hat{y}_j)$, as in Subsection \ref{subsec abs cont}, for a small open neighborhood $V$ of $y_j$ (which the size do not depend on the point), there is a homeomorphism $\boldsymbol{\Xi}: \mathbb{P}V\times \Sigma \to \mathbb{P}\pi_{f}^{-1}(V)$. We define $\boldsymbol{\Xi}_j^\gamma:\mathbb{P}V \to \mathbb{P}L_f$ as the unique embedding $\boldsymbol{\Xi}_{j}(\boldsymbol{x}) = \boldsymbol{\Xi}(\boldsymbol{x},\omega)$ (with $\omega$ fixed) satisfying $\boldsymbol{\Xi}_{j}(\boldsymbol{y}_s^j) = \boldsymbol{\Hat{y}}_s^j$, where $\boldsymbol{y}_s^j = (y_j,v^s_{\Hat{y}_j}) \in \mathbb{P}V$. For $\gamma>0$ small, we denote by $\boldsymbol{V}_{j}^\gamma = \boldsymbol{\Xi}_{j}(B(\boldsymbol{y}_s^j,\gamma))$, and define  
    \[
    \boldsymbol{\Hat{G}}_s^\gamma = \bigcup\limits_{j\ge 0}\boldsymbol{V}_{j}^\gamma.
    \]
    Note that, for each $\gamma>0$ fixed, $\boldsymbol{\Hat{G}}_s\subseteq \boldsymbol{\Hat{G}}_s^\gamma$. We fix $\gamma <<\delta<<\min(\theta,\ell,\alpha\epsilon)$ such that $\boldsymbol{\Hat{G}}_s^\gamma$ is a continuity set of $\boldsymbol{\zeta}$ and conclude:
    \[
    0<\boldsymbol{\zeta}(\boldsymbol{\Hat{G}}_s)\le \boldsymbol{\zeta}(\boldsymbol{\Hat{G}}_s^\gamma) = \lim_n\boldsymbol{\zeta}_{n}(\boldsymbol{\Hat{G}_s^\gamma}).
    \]

    Note also that $\boldsymbol{\Hat{C}}_\sigma := \{\boldsymbol{\Hat{y}}_\sigma: \Hat{y} \in \Hat{C}\}$ has full $\boldsymbol{\Hat{\mu}}_n$-measure for all $n$. Hence, for infinitely many $n \in \nn$, there is $\boldsymbol{\Hat{x}}^n_\sigma = (\Hat{x}^n,v_{\Hat{x}^n}) \in \boldsymbol{\Hat{C}}_\sigma$ such that $\Hat{F}^{-n}(\boldsymbol{x}^n_\sigma) \in \boldsymbol{V}_j^\gamma$ for some $j$ and $n \in E(\Hat{x}^n)$. Set $\boldsymbol{\Hat{y}}^n_s = (\Hat{y}^n,v^s_{\Hat{y}^n}) \in \boldsymbol{\Hat{G}}_s$ satisfying that both $\Hat{F}^{-n}(\boldsymbol{\Hat{x}}^n_\sigma)$ and $\boldsymbol{\Hat{y}}^n_s$ are in the same $\boldsymbol{V}_j^\gamma$ for some $j$.

    From the choice of $\gamma$ and the fact that the points are in the same $\boldsymbol{V}_j^\gamma$, the unstable holonomy is well defined between the curves $D_n^\delta(\Hat{x}^n):=D_n(\Hat{x}_n)\cap B(\Hat{f}^{-n}(\Hat{x}^n),\delta)$ and $\Hat{W}^s(\Hat{y}^n)$ (see Subsection \ref{subsec abs cont}).
\end{proof}

\section{Proof of Theorem \ref{main thm SPR}}\label{sect:C}

In this section we prove continuity of the Lyapunov exponents under the hypothesis that the sequence of measures has their Folding entropy approximating its maximum (Prop. \ref{prop continuity LE}). This will yield the SPR property for the stable geometric potential. We start with some preliminary results.

\begin{lema}(Projectivized cocycles \cite{viana2014lectures})\label{lema projectivized cocycles}
Let $f:M \to M$ be a $\mathcal{C}^1$ map on a 2 dimensional manifold, with an invariant measure $\mu$ and different Lyapunov exponents $\lambda^-(x)<\lambda^+(x)$ at $\mu$-almost every $x \in M$, and let $E^-(x)$ and $E^+(x)$ be the corresponding Lyapunov subspaces. Consider $\Hat{F}: \mathbb{P}M \to \mathbb{P}M$ the 
projectivization of $f$, which is itself a bundle map over $f$. Then:

\begin{enumerate}
    \item If $\mu$ is ergodic, then there are exactly two ergodic lifts of $\mu$ to $\mathbb{P}M$, $\mu^{+}$ and $\mu^-$. The disintegrations of $\mu^{+}$ (resp. $\mu^{^-}$) along the fibers of $\mathbb{P}M$ are exactly the Dirac measures at the Lyapunov spaces $E^+$ (resp. $E^-$).
    \item If $\mu^{\mathbb{P}}$ is a lift of $\mu$ to $\mathbb{P}M$ ($\mu$ is not necessarily ergodic), then there is a measurable $f$-invariant function $\rho: M \to [0,1]$ such that the disintegrations of $\mu^{\mathbb{P}}$ along the fibers of $\mathbb{P}M$  are $\rho(x)\delta_{E^+(x)}+(1-\rho(x))\delta_{E^-(x)}$
\end{enumerate}
\end{lema}

\begin{prop}\label{prop prod disint limit meas}
    Let $X,Y,Z$ be compact metric spaces. Let $\mu^k$ be a sequence of measures on $W = X\times Y \times Z$ such that $\mu^k$ converges to $\mu$ in the weak$^*$-topology and, for $(\pi_X)_*\mu^k$-a.e. $x \in X$, the disintegration $\mu^k_x = \eta^k_x\times \nu^k_x$ is a product measure on $Y \times Z$.

    Let $\eta$ be a probability measure on $Y$ and suppose that for any continuous $\varphi: Y \to \rr$ and any $\epsilon>0$
    \[
    (\pi_X)_*\mu^k\left\{x \in  X: \left| \int_Y \varphi \  d\eta^k_x-\int_Y \varphi \ d\eta \right|  \ge \epsilon\right\} \to 0 , \ \text{as} \ k \to \infty. 
    \]
    Then, for $(\pi_X)_*\mu$-a.e. $x \in X$, the disintegrations $\mu_x$ are product measures on $Y \times Z$ and satisfy $\mu_x = \eta \times \nu_x$. 
\end{prop}
\begin{proof}
    From Lemma 4.3 in \cite{martin-a}, a measure $\mu$ in $W = X\times Y \times Z$ satisfies that the disintegrations $\mu_x$ are product measures on $Y \times Z$ for $(\pi_X)*\mu$-a.e. $x \in X$ if, and only if, for every continuous functions $h_1:X \to \rr$, $h_2:Y \to \rr$, $h_3:Z \to \rr$ we have
    \begin{equation}\label{eq prod disit}
    \int_W h_1(x) \cdot h_2(y) \cdot h_3(z) \ d\mu(x,y,z) = \int_W \left(\int_Y h_2(y) \ d(\pi_Y)_*\mu_x (y)\right) \cdot h_1(x) \cdot h_3(z) \ d\mu(x,y,z)
    \end{equation}

    We'll show \eqref{eq prod disit} for the limit measure $\mu$ and $(\pi_Y)_*\mu_x = \eta$. Since all the $\mu^k$ have product disintegrations $\mu^k_x = \eta^k_x \times \nu^k_x$, it is enough to show that 
    \[
    \int_W \left(\int_Y h_2 \ d\eta^k_x \right) \cdot h_1 \cdot h_3 \ d\mu^k \underset{k}{\longrightarrow} \int_W \left(\int_Y h_2 \ d\eta \right) \cdot h_1 \cdot h_3 \ d\mu 
    \]

    We fix the functions $h_1,h_2,h_3$ and $\epsilon>0$. Given $\delta>0$, we take $K \in \nn$ be such that for every $k \ge K$
    \[
    (\pi_X)_*\mu^k(\{x \in X: |\eta_x^k(h_2)-\eta(h_2)|\ge \epsilon\})<\delta.
    \]
    Hence, writing $h = \eta(h_2)\cdot h_1\cdot h_3$, we have
    \begin{align*}
        \left| \int_W \eta^k_x(h_2) \cdot h_1 \cdot h_3 \ d\mu^k - \int_W h \ d\mu   \right| \le \ & \left| \int h \ d\mu^k - \int h \ d\mu\right|\\
        &+ |h_1|_\infty\cdot |h_3|_\infty \int _X |\eta_x^k(h_2)-\eta(h_2)| \ d(\pi_X)_*\mu^k.
    \end{align*}
    The first term goes to zero as $\mu^k \to \mu$. The second term is controlled by
    \[
    \int _X |\eta_x^k(h_2)-\eta(h_2)| \ d(\pi_X)_*\mu^k \le \epsilon+2|h_2|_\infty\cdot \delta.
    \]
    Since $\epsilon,\delta>0$ are arbitrarily small, we get the desired result.
\end{proof}

\begin{prop}\label{prop continuity LE}
    Let $f_n, f\in \End^1(\TT^2)$ be such that $f_n$ converges to $f$ in the $\mathcal{C}^1$-topology and $C(f)>0$. Let also $\{\mu_n\}_n$ be a sequence of $f_n$-invariant ergodic measures converging weak$^*$ to some $\mu$ (necessarily $f$-invariant). If $F_{\mu_n}(f_n) \to \log \deg f$ and $\chi^-(\mu_n)<\chi^+(\mu_n)$ for every $n$, then $\chi^{\pm}(\mu_n) \to \chi^{\pm}(\mu)$.
\end{prop}

\begin{proof}
    We consider for each $n$, $E^-_n$ the Oseledets space of $f_n$ corresponding to the Lyapunov exponent $\chi^-(\mu_n)$. Consider then $\Hat{\mu}_n$, $\Hat{\mu}$ the lifts of $\mu_n$,$\mu$ to $L_{f_n}$ and $L_f$. Then consider $\Hat{\mu}_n^-$, $\Hat{\mu}^-$ the  $E^-_n, E^-$-lifts of $\mu_n$, $\mu$ to $\mathbb{P} L_{f_n}$, $\mathbb{P} L_f$ given by Lemma \ref{lema projectivized cocycles}. Notice that $\mu^-$ is well defined since from Corollary \ref{corol dist meas on fibers to p}, $\mu$ is equidistributed, and since $C(f)>0$, $\chi^-(\mu)<0<\chi^+(\mu)$ from Corollary \ref{corol hyperbolicity of measures max fold}.
    
    By passing to a subsequence if necessary, we may suppose that $\Hat{\mu}_n^-$ converge to a probability measure $\Hat{\mu}^{p}$ in $\mathbb{P}L_f$ which projects to $\Hat{\mu}$. $\chi^-(\mu_n) \to \chi^-(\mu)$ if and only if $\Hat{\mu}^{p} = \Hat{\mu}^-$, and $\chi^+(\mu_n) \to \chi^+(\mu)$ if and only if $\chi^-(\mu_n) \to \chi^-(\mu)$. Therefore, we will only prove that $\Hat{\mu}^{\mathbb{P}} = \Hat{\mu}^-$. 

    We fix any simply connected compact set $V \subset \TT^2$, then there are homeomorphisms $\Xi_n:V \times \Sigma \to \pi_{f_n}^{-1}(V)$ and $\Xi:V \times \Sigma \to \pi_f^{-1}(V)$, see Subsection \ref{subsec Solenoid}. Thus, although $\Hat{\mu}_n$, $\Hat{\mu}$ are each in different spaces, their restriction $\Hat{\mu_n}|_{\pi_{f_n}^{-1}(V)}$, $\Hat{\mu}|_{\pi_f^{-1}(V)}$ may all be identified with measure $\eta_n$, $\eta$ in $V \times \Sigma$. Analogously, the restrictions of $\Hat{\mu}_n^-$, $\Hat{\mu}^-$ and $\Hat{\mu}^p$ can be identified with measures $\eta_n^-, \eta^-, \eta^p$ in $V \times \Sigma \times \PP \rr^2$. We'll show that for any $V \subset \TT^2$ as such, the corresponding measures satisfy $\eta^p = \eta^-$, yielding that $\Hat{\mu}^p = \Hat{\mu}^-$ as we wanted.
    
    Note that since the direction $E^-_n$ does not depend on the past, it is constant on the fibers of $V\times \{\omega\}\times \PP \rr^2$, hence
    \begin{equation}
        \eta_n^- = \int_{V\times \Sigma} \delta_{E^-_n(x)}\ d\eta_n (\Hat{x}).
    \end{equation}
    We are now in the following scenario:
    \begin{itemize}
        \item The measures $\eta_n^-$ converge to $\eta^p$ in the space $V \times \Sigma\times \PP \rr^2$.
        \item Each $\eta_n^-$ disintegrates along the partition $\{\{x\}\times \Sigma\times \PP \rr^2: x \in V\}$ as the product $\eta_{x,n}\times \delta_{E^-_n(x)}$ on $\Sigma \times \mathbb{P}\rr^2$ that satisfies 
        \[
        (\Xi_n)_*(\delta_x\times \eta_{x,n}) = \Hat{\mu}_{x,n}
        \]
        where $\Hat{\mu}_{x,n}$ is the disintegration of the $\Hat{\mu}_n$ at $\pi_{f_n}^{-1}(x)$.
        \item From Corollary \ref{corol dist meas on fibers to p}, we have that for any continuous $\varphi: \Sigma \to \rr$ and any $\epsilon>0$,
        \[
        (\pi_V)_*\eta_n^-\left\{x \in V: \left|\int_{\Sigma}\varphi \ d\eta_{x,n}-\int_\Sigma \varphi \ d p\right|\ge \epsilon \right\} \longrightarrow 0, \ \text{as} \ n \to \infty,
        \]
        where $\pi_V$ is the projection into the first coordinate and $p$ is the equally distributed Bernoulli measure on $\Sigma$.
    \end{itemize}

    Such observation allows us to apply Lemma \ref{prop prod disint limit meas} with $X = V$, $Y = \Sigma$ and $Z = \mathbb{P}\rr^2$ to obtain that $\eta^p$ disintegrates on $(\pi_V)^{-1}(x)$ as the product $p \times w_{x}$ on $\Sigma \times \mathbb{P}\rr^2$. This translates to the limit measure $\Hat{\mu}^p$ on $\mathbb{P}L_f$ as
    \[
    \Hat{\mu}^p = \int_{\TT^2} p_x\times w_x \ d\mu(x) = \int_{L_f}w_{\pi_{f}(\Hat{x})} \ d\Hat{\mu}(\Hat{x}),
    \]
    where the second equality holds since from Prop. \ref{prop limit meas is equidist}, $\Hat{\mu}$ is equidistributed.
    
    Since $\Hat{\mu}^p$ is $\Hat{F}$-invariant, Lemma \ref{lema projectivized cocycles} gives an $\Hat{f}$-invariant function $\Hat{\rho}: L_f\to [0,1]$ such that
    \[
    \Hat{\mu}^p = \int_{L_f} \Hat{\rho}\delta_{E^+} + (1-\Hat{\rho}) \delta_{E^-} \ d\Hat{\mu}.
    \]
    And since the ergodic components of $\Hat{f}$ are exactly the pre-images by $\pi_{f}$ of the ergodic components of $f$, the function $\Hat{\rho}$ is constant on the fibers $\pi_{f}^{-1}(x)$, hence there is an $f$-invariant function $\rho:\TT^2 \to [0,1]$ such that $\Hat{\rho} = \rho \circ \pi_{f}$. Since $E^-$ is also constant on the fibers, we get that
    \begin{align*}
        \int_{L_f}(\rho \circ \pi_{f}) \delta_{E^+} \ d\Hat{\mu} &= \Hat{\mu}^p - \int_{L_f}(1-\rho\circ\pi_{f})\delta_{E^-} \ d\Hat{\mu}\\
        &= \int_{L_f}w_{\pi_{f}(\Hat{x})} \ d\Hat{\mu} + \int_{L_f} (1-\rho\circ\pi_{f})\delta_{E^-} d\Hat{\mu}\\
        &=\int_{L_f} w_{\pi_{f}(\Hat{x})} + (1-\rho\circ\pi_{f})\delta_{E^-} \ d\Hat{\mu}.
     \end{align*}

     Hence, for $\Hat{\mu}$-almost every $\Hat{x}$, we obtain that
     \[
     \rho(\pi_{f}(\Hat{x}))\delta_{E^+(\Hat{x})} = w_{\pi_{f}(\Hat{x})}+(1-\rho(\pi_{f}(\Hat{x})) \delta_{E^-(\pi_{f}(\Hat{x}))}.
     \]
     If $\Hat{\mu}^p$ is different from $\Hat{\mu}^-$, then $\rho$ is non-zero on a set $A \subset \TT^2$ with $\mu(A)>0$. In particular, for every $x \in A$, $\rho(x)\delta_{E^+(\Hat{x})} = w_x + (1-\rho(x))\delta_{E^-(x)}$ holds for $p_x$-almost every $\Hat{x} \in \pi_{f}^{-1}(x)$. The right-hand depends only on $x \in \TT^2$, which implies that the unit vector $v^+(\Hat{x}) = v^+(x)$ inside $E^+(\Hat{x})$ is constant on the fiber. We conclude that
     \begin{align*}
     \lim\limits_{n\to \infty} \frac{1}{n}I(x,v^+(x),f^n) &= \lim\limits_n \frac{1}{n} \int_{\pi_{f}^{-1}(x)} \log \|D\Hat{f}_{\Hat{x}}^{-n}\cdot v^{+}\| \ dp_x\\
     &= \int_{\pi_{f}^{-1}(x)} \lim\limits_n \frac{1}{n}\log \|D\Hat{f}_{\Hat{x}}^{-n}\cdot v^{+}\| \ dp_x\\
     &= \int_{\pi_{f}^{-1}(x)} -\chi^+ \ dp_x.
     \end{align*}

     Hence, for $\mu$-a.e. $x \in A$
     \[
     \chi^+(x) = -\lim\limits_n \frac{1}{n}I(x,v,f^n) \le -C(f),
     \]
     which contradicts Corollary \ref{corol hyperbolicity of measures max fold}.
\end{proof}

As a consequence, we get continuity of the Lyapunov exponents of the equidistributed inverse SRB measures.

\begin{corol}\label{corol cont LE}
    Consider the $\mathcal{C}^2$-open neighborhood $\mathcal{U}$ of $\mathcal{F}_1$ given by Theorem \ref{thm uniq proved} and, for each $f \in \mathcal{U}$, let $\mu_f$ be the unique equidistributed inverse SRB measure of $f$. Then, the maps $\mathcal U \ni f \mapsto \chi^{\pm}(\mu_f)$ are continuous (with the $\mathcal{C}^1$-topology in $\mathcal U$).
\end{corol}

We are now ready for the proof of Theorem \ref{main thm SPR}

\begin{proof}[Proof of Theorem \ref{main thm SPR}]
Let $f \in \End^1(\TT^2)$ be such that $C(f)>0$ and consider $\phi_s:L_f \to \rr$ the stable geometric potential defined as
\[
\phi_s(x) = \log |\det Df|_{E^s(x)}|, \ \text{and} \ \phi_s(x) = -\infty \ \text{if} \ E^s(x) \ \text{is not defined}.
\]

We'll show that $f$ satisfies the Pressure Hyperbolicity \ref{PH Prop.} and the Pressure Continuity \ref{PC Prop.} properties for $\phi_s$, the result will follow from Theorem \ref{thm ph and pc imply SPR endos}. Let $\mu_n$ be a sequence of $f$-invariant measures converging weak$^*$ to some $\mu$ and suppose that $P(f,\mu_n,\phi_s) \to P_{top}(f,\phi_s)$. We recall that from Ruelle's inequality for endomorphisms \eqref{eq Ruelle ineq endo}, $P(f,\mu,\phi_s) = h_\mu(f) +\int \phi_s d\mu = h_\mu(f)+\sum_{\chi_i<0}\chi_i(\mu) \le F_{\mu}(f)$ for any $f$-invariant measure $\mu$.

Since $f$ admits an equidistributed inverse SRB measure $\eta$, we have that
    \[
    P_{top}(f,\phi) = P(f,\eta,\phi)  = F_\eta(f) = \log d, \  d= \deg f.
    \]
In particular, we have $P(f,\mu_n,\phi_s) \to \log d$. Since $P(f,\mu_n,\phi_s) \le F_{\mu_n}(f) \le \log d$ for every $n$, we also have that $F_{\mu_n}(f) \to \log d$. From Corollary \ref{corol dist meas on fibers to p} we know that the limit measure $\mu$ is equidistributed. Since from Corollary \ref{corol hyperbolicity of measures max fold} there exists $\lambda>0$ such that any equidistributed measure is $\lambda$-hyperbolic we conclude the pressure hyperbolicity of $f$ for $\phi_s$.

Moreover, from the definition of $\phi_s$, we may assume without loss of generality that $\chi^-(\mu_n) \le 0$ (otherwise $P(f,\mu_n,\phi_s) = -\infty$), and that $\chi^+(\mu_n)>0$ (otherwise $\mu_n$ is supported on a periodic sink, thus has zero entropy and $P(f,\mu_n,\phi_s) \le 0$). This way, Prop. \ref{prop continuity LE} gives us that $\chi^{\pm}(\mu_n) \to \chi^{\pm}(\mu)$ yielding that $f$ is pressure continuous for $\phi_s$., which finishes the proof.
\end{proof}

\section{Proof of Theorem \ref{thm main}}\label{sect:main}

\begin{proof}
    Consider $\mathcal{U}_1$ a $\mathcal{C}^1$ open neighborhood of $\mathcal{F}_1$ such that every $f \in \mathcal{U}_1$ satisfies $C(f)>0$. Theorem \ref{main thm burguet} yields that every $\mathcal{C}^\infty$ endomorphism $f \in \mathcal{U}_1$ has an equidistributed inverse SRB measure. 

    Now consider $\mathcal{U}_2$ the $\mathcal{C}^2$-open neighborhood of $\mathcal{F}_1$ given by Theorem \ref{main thm uniqueness} and let $\mathcal{U} = \mathcal{U}_1 \cap \mathcal{U}_2$. Then every $\mathcal{C}^\infty$ endomorphism $f \in \mathcal{U}$ has a unique equidistributed inverse SRB measure which we denote by $\mu_f$ with full basin.

    From Theorem \ref{main thm SPR}, every $f \in \mathcal{U}$ is SPR for the geometric potential, and from Corollary \ref{corol cont LE} the Lyapunov exponents of $\mu_f$ vary continuously with $f \in \mathcal{U}$. Since $h_{\mu_f}(f) = \log d + |\chi^-(\mu_f)|$ for every $f \in \mathcal{U}$, $h_{\mu_f}(f)$ also varies continuously with $f \in \mathcal U$.

    All there is left to show is the continuity of the measures $\mu_f$, which we prove in Theorem \ref{thm cont of equid inv SRB meas} in the sequel.
\end{proof}

\subsection{Continuity of equidistributed inverse SRB measures}

In this section, we prove that equidistributed inverse SRB measures for inverse expanding on average endomorphisms, when unique, vary continuously in the space of endomorphisms.

\begin{thm}\label{thm cont of equid inv SRB meas}
    Consider $f \in \End^r(\TT^2)$, $r>1$, satisfying $C(f)>0$. Suppose that $f_n \in \End^r(\TT^2)$, is converging to $f$ in the $\mathcal{C}^1$-topology and $\|f_n\|_{r}$ is uniformly bounded. Then for any sequence $\{\mu_n\}_n$ of ergodic equidistributed inverse SRB measures for $f_n$, any accumulation point of $\{\mu_n\}_n$ is an equidistributed inverse SRB measure for $f$.
\end{thm}

Before going into the proof of the Theorem let us note its direct consequence.

\begin{corol}\label{corol cont equid inv srb meas}
    Consider $\mathcal{U}$ the $\mathcal{C}^2$-open neighborhood of $\mathcal{F}_1$ given by Theorem \ref{thm uniq proved} and for $f \in \mathcal{U}$, let $\mu_f$ be the unique equidistributed inverse SRB measure for $f$. The map $\mathcal{U} \ni f \mapsto \mu_f$ is continuous in the $\mathcal{C}^1$-topology in $\mathcal{U}$ restricted to $\mathcal{C}^2$-bounded sets and the weak$^*$ topology in the space of probabilities.
\end{corol}

In order to prove Theorem \ref{thm cont of equid inv SRB meas}, we rely on the techniques developed by C. Luo and D. Yang in \cite{luo2024ergodic}. We start with the following lemma, proved in the diffeomorphism case, which we indicate how to extend to endomorphisms.

\begin{lema}[Theorem 2.1, \cite{luo2024ergodic}]\label{lema bound tail entropy}
    Given $r>1$, there exist a constant $C_r$ satisfying the following. For every $\Upsilon>0$ and $q \in \nn$, there exists $\epsilon>0$ such that
    \begin{itemize}
        \item for every $f \in \End^r(\TT^2)$ with $\|f\|_r \le \Upsilon$,
        \item for every ergodic $\mu$ with $\chi^-(\mu) \le 0< \chi^+(\mu)$,
        \item for every finite partition $\mathcal{P}$ with $\diam(\mathcal{P})<\epsilon$ and $\mu(\partial\mathcal{P}) = 0$,
    \end{itemize}
    we have
    \[
    h_\mu(f) \le h_\mu(f,\mathcal{P}) + \frac{1}{r-1}\left(\frac{1}{q}\int \log \|Df_x^q\| \ d\mu - \chi^+(\mu) +\frac{1}{q} \right) + \frac{\log 2q\Upsilon  C_r }{q} 
    \]
\end{lema}
\begin{proof}
    In Theorem 2.1 \cite{luo2024ergodic}, the authors prove such affirmation considering a $\mathcal{C}^r$ diffeomorphism on a compact manifold and an ergodic measure with exactly one positive Lyapunov exponent. In order to rely on their result to obtain ours, we consider $g:N \to N$ the $\mathcal{C}^r$ realization of $f \in \End^r(\TT^2)$ given by Prop. \ref{lema nice realization}.
    
    Let $X \subset N$ be such that there is a homeomorphism $h: L_f \to X$ satisfying $g|_X\circ h = h \circ \Hat{f}$, and, for $\mu$ an $f$-ergodic as in the hypothesis, let $\nu = h_*\Hat{\mu}$, then $\nu$ is a $g$-ergodic measure with exactly one positive Lyapunov exponent $\chi^+(\nu) = \chi^+(\mu)$ and $h_\nu(g) = h_\mu(f)$.
    
    Then we may apply Theorem 2.1 \cite{luo2024ergodic} to $(g,\nu)$ to conclude that for any $q \in \nn$ there is $\epsilon>0$ such that for any partition $\Tilde{\mathcal{P}}$ of $X$ with $\diam \Tilde{\mathcal{P}}<\epsilon$,
    \[
    h_\nu(g) \le h_\nu(g,\Tilde{\mathcal{P}}) + \frac{1}{r-1}\left(\frac{1}{q}\int \log \|Dg_x^q\| \ d\nu - \chi^+(\nu) +\frac{1}{q} \right) + \frac{\log 2q\Upsilon  C_r }{q}.
    \]
    As $h: L_f \to X$ is uniformly continuous, we may fix $\delta>0$ such that $d(\Hat{x},\Hat{y})< \delta$ implies $d(h(\Hat{x}),h(\Hat{y}))<\delta$.
    
    Given a partition $\mathcal{P}$ of $\TT^2$ with $\diam{\mathcal{P}}<\delta$, let $\mathcal{P}':=\pi_{f}^{-1}(\mathcal{P})$. Since $\Hat{f}$ is a uniform contraction at fibers $\pi_{f}^{-1}(x)$, we have that for some $n \ge 0$ sufficiently large $\mathcal{P}' \vee \Hat{f}^n\mathcal{P}'$ is a partition with diameter smaller than $\delta$ and it satisfies $h_{\Hat{\mu}}(\Hat{f},\mathcal{P}'\vee \Hat{f}^n\mathcal{P}') = h_{\Hat{\mu}}(\Hat{f},\mathcal{P}') = h_{\mu}(f,\mathcal{P})$. Therefore $\Tilde{\mathcal{P}}:= h^{-1}(\mathcal{P}' \vee \Hat{f}^n\mathcal{P}')$ satisfies $\diam \tilde{\mathcal{P}} <\epsilon$ and $h_{\nu}(g,\Tilde{\mathcal{P}}) = h_\mu(f,\mathcal{P})$. We conclude
    \begin{align*}
        h_\mu(f) = h_\nu(g) &\le h_\nu(g,\Tilde{\mathcal{P}}) + \frac{1}{r-1}\left(\frac{1}{q}\int \log \|Dg_x^q\| \ d\nu - \chi^+(\nu) +\frac{1}{q} \right) + \frac{\log 2q\Upsilon  C_r }{q}\\
        &=h_\mu(f,\mathcal{P}) + \frac{1}{r-1}\left(\frac{1}{q}\int \log \|Df_x^q\| \ d\mu - \chi^+(\mu) +\frac{1}{q} \right) + \frac{\log 2q\Upsilon  C_r }{q}.
    \end{align*}
\end{proof}

The following is a version of Theorem A in \cite{luo2024ergodic}.
\begin{prop}\label{prop upper semicont entropy}
    Let $f \in \End^r(\TT^2)$, $r>1$, and suppose that $f_n \in \End^r(\TT^2)$ are converging to $f$ in the $\mathcal{C}^1$-topology and $\|f_n\|_r$ is uniformly bounded. Consider also $\{\mu_n\}_n$ a sequence of $f_n$-ergodic measures converging to some $\mu$ (necessarily $f$-invariant). Suppose that
    \begin{itemize}
        \item $\chi^-(\mu_n) <0<\chi^+(\mu_n)$ for all $n$, and
        \item $\chi^+(\mu_n)$ is converging to $\chi^+(\mu)$.
    \end{itemize}
    Then $\limsup_{n} h_{\mu_n}(f_n) \le h_\mu(f)$.
\end{prop}

\begin{proof}
    Let $\Upsilon$ be such that $\|f_n\|_r \le \Upsilon$ for all $n$. From Lemma \ref{lema bound tail entropy} applied to each $(f_n,\mu_n)$, we get that for any $q \in \nn$, there is $\epsilon>0$ such that for any partition $\mathcal{P}$ with $\diam\mathcal{P}<\epsilon$ and $\mu(\partial \mathcal{P}) = \mu_n(\partial \mathcal{P} )=0$ for all $n$, we have
    \begin{align*}
        \limsup\limits_n h_{\mu_n}(f_n) &\le \limsup\limits_n h_{\mu_n}(f_n, \mathcal{P}) +\frac{1}{r-1}\left(\frac{1}{q}\int \log \|(Df_n^q)_x\| \ d\mu_n - \chi^+(\mu_n) +\frac{1}{q} \right)\\
        & \ \ \ +\frac{\log 2q\Upsilon  C_r }{q}\\
        &\le  h_\mu(f,\mathcal{P})+ \frac{1}{r-1}\left(\frac{1}{q}\int \log \|Df_x^q\| \ d\mu - \chi^+(\mu) +\frac{1}{q} \right) + \frac{\log 2q\Upsilon  C_r }{q}.
    \end{align*}
    The second inequality holds for three reasons: first, $\limsup_n h_{\mu_n}(f_n,\mathcal{P}) \le  h_{\mu}(f,\mathcal{P})$, see e.g. [Lemmas 3.3 and 3.4, \cite{gan2021statistical}]. Second, $f_n \to f$ in the $\mathcal{C}^1$ topology implies that $Df_n^q \to Df^q$ uniformly on $\TT^2$, and since $\mu_n \to \mu$ as well we conclude that the integrals converge. Finally, the convergence of $\chi^+(\mu_n)$ holds from the hypothesis.

    Since the inequality holds for any $q \in \nn$, we conclude that for every $\delta>0$, there exist a partition $\mathcal{P}$ such that
    \[
    \limsup_n h_{\mu_n}(f_n) \le h_\mu(f,\mathcal{P})+ \delta \le h_\mu(f) +\delta,
    \]
    which finishes the proof.
\end{proof}

Finally, we prove Theorem \ref{thm cont of equid inv SRB meas}.

\begin{proof}[Proof of Theorem \ref{thm cont of equid inv SRB meas}]
    Let $f_n \to f \in \End^r(\TT^2)$ be a converging sequence in the $\mathcal{C}^1$ topology with $\|f_n\|_r$ uniformly bounded and let $\mu_n$ be ergodic equidistributed inverse SRB measures for $f_n$. We may suppose without loss of generality that $\mu_n$ converges to some probability $\mu$ which from Corollary \ref{corol dist meas on fibers to p} is necessarily equidistributed for $f$. 

    Since $f_n \to f$ and $C(f)>0$, we may assume as well that $C(f_n)>0$ for all $n$, which from Corollary \ref{corol hyperbolicity of measures max fold} implies that $\chi^-(\mu_n)<0<\chi^+(\mu_n)$ for all $n$. Moreover, since $\mu_n$ are all equidistributed, $F_{\mu_n}(f_n) = \log d= F_\mu(f)$ for every $n$. Thus, Prop. \ref{prop continuity LE} yields that $\chi^\pm(\mu_n) \to \chi^\pm(\mu)$ and the previous Prop. \ref{prop upper semicont entropy} yields $\limsup_n h_{\mu_n}(f_n) \le h_\mu(f)$.

    Since every $\mu_n$ is equidistributed inverse SRB for $f_n$ and from Rulle's inequality for endomorphisms \eqref{eq Ruelle ineq endo}, we get
    \[
    \log d +|\chi^-(\mu)| = \limsup_n h_{\mu_n}(f_n) \le h_\mu(f) \le \log d+|\chi^-(\mu)|.
    \]
    Hence $\mu$ is an equidistributed inverse SRB measure for $f$.
\end{proof}

\section{Proof of Proposition \ref{cor}}\label{sect:Cor}

\begin{proof}[Proof of Proposition \ref{cor}]

We consider a linear Anosov endomorphism on $f:\TT^2\to \TT^2$ induced by $E \in M_2(\zz)$ with non-trivial stable bundle, for instance $E = \begin{pmatrix}
    3 &1\\
    1 &1
\end{pmatrix}$.

Since this is an open property, we may consider $\mathcal{U}$ a $\mathcal{C}^1$ neighborhood of $f$ such that every $g \in \mathcal{U}$ is an Anosov systems with non-trivial stable direction. Mihailescu \cite{mihailescu2010physical} proved that every $g \in \mathcal{U}$ preserves an equidistributed inverse SRB measure. 

Since $f$ is conservative, we may consider $g \in \mathcal{U}$ which is also conservative and does not have constant Jacobian ($\det Dg \nequiv \text{const.}$). We have that the Lebesgue measure is an inverse SRB measure \cite{liu2008invariant} and since $g$ has non-constant Jacobian, it satisfies that the folding entropy $F_{\Leb}(g) < \log d$, thus it is not equidistributed.
\end{proof}

We couldn't find on the literature any construction of a non-invertible conservative systems with non-constant Jacobian. Hence we expose below how to obtain the perturbations $g$ as we need for this proof.

\subsection{Perturbations in dimension one}
We begin in dimension 1, as this construction will be utilized for dimension 2. The only self-covers of $\TT^1$ with constant Jacobian are the maps $R_k: \TT^1 \to \TT^1$ given by $R_k(x) = kx(mod1)$, $k \in \zz_*$. For $\epsilon>0$ we construct a $\mathcal{C}^{\infty}$ conservative self-cover $S_\epsilon: \TT^1 \to \TT^1$ which is $\mathcal{C}^1$-$\epsilon$-close to $R_k$ with non-constant Jacobian.

For that, we fix any point $p \in \TT^1$, $\delta< \frac{1}{2k}$ and consider, for $0<\epsilon<\delta$ small, a $\mathcal{C}^\infty$ bump function $\varphi_{\epsilon}:[p-\delta,p+\delta] \to \rr$ satisfying:
\begin{itemize}
    \item $0 \le \varphi(x) \le \epsilon$, for every $x$; $\varphi(p) >0$,
    \item $\varphi(x) = 0$, if $|x-p| \ge \delta /2$, and
    \item $|\varphi'(x)|<\epsilon$ for every $x$.
\end{itemize}
We denote by $I_0 = [p-\delta,p+\delta]$, $J = R_k(I_0)$. Since $\delta$ is sufficiently small, we get that $R_k^{-1}(J)$ is the disjoint union of $k$ intervals $I_0, \cdots I_{k-1}$ with $I_j = I_0+\frac{j}{k}$. 

We consider the diffeomorphism $T: I_0 \to I_1$ given by $T(x) = x + \frac{1}{k}+\frac{2}{k}\varphi_{\epsilon}(x)$ and define the perturbation $S_\epsilon:\TT^1 \to \TT^1$ as:
\[
S_\epsilon(x) = \left\{\begin{array}{ll}
    R_k(x), & x \notin I_o \cup I_1,\\
    R_k(x) + \varphi_{\epsilon}(x), & x \in I_0, \\
    R_k(x) -\varphi_\epsilon \circ T^{-1}(x), & x \in I_1.
\end{array}  \right.
\]
It is clear from the definition of $\varphi_\epsilon$ that $S_\epsilon$ is a $\mathcal{C}^\infty$ self-cover of $\TT^1$ which is $\mathcal{C}^1$-$\epsilon$-close to $R_k$. It has non constant Jacobian, more specifically $|S_\epsilon '(x)|\neq k$ in the points $x \in I_0 \cup I_1$ where the derivative of $\varphi_\epsilon$ and $\varphi_\epsilon\circ T^{-1}$ are non-zero.  All there is left to show is that $S_\epsilon$ is conservative, for that we show that for every $y \in \TT^1$
\begin{equation}\label{eq condition conservative}
\sum\limits_{z \in S_\epsilon^{-1}(y)} \frac{1}{|S_\epsilon'(z)|} = 1.
\end{equation}

For points $y \notin J = R_k(I_0) = S_\epsilon(I_0)$, we have $S_{\epsilon}^{-1}(y) \cap (I_0 \cup I_1) = R_k^{-1}(y)\cap (I_0\cup I_1) = \emptyset$. Hence $S_{\epsilon}'(z) = R_k'(z) = k$ for every $z \in S_\epsilon^{-1}(y)$ and we get \eqref{eq condition conservative} trivially. Now, for $y \in J$, we have $S_{\epsilon}^{-1}(y) = \{z_0,\cdots,z_{k-1}\}$ with $z_i \in I_i$. For $i \notin \{0,1\}$ we have $|S_\epsilon'(z_i)| = k$. We also have that $z_1 = T(z_0)$, indeed
\begin{align*}
S_\epsilon(T(z_0)) &= R_k(T(z_0)) - \varphi_\epsilon\circ T^{-1}(T(z_0)) = kz_0 +\varphi_\epsilon(z_0) +1 = S_\epsilon(z_0).
\end{align*}

We then have
\[
S_\epsilon'(z_1) = k - \frac{\varphi_\epsilon'(z_0)}{1+\frac{2}{k}\varphi_\epsilon'(z_0)}.
\]
Hence, for $z \in J$:
\begin{align*}
\sum\limits_{z \in S_\epsilon^{-1}(y)}\frac{1}{|S_\epsilon'(z)|} = \frac{k-2}{k}+\frac{1}{|S_\epsilon'(z_0)|}+\frac{1}{|S_\epsilon'(z_1)|} 
&= \frac{k-2}{k} + \frac{1}{k+\varphi_\epsilon'(z_0)}+ \frac{1+\frac{2}{k}\varphi_\epsilon'(z_0)}{k+\varphi_\epsilon'(z_0)}  =1. 
\end{align*}
This shows that \eqref{eq condition conservative} holds for every $z \in \TT^1$, hence $S_\epsilon$ is conservative. 

\subsection{Perturbation for linear systems in dimension 2}

We consider $E \in M_2(\zz)$ with $d = |\det E|\ge 2$ and we write $E = U A V$, where $U, V \in SL_2(\zz)$, and 
\[
A = \begin{pmatrix}
\tau_1 & 0\\
0 & \tau_2
\end{pmatrix} \in GL_2(\zz)
\]
is the Smith normal form of $E$ where $\tau_2 = \gcd\{(e_{ij})\}$ is the greater common divisor of the entries of $E$, and $\tau_1 \cdot \tau_2 = d = |\det E|$, in particular $\tau_1>1$. We denote by the same letter $E: \TT^2 \to \TT^2$ the induced endomorphism.

Considering $\TT^2 = \TT^1 \times \TT^1$, we get $A(x,y) = (R_{\tau_1}(x),R_{\tau_2}(y))$. We then consider a perturbation $A_\epsilon: \TT^2 \to \TT^2$ given by $A_{\epsilon}(x,y) = (S_\epsilon(x), R_{\tau_2}(y))$, where $S_\epsilon$ is the perturbation constructed in the previous section. We define $g_\epsilon:\TT^2 \to \TT^2$ by
\[
g_\epsilon = U \circ A_\epsilon\circ V
\]
Again, it is clear that $g_\epsilon$ is $\mathcal{C}^1$-$\epsilon$-close to $E$ and $|\det Dg_\epsilon(p)|\neq d$ for every $p \in V^{-1}((I_0\cup I_1) \times \rr)$. All there is left to demonstrate is that $g_\epsilon$ is conservative and we show it by obtaining that for every $q \in \TT^2$
\[
\sum\limits_{z \in g_\epsilon^{-1}(q)} \frac{1}{|\det Dg_\epsilon(z)|} = 1.
\]
Now, given $q \in \TT^2$, we write $\tilde{q} = V ^{-1}(q)$, then $U (g_\epsilon^{-1}(q)) = A_\epsilon^{-1}(\tilde{q})$ and since $ |\det U| = |\det V| = 1$, we get
\[
\sum\limits_{z \in g_\epsilon^{-1}(q)} \frac{1}{|\det Dg_\epsilon(z)|} = \sum\limits_{\tilde{z} \in A_\epsilon^{-1}(\tilde{q})}\frac{1}{|\det D A_\epsilon(\Tilde{z})|}.
\]
Hence we only need to show that the right side of the equality above is $1$ for every $\Tilde{q} \in \TT^2$.

As in the previous section, we have that if $q \notin J \times \rr$, then for every $z \in A_\epsilon^{-1}(q)$, $A_\epsilon(z) = A(z)$, hence $|\det D A_\epsilon(z)| = d = |\det A|$ and we get the result trivially.

Now for $q = (x,y) \in J\times \rr$, we get $A_\epsilon^{-1}(q) = \{z_{ij} = (x_i,y_j): i=0,\cdots ,\tau_1-1; \ j = 0,\cdots \tau_2-1\}$ with $x_i \in I_i$. For every $i \ge 2$ we have $|\det DA_\epsilon(z_{ij})| =|\det A| = d$ and there are exactly $(\tau_1-2)\tau_2$ of such points.

For $i = 0,1$, if we write $z_{ij} =(x_i,y_j)$, we have that $(x_1,y_j) = (T(x_0),y_j)$ for every $j=0,\cdots \tau_2-1$ and it satisfies
\[
|\det DA_\epsilon(z_{1j})|=|\tau_2||S_\epsilon'(x_1)| = \tau_2\left(\tau_1 -\frac{\varphi_\epsilon'(z_0)}{1+\frac{2}{\tau_1}\varphi_\epsilon'(z_0)} \right)=\frac{\tau_1\tau_2(\tau_1+\varphi_\epsilon'(z_0))}{\tau_1+2\varphi_\epsilon'(z_0)}.
\]
Hence, for every $j$ fixed:
\[
\sum\limits_{i=0}^1\frac{1}{|\det DA_\epsilon(z_{ij})|} = \frac{1}{\tau_2(\tau_1+\varphi_\epsilon'(z_0))}+\frac{1+\frac{2}{\tau_1}\varphi_\epsilon'(z_0)}{\tau_2(\tau_1+\varphi_\epsilon'(z_0))}= \frac{2}{\tau_1\tau_2}.
\]
Since $j = 0,\cdots,\tau_2-1$, we conclude:
\[
\sum\limits_{z \in A_\epsilon^{-1}(q)} \frac{1}{|\det DA_\epsilon(z)|} = (\tau_1-2)\tau_2\frac{1}{d}+\tau_2\frac{2}{d}=1.
\]

\begin{rmk}
    This construction can be generalized to perturb any conservative map $f: \TT^n \to \TT^n$ with constant Jacobian with the hypothesis that the homotopy between $f$ and its linear part is a volume preserving diffeomorphism. For instance, one could perturb the examples in \cite{martin-a} and apply our Theorem \ref{main thm burguet} to obtain such examples of topologically mixing maps with at least two inverse SRB measures inside the open set given by Theorem \ref{thm main}.
\end{rmk}

\bibliographystyle{amsalpha}
\bibliography{bibliography}

\end{document}